\documentclass[10pt]{article}

\usepackage{mathtools, amsthm, accents, tikz, amssymb, latexsym, amsmath, amscd, amsfonts, array, lmodern, enumerate, stmaryrd, rotating, caption, graphicx, hyperref, float,tikz-cd, color, yfonts}
\usepackage[all]{xy}
\CompileMatrices
\hypersetup{nesting=true,debug=true,naturalnames=true}
\def\red{}

\def\l{\hspace{.5mm}\vert\hspace{.5mm}}
\def\ll{\hspace{.3mm}\Big\vert\hspace{.3mm}}
\def\card{\protect\operatorname{card}}
\def\eor{\protect\operatorname{eor}}
\def\image{\protect\operatorname{Im}(\pi^*)}

\def\sor{\protect\operatorname{sor}}

\def\sign{\protect\operatorname{sign}}

\def\F{\protect\operatorname{Conf}}
\def\UF{\protect\operatorname{UConf}}

\def\DnT{\protect\operatorname{D}_n\!T}

\def\aop{\langle k_1,x_1,p_{1},q_{1}\rangle\cdots\langle k_m,x_m,p_{m},q_{m}\rangle}

\def\UDnT{\protect\operatorname{UD}_n\!T}

\newcommand*{\dt}[1]{%
  \accentset{\mbox{$_\centerdot$}}{#1}}

\newtheorem{proposition}{Proposition}[section]
\newtheorem{corollary}[proposition]{Corollary}
\newtheorem{definition}[proposition]{Definition}
\newtheorem{theorem}[proposition]{Theorem}
\newtheorem{remark}[proposition]{Remark}
\newtheorem{example}[proposition]{Example}
\newtheorem{lemma}[proposition]{Lemma}
\newtheorem{examples}[proposition]{Examples}

\begin{document}

\title{Cohomology ring of \red{tree braid groups and exterior face rings}}

\author{Jes\'us Gonz\'alez and Teresa Hoekstra-Mendoza}

\date{\empty}

\maketitle

\begin{abstract}
For a tree $T$ and a positive integer $n$, let $B_nT$ denote the $n$-strand braid group on $T$. We use discrete Morse theory techniques to show that the cohomology ring $H^*(B_nT)$ is encoded by an explicit abstract simplicial complex $K_nT$ that measures $n$-local interactions among essential vertices of $T$. We show that, in many cases (for instance when $T$ is a binary tree), $H^*(B_nT)$ is the exterior face ring determined by $K_nT$.
\end{abstract}



\section{Main results}\label{mainresult}
For a finite graph $\Gamma$ and a positive integer $n$, let $\F_n\Gamma$ denote the configuration space of $n$ ordered points on $\Gamma$,
$$
\F_n\Gamma\coloneqq\left\{(x_1,\ldots,x_n)\in \Gamma^n\colon x_i\neq x_j \mbox{ for } i\neq j \right\}.
$$
The usual right action of the $n$-symmetric group $\Sigma_n$ on $\F_n\Gamma$ is given by $(x_1,\ldots,x_n)\cdot \sigma=(x_{\sigma(1)},\ldots,x_{\sigma(n)})$, and $\UF_n\Gamma$ stands for the corresponding orbit space, the configuration space of $n$ unlabelled points on $\Gamma$. Both $\F_n\Gamma$ and $\UF_n\Gamma$ are known to be aspherical (\cite{MR2701024,MR1873106}); their corresponding fundamental groups are denoted by $P_n\Gamma$ (the pure $n$-braid group on $\Gamma$) and $B_n\Gamma$ (the full $n$-braid group or, simply, the $n$-braid group on $\Gamma$). We focus on the case of a tree $\Gamma=T$.

Besides its central role in geometric group theory, graph braid groups have applications in areas outside pure mathematics such as robotics, topological quantum computing and data science. Yet, there is a relatively limited knowledge of the algebraic topology \red{properties} of \red{a} graph braid group (or, for that matter, of a tree braid group), particularly concerning its cohomology ring structure.

Using discrete Morse theory techniques on Abrams' cubical model $\UDnT$ for $\UF_nT$ (reviewed below), D.~Farley gave in \cite{MR2216709} an efficient description of the \emph{additive} structure of the cohomology of $B_nT$. Later, and in order to get at the \emph{multiplicative} structure, the Morse theoretic methods were replaced in \cite{MR2359035} by the use of a Salvetti complex $\mathcal{S}$ obtained by identifying opposite faces \red{of cells} in $\UDnT$. Being a union of tori, $\mathcal{S}$ has a well understood cohomology ring. Yet more importantly, the projection map $q\colon\UDnT\to\mathcal{S}$ induces a surjection in cohomology. Farley's main result in \cite{MR2359035} is a description of a set of generators for Ker$(q^*)$, which yields a presentation for the cohomology ring of $B_nT$.

Although \cite{MR2359035} includes an algorithm for performing computations mod Ker$(q^*)$, the price of not working at the Morse theoretic level is that Farley's presentation includes many non-essential generators. As a result, calculations are hard to work with, both in concrete examples, as well as in theoretical developments (cf.~Remark~\ref{comparewithSchierer} below). In particular, Farley-Sabalka's conjecture (\cite[Conjecture~5.7]{MR2355034}) that $H^*(B_nT;\mathbb{Z}_2)$ is an exterior face ring, suggested on the basis of extensive concrete calculations, was left open. 

\red{In this paper} we combine Farley-Sabalka's original Morse theoretic approach with Forman's Morse-theoretic description of cup products to prove the integral version of Farley-Sabalka's conjecture for a large family of trees. The statement in Theorem~\ref{maintheorem} below, which focuses on binary trees, i.e., on trees all whose essential vertices have degree three, disproves Conjecture 5.17 in~\cite{MR2516176} by exhibiting an infinite family of non-linear trees $T$ all whose braid group cohomology rings are exterior face rings.

\begin{theorem}\label{maintheorem}
Assume $T$ is a binary tree. For a commutative ring $R$ with 1, the cohomology ring $H^*(B_nT;R)$ is the exterior face ring $\Lambda_R(K_nT)$ determined by a simplicial complex $K_nT$. Explicitly, $H^*(B_nT;R)$ is the quotient $\Lambda/I$, where $\Lambda$ is the exterior graded $R$-algebra generated by the vertex set of $K_nT$, and $I$ is the ideal generated by monomials corresponding to non-faces of $K_nT$.
\end{theorem}

As noted in \cite[p.~68]{MR2355034}, the isomorphism type of a complex $K_nT$ as the one in Theorem~\ref{maintheorem} is well determined. We refer to $K_nT$ as the \emph{$n$-interaction complex of $T$}. A description of $K_nT$ as an abstract simplicial complex is given in Definition~\ref{interactioncomplex} below. The explicit definition allows us, for instance, to easily deduce \red{a concrete} right-angled Artin group presentation for $B_nT$ when $T$ is a linear \red{binary} tree (Example \ref{linearcase} below). This complements the inductive method in \cite{MR3253777} proving that linearity is a sufficient\footnote{The condition is known to be necessary and sufficient.} condition for a tree to have right-angled Artin braid groups.

The definition of $K_nT$ applies for any tree and we show that the resulting combinatorial object encodes much of the ring structure of $H^*(B_nT;R)$, whether $T$ is binary of not. Indeed, we generalize Theorem~\ref{maintheorem} in two directions. On the one hand, the ring-isomorphism assertion $H^*(B_nT;R)\cong\Lambda_R(K_nT)$ holds as long as $T$ is a tree with binary core (Theorem~\ref{binarycore} below). Furthermore, we show that, for any tree $T$, the vertices of $K_nT$ can be thought of as giving an $R$-basis of $H^1(B_nT;R)$, while the cup-product-based rule $\{v_1,\ldots,v_m\}\mapsto v_1\cdots v_m$ sets a 1-1 correspondence between the family of ($m-1$)-simplices of $K_nT$ and an $R$-basis of $H^m(B_nT;R)$. More importantly, while cup squares are known to vanish in $H^*(B_nT;R)$, certain (square-free) products $v_1\cdots v_m$ are non-zero even when $\{v_1,\ldots,v_m\}$ fails to be a face of $K_nT$ (this can happen only if $T$ is not a tree with binary core). In any such case, we give a closed formula (Theorem~\ref{productsvialowerpaths}) to write any such product $v_1\cdots v_m$ as an $R$-linear combination of basis elements, thus completing a full description of the cup-product structure in the cohomology of $B_nT$ for any tree $T$. Details are summarized in Theorem~\ref{casogeneral} below.


The techniques used in this work (discrete Morse theoretic approach to cup products) should be a valuable tool \red{in understanding the algebraic topology properties of} discrete models for other spaces, such as non-particle configuration spaces, \red{as well as} generalized (e.g., no-$k$-equal) configuration spaces.

\begin{remark}{\em
Ghrist's pioneering work led to conjecture that any pure braid group $P_n\Gamma$ on a graph $\Gamma$ would be a right-angled Artin group. In the case of full braid groups $B_n\Gamma$, \cite{MR2833585,MR3426912} give two characterizations (one combinatorial and another cohomological) of the right-angled-Artin condition. For instance, for $\Gamma=T$ a tree, $B_nT$ is a right-angled Artin group if and only if $H^*(B_nT)$ is the exterior face ring of a \emph{flag} complex. Theorem~\ref{maintheorem} and its generalized version in Theorem~\ref{binarycore} assert that, in the full braid group setting and for trees with binary core, Ghrist's conjecture is true after removal of the flag requirement.
}\end{remark}

The description of the complex $K_nT$, as well as an explicit statement of Theorem~\ref{casogeneral}, and a couple of explicit illustrations (Examples~\ref{3aspectos} and~\ref{linearcase}) of Theorem~\ref{maintheorem} require a few preparatory constructions. Unless otherwise noted, throughout the rest of the section $T$ stands for an arbitrary tree.

Fix once and for all a planar embedding together with a root (a vertex of degree~1) for $T$. Order the vertices of $T$ as they are first encountered through the walk along the tree that (a) starts at the root vertex, which is assigned the ordinal 0, and that (b) takes the left-most branch at each intersection given by an essential vertex (turning around when reaching a vertex of degree 1). Vertices of $T$ will be denoted by the assigned non-negative integer. An edge of $T$, say with endpoints $r$ and $s$, will be denoted by the ordered pair $(r,s)$, where $r<s$. Furthermore, the ordering of vertices will be transferred to \red{an ordering of} edges by declaring that the ordinal of $(r,s)$ is $s$. The resulting ordering of vertices and edges will be referred \red{to} as the $T$-order\footnote{This of course depends on the embedding and root chosen.}.

\begin{figure}[h!]
\centering
\begin{tikzpicture}[  scale=.75, rotate = 180, xscale = -1]
\tikzset{EdgeStyle/.style={->,font=\scriptsize,above,sloped,midway}}
\node[style={circle, draw, scale=.6}] (1) at ( 2, 3.45) {$0$};
\node[style={circle, draw, scale=.6}] (2) at ( 7.4, 3.45) {$x$};
\node[style={circle, draw, scale=.6}] (3) at ( 9.58, 1.28) {};
\node[style={circle, draw, scale=.6}] (4) at ( 10.5, 2.72) {};
\node (5) at ( 10.25, 3.8) {};
\node[style={circle, draw, scale=.6}] (6) at ( 10.6, 4) {};
\node (7) at ( 8.25, 1.91) {};
\node (8) at ( 8.97, 2.6) {};
\node (9) at ( 9.1, 3.34) {};
\node (10) at ( 8.94, 4.04) {};
\node (11)[scale=.3] at (4, 3.45){};
\node (12)[scale=.3] at (6, 3.45){};
\node at (3.1, 3.24){\tiny $\;\;0$-direction 1};
\draw (4)[dotted] to[bend right] (6); 
\draw (1) -- (11);
\draw (11)[dashed] -- (12);
\draw (12) node[above]{\tiny \hspace{5mm}$x$-direction 0 \ }-- (2);
\draw (2)--node[sloped,above,rotate=90]{\tiny $x$-direction 1}(3);
\draw (4) --node[sloped,above,rotate=30]{\tiny $x$-direction 2} (2);
\draw (6) --node[sloped, below, rotate=-20]{\tiny $x$-direction $d(x){-}1$} (2);
\end{tikzpicture}
\caption{The \red{$d(x)$} $x$-directions from an essential vertex $x$}
\label{x-directions}
\end{figure}
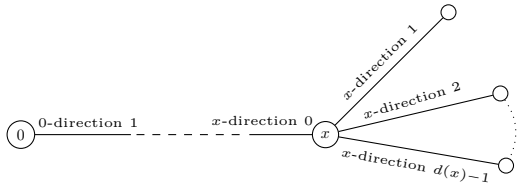

\red{Let $d(x)$ stand for the degree of a vertex $x$ of $T$, so} there are $d(x)$ ``directions'' \red{from} $x$. \red{For a vertex~$x$ different from the root, the direction from $x$} that leads to the root is defined to be the $x$-direction 0; $x$-directions $1, 2, \dots , d(x)-1$ \red{(if any)} are then chosen following the positive orientation coming from the planar embedding. See Figure \ref{x-directions}. For instance, \red{if $x$ is not the root and} the vertex \red{$y$} incident to $x$ in $x$-direction 0 is \red{not essential (i.e.~$d(y)\leq2$), then $y=x-1$. Likewise, if $d(x)\geq2$, then} $x+1$ is the vertex incident to $x$ in $x$-direction 1. \red{It will be convenient to think of the only direction from the root vertex $0$ as 0-direction 1, in particular there is no $0$-direction 0.}

Fix essential vertices $x_1<\cdots<x_m$ of T. The complement in $T$ of the set $\{x_1,\ldots,x_m\}$ decomposes into $1+ \sum_{i=1}^m \red{\left(d(x_i)-1\right)}$ components $C_{i,\red{\ell_i}}=C_{i,\red{\ell_i}}(x_1,\ldots,x_m)$, where $0\leq i\leq m$, \red{$\ell_0=1$,} and $1\leq\red{\ell_i}\leq d(x_i)-1$ for $i>0$. The closure of each $C_{i,\red{\ell_i}}$ is a subtree of $T$. $C_{0,1}$ is the component containing the root $0$, while $C_{i,\red{\ell_i}}$ (for $i>0$) is the component whose closure \red{contains} $x_i$ and is located on the $x_i$-direction $\red{\ell_i}$. The set $B(C_{i,\red{\ell_i}})$ of ``bounding'' vertices of a component $C_{i,\red{\ell_i}}$ is defined to be the intersection of the closure of $C_{i,\red{\ell_i}}$ with $\{x_1,\ldots,x_m\}$. \red{Note that $x_i\in B(C_{i,\ell_i})$ for $i>0$, however the root $0$ is not considered to be a bounding vertex of $C_{0,1}$, just as no leave of $T$ (i.e., a vertex of degree 1 other than the root) is considered to be a bounding vertex of any $C_{i,\ell_i}$.}

\begin{definition}[The $n$-interaction complex of $T$, $K_nT$]\label{interactioncomplex}
\begin{itemize}
\item[(a)] The vertex set $V_nT$ of $K_nT$ is the collection of all 4-tuples $\nu=\langle k,x,p,q\rangle$, where $k$ is a non-negative integer number, $x$ is an essential vertex of $T$, and $p=(p_{1},\dots, \red{p_r})$ and $q=(q_1, \dots, \red{q_s})$ are tuples of non-negative integer numbers satisfying the three conditions
\begin{itemize}
\item[$\bullet$]\red{$r+s=d(x)-1$, with $r>0<s;$}
\item[$\bullet$] \red{$k+|p|+|q|=n-1$, where $|p|:=\sum_{j=1}^r p_j$ and $|q|:=\sum_{j=1}^s q_j;$}
\item[$\bullet$] $p_j>0$ \red{for at least one $j\in\{1,\ldots,r\}$.}
\end{itemize}
\red{We stress that $r$ (i.e., the length of $p$) is one of the parameters determining the 4-tuple $\nu$. For instance, if $d(x)=6$ and $n=4$, then $\langle1,x,(0,1,0),(1,0) \rangle$ and $\langle1,x,(0,1),(0,1,0) \rangle$ are two different elements in $V_nT$. The length $s$ of $q$, on the other hand, is determined by $r$ and $d(x)$.}
\item[(b)] \red{For $\nu_1,\ldots,\nu_m\in V_nT$ with $\nu_i=\langle k_i,x_i,p_i,q_i\rangle$, $p_i=(p_{i,1},\dots, p_{i,r_i})$, $q_i=(q_{i,1}, \dots, q_{i,s_i})$ and so that $\hspace{.1mm}x_1<\cdots<x_m$, consider the components $C_{0,1}$ and $C_{i,\ell_i}$ $(1\leq i\leq m$ and $1\leq \ell_i\leq d(x_i)-1)$ of $T\setminus\{x_1, \ldots, x_m\}$ as defined above. Then, for $C\in\{C_{0,1},C_{i,\ell_i}\}$, the $C$-local information of $\nu_j$, denoted by $\ell_C(\nu_j)$, is defined by}
$$\red{\ell_{C_{0,1}}(\nu_j)=\begin{cases}
 k_j, & \mbox{if \,$x_j\in B(C_{0,1})$;} \\ 0, & \mbox{otherwise,}
\end{cases}
}$$
\red{and, for $i>0$,} 
\begin{equation}\label{localinformation}
\ell_{C_{i,\red{\ell_i}}}(\nu_j)=\begin{cases}
p_{i,\red{\ell_i}}, & \mbox{if \,$j=i$ and $\red{\ell_i} \leq \red{r_i};$} \\
q_{i,\red{\ell_i-r_i}}, & \mbox{if \,$j=i$ and $\red{\ell_i} > r_i;$} \\
\red{k_j,} & \mbox{if \,$j\neq i$ and $\red{x_j\in B(C_{i,\ell_i})};$} \\ 0, & \mbox{in any other case.}
\end{cases}
\end{equation}
Note that $\ell_C(\nu_j)=0$ whenever $x_j\not\in B(C)$.
\item[(c)] \red{The $n$-interaction complex of $T$ is the abstract simplicial complex $K_nT$ whose vertex set is $V_nT$ and whose $(m-1)$-simplices are given by families of vertices $\nu_1,\ldots,\nu_m$ as in item (b) satisfying}
\begin{equation}\label{localinteraction}
\sum_{j=1}^m\ell_{C_{i,\red{\ell_i}}}\!\left(\nu_j \right)\geq n\left(\rule{0mm}{4mm}\card(B(C_{i,\red{\ell_i}}))-1\right),
\end{equation}
for all $i\in\{0,1,\ldots,m\}$ and all relevant $\ell_i$, and in such a way that, for every $i>0$,~\emph{(\ref{localinteraction})} is a strict inequality for at least one $\ell_i\in\{1,\dots , r_i\}$.
\end{itemize}
\end{definition}
It is \red{an easy} arithmetic exercise (whose verification is left to the reader) to check that $K_nT$ is indeed a simplicial complex.

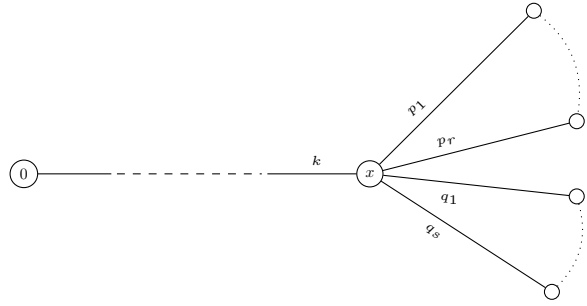
\begin{figure}
\centering
\begin{tikzpicture}[  scale=1.0, rotate = 180, xscale = -1]
\tikzset{EdgeStyle/.style={->,font=\scriptsize,above,sloped,midway}}
\node[style={circle, draw, scale=.6}] (1) at ( 2.8, 3.45) {$0$};
\node[style={circle, draw, scale=.6}] (2) at ( 7.4, 3.45) {$x$};
\node[style={circle, draw, scale=.6}] (3) at ( 9.58, 1.28) {};
\node[style={circle, draw, scale=.6}] (4) at ( 10.15, 2.75) {};
\node[style={circle, draw, scale=.6}] (0) at ( 10.15, 3.75) {};
\node (5) at ( 10.25, 3.8) {};
\node[style={circle, draw, scale=.6}] (6) at ( 9.82, 5.02) {};
\node (7) at ( 8.25, 1.91) {};
\node (8) at ( 8.97, 2.6) {};
\node (9) at ( 9.1, 3.34) {};
\node (10) at ( 8.94, 4.04) {};
\node (11)[scale=.3] at (4, 3.45){};
\node (12)[scale=.3] at (6, 3.45){};
\draw (0)--node[below]{\tiny $q_1\hspace{.7cm}$} (2);
\draw (0)[dotted] to[bend right] (6); 
\draw (3)[dotted] to[bend right] (4); 
\draw (1) -- (11);
\draw (11)[dashed] -- (12);
\draw (12) node[above]{\tiny  \hspace{1.3cm} $k$}-- (2);
\draw (2)--node[sloped,above,rotate=90]{\tiny  $p_1\hspace{1cm}$}(3);
\draw (4) --node[sloped,above,rotate=30]{\tiny  $p_{\red{r}}\hspace{.7cm}$} (2);
\draw (6) --node[sloped, below, rotate=-65]{\tiny  $q_{\red{s}}\hspace{.7cm}$} (2);
\end{tikzpicture}
\caption{The local information given by a vertex $\langle k,x,\red{(p_1,\ldots,p_r), (q_1,\ldots,q_s)} \rangle$ of $K_nT$}
\label{localinformationinapicture}
\end{figure}

Definition \ref{interactioncomplex} is dictated by discrete Morse theoretic considerations ---reviewed in latter sections. Our choice for using angle brackets instead of parenthesis for 4-tuples \red{in $V_nT$} will be justified later in the paper (Remark \ref{angles}). More important at this point is to explain the role of $K_nT$ as an object measuring ``local interactions'' between systems of ``local informations'' around essential vertices of $T$. For starters, we refer to a vertex $\nu=\langle k,x,\red{(p_1,\ldots,p_r),(q_1,\ldots,q_s)} \rangle\in V_nT$ as a \red{system of} local informations around the essential vertex $x$ \red{of $T$}. Indeed, \red{as illustrated in Figure \ref{localinformationinapicture}, we think of:}

\vspace{1.7mm}\ \hspace{2.3mm} \emph{(i)} \hspace{.3mm} 
$k$ \red{as} the local information of $\nu$ in $x$-direction~0,

\vspace{1.7mm}\ \hspace{2.3mm}\emph{(ii)} \hspace{.15mm} 
\red{$p_j$ ($1\leq j\leq r$) as} the local information of $\nu$ in $x$-direction \red{$j$}, and

\vspace{1.7mm}\ \hspace{1mm}\emph{(iii)} \hspace{.1mm}
\red{$q_j$ ($1\leq j\leq s$) as} the local information of $\nu$ in $x$-direction $\red{j+r}$. 

\vspace{2mm}\noindent 
\red{In these terms,} (\ref{localinformation}) gives a systematic way to spell out the information ingredients on a given \red{family of systems of} local informations. \red{Likewise,} \red{item \emph{(c)} in Definition~\ref{interactioncomplex} asserts that a family $\{\nu_1,\ldots,\nu_m\}$ of systems of} local informations around essential vertices \red{$x_1<\cdots<x_m$} of $T$ assemble a simplex of $K_nT$ \red{if, for each component $C$ of $T\setminus\{x_1,\ldots,x_m\}$,} the sum of the \red{$C$-local} informations of vertices~\red{$x_j$} bounding \red{$C$} is suitably large, depending on $n$ and on the number of bounding vertices of \red{$C$.} 

\begin{definition}\label{strongvsweak}
Let $\nu_1,\nu_2,\ldots,\nu_m\in V_nT$ be a family of systems of local informations around essential vertices \red{$x_1<x_2<\cdots<x_m$} of $T$. We say that $\nu_1,\ldots,\nu_m$ interact strongly provided $\{\nu_1,\ldots,\nu_m\}$ is a simplex of $K_nT$. We say that $\nu_1,\ldots,\nu_m$ interact weakly provided~(\ref{localinteraction}) holds for all relevant $i$ and $\ell_i$ but $\{\nu_1,\ldots,\nu_m\}$ fails to be a simplex of $K_nT$ ---so that, in fact,~(\ref{localinteraction}) is an equality for some $i>0$ and all $\ell_i\in\{1,\ldots,r_i\}$. In all other cases, we say that $\nu_1,\ldots,\nu_m$ do not interact.
\end{definition}

\begin{figure}[h!]
$$
\begin{tikzpicture}[x=.51cm,y=.51cm]
\draw(0,0)--(0,2);\draw(0,0)--(1.4,-1.4);\draw(0,0)--(-1.4,-1.4);
\draw(0,2)--(1.3,3.3);\draw(0,2)--(-1.3,3.3);
\draw(1.4,-1.4)--(3,-1.4);\draw(1.4,-1.4)--(1.4,-2.8);
\draw(-1.4,-1.4)--(-3,-1.4);\draw(-1.4,-1.4)--(-1.4,-2.8);
\node [below] at (-1.4,-2.8) {\scriptsize \red{root}};
\node [above] at (1.6,-1.45) {\scriptsize{$x_4$}};
\node [below] at (.05,-.1) {\scriptsize{$x_2$}};
\node [above] at (-1.6,-1.45) {\scriptsize{$x_1$}};
\node [above] at (0,2.2) {\scriptsize{$x_3$}};

\draw(-10,0)--(-10,2);\draw(-10,0)--(-8.6,-1.4);\draw(-10,0)--(-11.4,-1.4);
\draw(-10,2)--(-8.7,3.3);\draw(-10,2)--(-11.3,3.3);
\draw(-8.6,-1.4)--(-7,-1.4);\draw(-8.6,-1.4)--(-8.6,-2.8);
\draw(-11.4,-1.4)--(-13,-1.4);\draw(-11.4,-1.4)--(-11.4,-2.8);
\node [below] at (-8,-1.3) {\scriptsize{$\red{1}$}};
\node [below] at (-8.8,-1.55) {\scriptsize{$\red{0}$}};
\node [below] at (-8.8,-.43) {\scriptsize{$\red{2}$}};
\node [below] at (-10.85,-.8) {\scriptsize{$\red{2}$}};
\node [below] at (-11.6,-1.6) {\scriptsize{$\red{0}$}};
\node [below] at (-11.85,-.72) {\scriptsize{$\red{1}$}};
\node [below] at (-10.2,2) {\scriptsize{$\red{2}$}};
\node [below] at (-10.25,3.05) {\scriptsize{$\red{1}$}};
\node [below] at (-9.4,2.65) {\scriptsize{$\red{0}$}};
\node [below] at (8.6,-2.8) {\scriptsize \red{root}};

\draw(10,0)--(10,2);\draw(10,0)--(11.4,-1.4);\draw(10,0)--(8.6,-1.4);
\draw(10,2)--(11.3,3.3);\draw(10,2)--(8.7,3.3);
\draw(11.4,-1.4)--(13,-1.4);\draw(11.4,-1.4)--(11.4,-2.8);
\draw(8.6,-1.4)--(7,-1.4);\draw(8.6,-1.4)--(8.6,-2.8);
\node [below] at (12,-1.2) {\scriptsize{$\red{1}$}};
\node [below] at (11.2,-1.55) {\scriptsize{$\red{0}$}};
\node [below] at (11.2,-.43) {\scriptsize{$\red{7}$}};
\node [below] at (9.1,-.85) {\scriptsize{$\red{7}$}};
\node [below] at (8.4,-1.55) {\scriptsize{$\red{0}$}};
\node [below] at (8.15,-.72) {\scriptsize{$\red{1}$}};
\node [below] at (9.8,2) {\scriptsize{$\red{6}$}};
\node [below] at (9.75,3.1) {\scriptsize{$1$}};
\node [below] at (10.6,2.65) {\scriptsize{$\red{1}$}};
\node [below] at (10.25,-.15) {\scriptsize{$2$}};
\node [left] at (9.8,-.15) {\scriptsize{$2$}};
\node [below] at (10.25,.9) {\scriptsize{$4$}};
\node [below] at (-11.37,-2.8) {\scriptsize \red{root}};
\node at (0,3.5) {\rule{4mm}{0mm}};
\end{tikzpicture}
$$
\caption{Three different aspects of the mininal non-linear tree $T_0$}
\label{ejemplitos}
\end{figure}
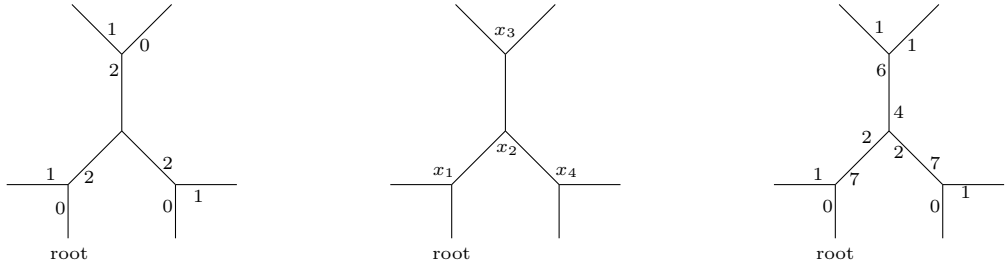

\begin{example}\label{3aspectos}{\em
Figure~\ref{ejemplitos} shows three aspects of the smallest possible non-linear tree $T_0$. The four essential vertices are labelled \red{(following the $T_0$-order)} in the central picture. The fact that the 4-fold product
\begin{equation}\label{strwea}
\langle 0,x_1,\red{(1)},\red{(7)}\rangle\langle 2,x_2,\red{(4)},\red{(2)}\rangle\langle 6,x_3,\red{(1)},\red{(1)}\rangle\langle7,x_4,\red{(1)},\red{(0)}\rangle\in H^4(B_9T_0;R)
\end{equation}
is a basis element follows from Theorem~\ref{maintheorem}, as inspection in the picture on the right of Figure~\ref{ejemplitos}
reveals that the factors in~(\ref{strwea}) interact strongly. \red{Note that $r=s=1$ for each factor in~(\ref{strwea}), and that the cases with a strict inequality in~(\ref{localinteraction}) hold as required in the last clause of item (c) of Definition~\ref{interactioncomplex}.} Likewise, interaction analysis in the picture on the left exhibits the well known fact that $K_4T_0$ is not flag (i.e., $B_4T$ is not a right-angled Artin group): the three basis elements
$\langle0,x_1,\red{(1)},\red{(2)}\rangle$, $\langle2,x_3,\red{(1)},\red{(0)}\rangle$ and $\langle2,x_4,\red{(1)},\red{(0)}\rangle$ in $H^1(B_4T_0;R)$ have pairwise strong interactions (so their three double products are part of a basis of $H^2(B_4T_0)$), \red{but the three basis elements do not interact (so} their triple product vanishes).
}\end{example}

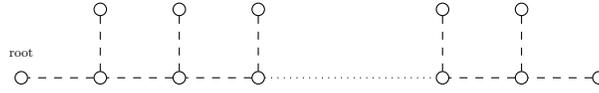
\begin{figure}[h!]
\centering
\begin{tikzpicture}[x=.7cm,y=.7cm,every node/.style={circle, draw, scale=.5}, scale=1.0, rotate = 180, xscale = -1]
		\node[label={\red{root}}] (1) at ( 0, 3.2) {};
		\node (2) at ( 1.5, 3.2) {};
		\node (3) at ( 3, 3.2) {};
		\node (4) at ( 8, 3.2) {};
		\node (6) at ( 4.5, 3.2) {};
		\node (8) at ( 4.5, 1.9) {};
		\node (5) at ( 9.5, 3.2) {};
		\node (7) at ( 1.5, 1.9) {};
		\node (9) at ( 3, 1.9) {};
		\node (14) at ( 8,1.9) {};
		\node (15) at ( 9.5, 1.9) {};
		\node (21) at ( 10.97, 3.2) {};
		\draw[dashed] (2) -- (1);
		\draw[dashed] (3) -- (2);
		\draw[dashed] (6) -- (8);
		\draw[dashed] (6) -- (3);
		\draw[dotted] (6) -- (4);
		\draw[dashed] (5) -- (4);
		\draw[dashed] (7) -- (2);
		\draw[dashed] (3) -- (9);
		\draw[dashed] (14) -- (4);
		\draw[dashed] (15) -- (5);
		\draw[dashed] (21) -- (5);
	\end{tikzpicture}
\caption{A planar embedding of a \red{binary} linear tree}
\label{presentation}
\end{figure}

\begin{example}\label{linearcase}{\em
Let $T$ be a \red{binary} tree \red{whose essential vertices lie along a single embedded arc.} Choosing the planar embedding shown in Figure~\ref{presentation}, we see that $B_nT$ has a right-angled Artin group presentation with generators $\langle k,\red{x},p,q\rangle$, where $\red{x}$ is an essential vertex of $T$ \red{and} $k,p,q$ are non-negative integer \red{numbers}\footnote{\red{Instead of writing the 1-tuples $(p)$ and $(q)$, we have simply written $p$ and $q$.} } \red{satisfying} $p>0$ \red{and} $k+p+q=n-1$. \red{In these terms, $B_nT$ has} a commutativity relation $\langle k,\red{x},p,q\rangle \langle k',\red{x'},p',q'\rangle= \langle k',\red{x'},p',q'\rangle \langle k,\red{x},p,q\rangle$ whenever $\red{x<x'}$ and $q+\red{k'}\red{{}\geq{}} n$, \red{where the former inequality refers to the $T$-order resulting from the embedding. Note that the chosen planar embedding of $T$ rules out weak interactions.}
}\end{example}
 
\begin{theorem}\label{casogeneral}
For any tree $T$, any non-negative integer $n$ and any commutative ring $R$ with unit 1, there is a set-theoretic inclusion $V_nT\hookrightarrow H^1(B_nT;R)$ so that the faces of $K_nT$ yield, via cup-product of their vertices, a graded basis of $H^*(B_nT;R)$. For instance, the empty face $\varnothing\in K_nT$ corresponds to the unit $1\in H^0(B_nT;R)=R$. Furthermore, any product $\langle k,x,p,q\rangle \cdot \langle k',x',p',q'\rangle$ with $x=x'$ vanishes (in particular cup-squares vanish), as do cup-products of non-interacting elements in $V_nT$. 
\end{theorem}

The only piece of multiplicative information missing in Theorem~\ref{casogeneral}, namely a description of cup-products of weak-interacting basis elements in $V_nT$, is fully addressed in Section~\ref{lgp} (see Theorem~\ref{productsvialowerpaths}) through the concept of ``interaction parameters'' introduced in Section~\ref{sectioncupproducts} (Definition~\ref{interactionparameters}).

\begin{remark}\label{obstruction}{\em
The only obstructions for realizing $H^*(B_nT;R)$ in Theorem~\ref{casogeneral} as the exterior face ring determined by $K_nT$ are the non-vanishing products whose factors interact weakly. For trees with binary core, such weak-interacting non-trivial products are effectively ruled out in the final section of this paper (Theorem~\ref{binarycore}) by means of a suitable change of basis that adjusts the inclusion $V_nT\hookrightarrow H^1(B_nT;R)$ in Theorem~\ref{casogeneral}.
}\end{remark}


\begin{remark}\label{comparewithSchierer}{\em
The results in this paper allow us to recover and generalize Scheirer's main technical tool \cite[Lemma~3.6]{MR3773741} for studying Farber's topological complexity of $B_nT$. Extensions of Scheirer's results will be the topic of a future publication.
}\end{remark}

In the rest of the paper we shall omit writing the coefficient ring $R$ in cohomology groups and associated (co)chain complexes.

\section{Preliminaries}  
We start by collecting the ingredients and facts we need: cup-products in the cubical setting (\cite{MR2028588,MR3105945}), reviewed in Subsection~\ref{sectionofcubicalproducts}, Forman's discrete Morse theory  (\cite{MR1358614,MR1926850}), reviewed in Subsection~\ref{secndmt}, and Farley-Sabalka's gradient field on Abrams' discrete model for (ordered and unordered) graph configuration spaces (\cite{MR2701024,jorgetereyo,MR2171804,MR2833585}), reviewed in Subsection~\ref{abrfarsabgradient}. This will set the notation we use in the rest of the paper.

\subsection{Cup products in cubical sets}\label{sectionofcubicalproducts}
An elementary cube in $\mathbb{R}^k$ is a cartesian product $c=I_1\times\cdots\times I_k$ of intervals $I_i=[m_i,m_i+\epsilon_i]$, where $m_i\in\mathbb{Z}$ and $\epsilon_i\in\{0,1\}$. For simplicity, we write $[m]:=[m,m]$ for a degenerate interval. \red{We say that $c$ is an $\ell$-cube if there are $\ell$ non-degenerate intervals among the cartesian factors $I_j$ of $c$, say $I_{i_1},\ldots,I_{i_\ell}$ with $1\leq i_1<\cdots<i_\ell\leq k$. In such a case,} the product orientation of $c$ is determined by (a) the orientation (from smaller to larger endpoints) of the non-degenerate intervals $I_{i_1},\ldots,I_{i_\ell}$, and (b) the order $i_1<\cdots<i_\ell$, i.e., the order of factors in the cartesian product. Under these conditions, and for $1\leq r\leq \ell$, set
\begin{equation}\label{productboundaries}\begin{aligned}
\delta_{2r}(c)&=I_1\times\cdots\times I_{i_r-1}\times[m_{i_r}+1]\times I_{i_r+1}\times\cdots\times I_{k}, \\ \delta_{2r-1}(c)&=I_1\times\cdots\times I_{i_r-1}\times[m_{i_r}]\times I_{i_r+1}\times\cdots\times I_{k}.
\end{aligned}\end{equation}
Then, for a cubical set $X\subset\mathbb{R}^k$, i.e., a union of elementary cubes in $\mathbb{R}^k$, the boundary map \red{$\partial\colon C_\ell(X)\to C_{\ell-1}(X)$} in the oriented cubical chain complex $C_*(X)$ is determined by
\begin{equation}\label{productboundary}
\partial\left(c\right)=\sum_{r=1}^\ell(-1)^{r-1}\left(\rule{0mm}{4mm}\delta_{2r}(c)-\delta_{2r-1}(c)\right).
\end{equation}
For instance, the oriented \red{cubical} boundary of the square $[0,1]\times[0,1]$ can be depicted as
$$
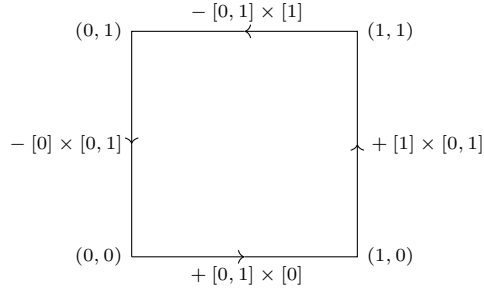
\begin{tikzpicture}[x=.5cm,y=.5cm]
\draw[->](0,0)--(3,0);\node[below] at (3,0) {\scriptsize${}+[0,1]\times [0]$};
\draw[->](3,0)--(6,0)--(6,3);\node[right] at (6,3) {\scriptsize${}+[1]\times [0,1]$};
\draw[->](6,3)--(6,6)--(3,6);\node[above] at (3,6) {\scriptsize${}-[0,1]\times [1]$};
\draw[->](3,6)--(0,6)--(0,3);\node[left] at (0,3) {\scriptsize${}-[0]\times [0,1]$};
\draw(0,3)--(0,0);
\node[left] at (0,0) {\scriptsize$(0,0)$};
\node[left] at (0,6) {\scriptsize$(0,1)$};
\node[right] at (6,6) {\scriptsize$(1,1)$};
\node[right] at (6,0) {\scriptsize$(1,0)$};
\end{tikzpicture}
$$

\begin{example}\label{treeasacubicalset}{\em
\red{Let $T$ be a tree whose vertices and edges have been ordered as described in the previous section. Think of $T$ as cubical set. In fact, orient the edges of $T$ from the smaller to the larger endpoints and fix an orientation-preserving embedding $T\subset\mathbb{R}^t$ of cubical sets, where elementary cubes in $\mathbb{R}^t$ have product orientation. \red{Thus, a vertex of $T$ becomes a 0-cube $[k_1]\times\cdots\times[k_t]$ in $\mathbb{R}^t$, while} an oriented edge in $T$ corresponds in $\mathbb{R}^t$ to an oriented 1-cube $I_1\times\cdots\times I_t$, i.e., an elementary cube all but one of its interval factors $I_j$ are degenerate.}
}\end{example}

Cup products in cubical cohomology are fairly similar to their classic simplicial counterparts. At the oriented cubical cochain level, there is a cup product graded map $C^*(X)\times C^*(X)\to C^*(X)$ that is associative, $R$-bilinear and is described on basis elements as follows. Firstly, for intervals $[a,b]$ and $[a',b']$, let
$$[a,b]\cdot[a',b']:=\begin{cases} [a,b'], & \mbox{if $b=a'$ and either $a=b$ or $a'=b'$ (or both);} \\ 0, & \mbox{otherwise.}\end{cases}$$
Then, for elementary cubes $c=I_1\times\cdots\times I_k$ and $d=J_1\times\cdots\times J_k$ in $X$, the cubical cup product $c\cdot d$ of the corresponding basis elements\footnote{We shall omit the use of an asterisk for dual elements. The intended meaning will be clear from the context.} $c,d\in C^*(X)$ vanishes if either $I_i\cdot J_i=0$ for some $i\in\{1,\ldots,k\}$ or, else, if $(I_1\cdot J_1)\times\cdots\times(I_k\cdot J_k)$ is not a cube in $X$; otherwise $c\cdot d$ is up to a sign $\epsilon_{c,d}$, the dual of the cube $(I_1\cdot J_1)\times\cdots\times(I_k\cdot J_k)$. Given our product-orientation settings, the sign is given by the usual algebraic-topology convention:
$$
\epsilon_{c,d}=\sum_{j=1}^{k-1}\left(\dim J_j \sum_{i=j+1}^{k}\dim I_i\right).
$$

\begin{remark}\label{agreeable}{\em
Particularly agreeable is the fact that a finite cartesian product of cubical sets comes equipped for free with the obvious structure of a cubical set. \red{For instance, in the situation of Example~\ref{treeasacubicalset},} the cartesian power $T^n$ is a (product-oriented) cubical set in $\mathbb{R}^{nt}$. In \red{such a} setting, an oriented cube $c=c_1\times\cdots\times c_n$ in $T^n$ (where each $c_i$ is either a vertex or an edge of $T$) corresponds in $\mathbb{R}^{nt}$ to an oriented cube $\left(I_{1,1}\times\cdots\times I_{1,t}\right)\times\cdots\times\left(I_{n,1}\times\cdots\times I_{n,t}\right)$ where, for each $i=1,\ldots,n$, at most one of the intervals $I_{i,1},\ldots,I_{i,t}$ is non-degenerate. These considerations, coupled with the fact that cubes of a single factor $T$ are at most one-dimensional, yield \red{the next} explicit description of cubical cup-products associated to $T$ and $T^n$.}
\end{remark}

\begin{proposition}\label{productsinTn}
The cup product in $C^*(T)$ of the duals of a pair of (oriented) cubes $c$ and $d$ in $T$ is given by the dual of
$$
c\cdot d=\begin{cases} (x,y), & \mbox{if $c=(x,y)$, an edge of $T$, and $d=y$, a vertex of $T$;} \\ (x,y), & \mbox{if $c=x$, a vertex of $T$, and $d=(x,y)$, an edge of $T$;} \\ $\ \ x,$ & \mbox{if $c=d=x$, a vertex of $T$;}\\ \ \ 0,& \mbox{otherwise.}
\end{cases}
$$
More generally, let $D$ be a (product-oriented) cubical subset of $\hspace{.3mm}T^n$. The cup product in $C^*(D)$ of the duals  of a pair of cubes $c=c_1\times\cdots\times c_n$ and $d=d_1\times\cdots\times d_n$ in $D$ vanishes provided $c_i\cdot d_i=0$ for some $i\in\{1,\ldots,n\}$ or, else, provided the cube $c\cdot d:=(c_1\cdot d_1)\times\cdots\times(c_n\cdot d_n)$ is not contained in $D$. Otherwise, the cup product is the multiple $(-1)^{\varepsilon_{c,d}}$ of the dual of $c\cdot d$, where $$\varepsilon_{c,d}=\sum_{j=1}^{n-1}\left(\dim(d_j)\sum_{i=j+1}^n\dim(c_i)\right).$$
\end{proposition} 

\subsection{Discrete Morse theory}\label{secndmt}
Let $X$ denote a finite regular cell complex with face poset $(\mathcal{F},\subset$), i.e., $\mathcal{F}$ is the set of (closed) cells of $X$ partially ordered by inclusion. For a cell $a\in\mathcal{F}$, we write $a^{(p)}$ to indicate that $a$ is $p$-dimensional. We think of the Hasse diagram $H_\mathcal{F}$ of $\mathcal{F}$ as a directed graph: it has vertex set $\mathcal{F}$, while directed edges (called also ``arrows'') are given by the family of ordered pairs $(a^{(p+1)},b^{(p)})$ with $b\subset a$. Such an arrow will be denoted as $a^{(p+1)}\searrow b^{(p)}$. Let $W$ be a partial matching on $H_\mathcal{F}$, i.e., a directed subgraph of $H_\mathcal{F}$ whose vertices have degree \red{precisely}~1. The modified Hasse diagram $H_\mathcal{F}(W)$ is the directed graph obtained from $H_\mathcal{F}$ by reversing all arrows of $W$. A reversed edge is denoted as $b^{(p)}\nearrow a^{(p+1)}$, in which case $a$ is said to be $W$-collapsible and $b$ is said to be $W$-redundant.

Discrete Morse theory focuses on gradient paths, i.e., directed paths in $H_\mathcal{F}(W)$ given by an alternate chain of up-going and down-going arrows,
\begin{equation}\label{path}
a_0\nearrow b_1\searrow a_1\nearrow\cdots\nearrow b_k\searrow a_k \mbox{ \ \ and \ \ } c_0\searrow d_1\nearrow c_1\searrow\cdots\searrow d_k\nearrow c_k.
\end{equation}
\red{A gradient path as the one on the left (right) hand-side of~(\ref{path}) is called an upper (respectively, lower) path,} and the gradient path is called elementary when $k=1$, or constant when $k=0$. The sets of upper and lower paths that start on a $p$-cell $a$ and end on a $p$-cell $b$ are denoted by $\overline{\Gamma}(a,b)$ and $\underline{\Gamma}(a,b)$, respectively. Concatenation of upper/lower paths $\overline{\Gamma}(a,b)\times\overline{\Gamma}(b,c)\to\overline{\Gamma}(a,c)$ and $\underline{\Gamma}(a,b)\times\underline{\Gamma}(b,c)\to\underline{\Gamma}(a,c)$ is defined in the obvious way; for instance, any upper/lower path is a concatenation of corresponding elementary paths. An upper/lower path is called a cycle if $a_0=a_k$ in the upper case of (\ref{path}), or $c_0=c_k$ in the lower case. (By construction, the cycle condition can only hold with $k>1$.) The matching $W$ is said to be a gradient field on~$X$ if $H_\mathcal{F}(W)$ has no cycles. In such a case, cells of $X$ that are neither $W$-redundant nor $W$-collapsible are said to be $W$-critical or, simply, critical when $W$ is clear from the context. We follow Forman's convention to use capital letters to denote critical cells.

It is well known that a gradient field on $X$ carries all the homotopy information of $X$. For our purposes, we only need to recall \red{how} gradient paths recover (co)homological information. In the rest of the section we assume $W$ is a gradient field on~$X$.

Start by fixing an orientation on each cell of $X$ and, for cells $a^{(p)}\subset b^{(p+1)}$, consider the incidence number $\iota_{a,b}$ of $a$ and $b$, i.e., the coefficient ($\pm1$, since $X$ is regular) of $a$ in the expression of $\partial(b)$. Here $\partial$ is the boundary operator in the cellular chain complex $C_*(X)$. The Morse cochain complex $\mathcal{M}^*(X)$ is then defined to be the graded $R$-free\footnote{Cochain coefficients are taken in \red{a ground} ring $R$, \red{as} we are interested in cup-products.} module generated in dimension $p\geq0$ by the duals\footnote{Recall we omit the use of an asterisk for dual elements.} of the oriented critical cells $A^{(p)}$ of $X$. The definition of the Morse coboundary map in $\mathcal{M}^*(X)$ requires the concept of multiplicity of upper/lower paths. In the elementary case, multiplicity is given by
\begin{equation}\label{multiplicitydefinition}
\mu(a_0\nearrow b_1\searrow a_1)=-\iota_{a_0,b_1}\cdot\iota_{a_1,b_1}\mbox{ \ \ and \ \ }
\mu(c_0\searrow d_1\nearrow c_1)=-\iota_{d_1,c_0}\cdot\iota_{d_1,c_1},
\end{equation}
and, in the general case, it is defined to be a multiplicative function with respect to concatenation of elementary paths. The Morse coboundary is then defined by
\begin{equation}\label{morsecoboundary}
\partial(A^{(p)})=\sum_{B^{(p+1)}}\left(\sum_{b^{(p)}\subset B} \left(\iota_{b,B} \sum_{\gamma\in\overline{\Gamma}(b,A)}\mu(\gamma)\right)\right)\cdot B.
\end{equation}
In other words, the Morse theoretic incidence number of $A$ and $B$ is given by the number of \red{``mixed'' gradient} paths \red{$\overline{\gamma}$ from $B$ to $A$ given as the concatenation of an arrow $B\searrow b$ and a path $\gamma\in\overline{\Gamma}(b,A)$, counted with multiplicity $\mu(\overline{\gamma}):=\iota_{b,B}\cdot\mu(\gamma)$}.

Gradient paths \red{yield}, in addition, a \red{homotopy equivalence} between $\mathcal{M}^*(X)$ and the usual cellular cochain complex $C^*(X)$. Indeed, the formul\ae
\begin{equation}\label{quasi-isomorphisms}
\overline{\Phi}(A^{(p)})=\sum_{a^{(p)}}\left(\sum_{\gamma\in\overline{\Gamma}(a,A)}\mu(\gamma)\right)a
\qquad\mbox{and}\qquad
\underline{\Phi}(a^{(p)})=\sum_{A^{(p)}}\left(\sum_{\gamma\in\underline{\Gamma}(A,a)}\mu(\gamma)\right)A
\end{equation}
define (on generators) cochain maps $\overline{\Phi}\colon \mathcal{M}^*(X)\to C^*(X)$ and $\underline{\Phi}\colon C^*(X)\to \mathcal{M}^*(X)$ inducing cohomology isomorphisms $\overline{\Phi}^*$ and $\underline{\Phi}^*$ with $(\underline{\Phi}^*)^{-1}=\overline{\Phi}^*$.

\subsection{Abrams discrete model and Farley-Sabalka's gradient field}\label{abrfarsabgradient}
\red{For a tree $T$, think of $T^n$ as the  cubical set described in Remark~\ref{agreeable}.} Abrams discrete model for $\F_nT$ is the largest cubical subset $\DnT$ of $T^n$ inside $\F_nT$. In other words, $\DnT$ is obtained by removing open cubes from $T^n$ whose closure intersect the fat diagonal. As usual, the symmetric group $\Sigma_n$ acts on the right of $\DnT$ by permuting factors. The action permutes in fact cubes, and the quotient complex is denoted by $\UDnT$. Following Farley-Sabalka's lead, from now on we use the notation $(a_1,\ldots,a_n)$, and even $(a)$, for a cube $a_1\times\cdots\times a_n$ in $T^n$ (so each $a_i$ is either a vertex or an edge of $T$), and the notation $\{a_1,\ldots,a_n\}$, and even $\{a\}$, for the corresponding $\Sigma_n$-orbit. Beware not to confuse the parenthesis notation with a point of $T^n$, or the braces notation with a set of elements of~$T$ ---even if all the $a_i$'s are vertices. The ``coordinates'' $a_i$ in a cube $(a)$ or in its $\Sigma_n$-orbit $\{a\}$ are referred to as the ingredients of the cube. Closures of ingredients \red{of cubes in $D_nT$ and $\UDnT$} are therefore pairwise disjoint.

In his Ph.D.~thesis, Abrams showed that $\DnT$ is a $\Sigma_n$-equivariant strong deformation retract of  $\F_nT$ provided \red{$T$ is $n$-sufficiently subdivided in the sense that} each path in $T$ between distinct vertices of degree not equal to 2 passes through at least $n-1$ edges. \red{Such a condition will be in force throughout the paper, although it} is not a real restriction because $T$ can be subdivided as needed without altering the homeomorphism type of its configuration spaces. The $\Sigma_n$-equivariance of the strong deformation retraction above implies that $\UDnT$ is a strong deformation retract of $\UF_nT$. Consequently, we will switch attention from $\F_nT$ and $\UF_nT$ to their homotopy equivalent discrete models $\DnT$ and $\UDnT$.

For a vertex $x$ of $T$ different from the root $0$, let $e_x$ be the unique edge of $T$ of the form $(y,x)$ ---recall this requires $y<x$. \red{Let $c$ be a cube either in $\DnT$ or $\UDnT$.} A vertex-ingredient $x$ of $c$ is said to be blocked in $c$ if $x=0$ or, else,
if replacing in $c$ the ingredient $x$ by the edge $e_x$ fails to render a cube in the corresponding discrete model; $x$ is said to be unblocked in $c$ otherwise. An edge-ingredient $e$ of a cube $c$ is said to be order-disrespectful in $c$ provided $e$ is of the form $(x,\red{y})$ and \red{there is a vertex ingredient $z$ in $c$ with $x<z<y$ and $z$ adjacent to $x$ (in particular $x$ must be an essential vertex);} $e$ is said to be order-respecting in $c$ otherwise. Blocked vertex-ingredients and order-disrespectful edge ingredients in $c$ are \red{said} to be critical. Farley-Sabalka's gradient field (on $\DnT$ and $\UDnT$) then works as follows. Order the ingredients of a cube $c$ by \red{their} $T$-ordering (as described in Section \ref{mainresult}), and look for non-critical ingredients:
\begin{description}
\item[\emph{(i)}] If the first such ingredient is an unblocked vertex $y$ in $c$, then $c$ is redundant, and one sets $c\nearrow c'$, where $c'$ is the cube obtained from $c$ by replacing $y$ by $e_y$. We say that the pairing $c\nearrow c'$ creates the edge $e_y$. In this case $e_y$ is an order-respecting edge \red{in} $c'$, and all ingredients of $c'$ smaller than $e_y$ are critical.
\item[\emph{(ii)}] If the first such ingredient is an order-respecting edge $(w,z)$ \red{in $c$,} then $c$ is collapsible, and one sets $c''\nearrow c$, where $c''$ is the cube obtained from $c$ by replacing $(w,z)$ by $z$. Again, we say that the edge $(w,z)$ is created by the pairing $c''\nearrow c$. In this case $z$ is an unblocked vertex \red{in} $c''$, and all ingredients of $c''$ smaller than $e_z$ are critical.  
\item[\emph{(iii)}] If all ingredients of $c$ are critical, then $c$ is critical.
\end{description}

\begin{definition}
\red{For a vertex $x$ and a non-negative integer $t$, let $S_x(t)$ stand for the family of vertices $x,x+1,\ldots,x+t-1$. We think of $S_x(t)$ as a size-$t$ stack of vertices supported by $x$. Whenever we use such a stack of vertices, the $n$-sufficiently subdivided condition on $T$ will assure the existence of the required $t$ vertices. Furthermore, for $\ell\in\{0,1,\ldots,d(x)-1\}$, let $x[\ell\hspace{.4mm}]$ denote the vertex adjacent to $x$ that lies in $x$-direction~$\ell$. For instance $x[0]=x-1$ and $x[1]=x+1$, if $x$ is essential.}
\end{definition}

\begin{figure}[h!]
\centering
\begin{tikzpicture}
\node[style={circle, draw, scale=.6}, scale=1.0, xscale = -1] (1) at ( 1.3, 3) {0};
\node[scale=.2] (2) at ( 3.06, 3) {};
\node[scale=.2] (3) at ( 4.53, 3) {};
\node[style={circle, draw, scale=.5}] (4) at ( 6.0, 3) {$x_i$};
\node[style={circle, draw, scale=.6}] (5) at ( 6.6, 2.4) {};
\node[style={circle, draw, scale=.6}] (7) at ( 6.9, 3) {};
\node[style={circle, draw, scale=.6}] (8) at ( 7.8, 3) {};
\node[style={circle, draw, scale=.6}] (9) at ( 8.7, 3) {};
\node[style={circle, draw, scale=.6}] (10) at ( 6.8, 3.4) {};
\node[style={circle, draw, scale=.6}] (11) at ( 7.6, 3.8) {};
\node[style={circle, draw, scale=.6}] (6) at ( 8.4, 4.2) {};
\node[style={circle, draw, scale=.6}] (12) at ( 6.4, 3.9) {};
\node[style={circle, draw, scale=.6}] (13) at ( 2.3, 3 ) {};
\draw[very thin] (13) -- (1); \draw[very thin] (13) -- (2);
\draw[dashed, very thin] (3) -- (2);
\draw[very thin] (4) -- (3);
\draw[very thin] (5) -- (4);
\draw[very thin] (6) -- (11);
\draw[very thick] (7) -- (4);
\draw[very thin] (8) -- (7);
\draw[very thin] (9) -- (8);
\draw[very thin] (10) -- (4);
\draw[very thin] (11) -- (10);
\draw[very thin] (12) -- (4);
\end{tikzpicture}
\caption{\red{Critical ingredients blocked by the root ($k=2$) and by an order-disrespectful edge $(x_i,x_i[3])$ ($r_i=2$, $t_{i,1}=1$, $t_{i,2}=3$, $t_{i,3}=2$ and $t_{i,4}=1$) }}
\label{ing}
\end{figure}
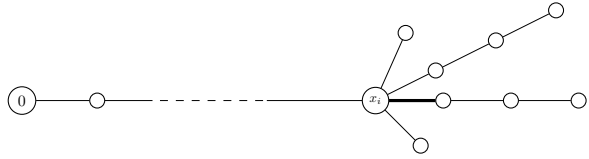

\begin{figure}[h!]
	\centering
	\begin{tikzpicture}[x=.6cm,y=.6cm]
\node[fill=black,style={circle,draw,scale=.4},scale=1.0,rotate = 180, xscale = -1] (1) at ( 3.12, 2.86) {};
\node[fill=black,style={circle, draw, scale=.4},scale=1.0,rotate = 180,xscale = -1] (24) at (3.7, 2.86) {};
\node[fill=black,scale=.1] (2) at ( 6.09, 2.86) {};
\node[fill=black,scale=.1] (3) at ( 7.12, 1.51) {};
\node[fill=black,scale=.1] (4) at ( 7.66, 2.88) {};
\node[fill=black,style={circle,draw,scale=.4},scale=1.0,rotate = 180, xscale = -1] (5) at ( 7.14, 3.58) {};
\node[fill=black,scale=.1] (6) at ( 6, 4.14) {};
\node[fill=black,style={circle,draw,scale=.4},scale=1.0,rotate = 180, xscale = -1] (7) at ( 8.19, 1.12) {};
\node[fill=black,style={circle,draw,scale=.4},scale=1.0, rotate = 180, xscale = -1] (8) at ( 8.69, 0.58) {};
\node[fill=black,style={circle,draw,scale=.4},scale=1.0, rotate = 180, xscale = -1] (9) at ( 8.9, 1.12) {};
\node[fill=black,style={circle,draw,scale=.4},scale=1.0, rotate = 180, xscale = -1] (10) at ( 8.7, 1.66) {};
\node[fill=black,style={circle, draw, scale=.4}, scale=1.0, rotate = 180, xscale = -1] (32) at (9.5,1.12){};
\node[fill=black,scale=.1] (11) at ( 8.66, 2.87) {};
\node[fill=black,scale=.1] (12) at ( 9.25, 2.35) {};
\node[fill=black,scale=.1] (13) at ( 9.26, 3.09) {};
\node[fill=black,scale=.1] (14) at ( 8.21, 4.35) {};
\node[fill=black,scale=.1] (15) at ( 8.92, 4.27) {};
\node[fill=black,scale=.1] (16) at ( 8.84, 4.76) {};
\node[fill=black,scale=.1] (17) at ( 6.79, 4.83) {};
\node[fill=black,scale=.1] (18) at ( 7.47, 5.15) {};
\node[fill=black,scale=.1] (19) at ( 6.98, 5.58) {};
\node[fill=black,scale=.1] (20) at ( 7.46, 1.01) {};
\node[fill=black,scale=.1] (21) at ( 6.95, 1.04) {};
\node[fill=black,scale=.1] (22) at ( 7.91, 2.39) {};
\node[fill=black,style={circle,draw,scale=.4},scale=1.0,rotate = 180,xscale = -1] (23) at ( 7.19, 4.08) {};
\node[fill=black,style={circle,draw,scale=.4},scale=1.0,rotate = 180,xscale = -1] (25) at ( 5.8, 5.14) {};
\node[fill=black,style={circle, draw, scale=.4}, scale=1.0, rotate = 180, xscale = -1] (33) at (5,6.1){};
\node[fill=black,scale=.1] (26) at ( 5.38, 4.7) {};
\node[fill=black,scale=.1] (27) at ( 8.41, 5.03) {};
\node[fill=black,style={circle,draw,scale=.4},scale=1.0,rotate = 180,xscale = -1] (28) at ( 7.67, 3.52) {};
\node[fill=black,style={circle,draw,scale=.4},scale=1.0,rotate = 180,xscale = -1] (29) at ( 6.14, 5.71) {};
\node[fill=black,style={circle,draw,scale=.4},scale=1.0,rotate = 180,xscale = -1] (30) at ( 5.33, 5.73) {};
\node[fill=black,scale=.1] (31) at ( 7.67, 1.71) {};
\node[fill=black,style={circle, draw, scale=.4}, scale=1.0, rotate = 180, xscale = -1] (34) at (8.2,3.46){};
\draw[very thin] (28) -- (34);
\draw[very thin] (33) -- (30);
\draw[very thin] (1)node[above]{\scriptsize 0} -- (24);
\draw[very thin] (24) -- (2);
\draw[very thin] (32) -- (9);
\draw[very thin] (3) -- (2);
\draw[very thin] (4) -- (2);
\draw[very thin] (5) -- (2);
\draw[very thin] (6) -- (2);
\draw[very thin] (7) -- (3);
\draw[very thin] (8) -- (7);
\draw[very thick] (7)node[below]{\scriptsize $x_3$} -- (9);
\draw[very thin] (7) -- (10);
\draw[very thin] (11) -- (4);
\draw[very thin] (13) -- (11);
\draw[very thin] (12) -- (11);
\draw[very thin] (14) -- (5);
\draw[very thin] (15) -- (14);
\draw[very thin] (16) -- (14);
\draw[very thin] (17) -- (6);
\draw[very thin] (19) -- (17);
\draw[very thin] (18) -- (17);
\draw[very thin] (3) -- (21);
\draw[very thin] (20) -- (3);
\draw[very thin] (22) -- (4);
\draw[very thin] (23) -- (5);
\draw[very thin] (6) -- (25);
\draw[very thin] (26) -- (6);
\draw[very thin] (27) -- (14);
\draw[very thick] (28) -- (5)node[below]{\scriptsize $x_2$};
\draw[very thick] (29) -- (25)node[right]{\scriptsize $x_1$};
\draw[very thin] (25) -- (30);
\draw[very thin] (31) -- (3);
\end{tikzpicture}
\caption{\red{A} critical \red{3}-cell $\{2 \l x_1,\red{(2),(0)} \l x_2,\red{(1,0),(1)} \l x_3,\red{(1),(1,1)}\}$}
\label{acriticalcube}
\end{figure}
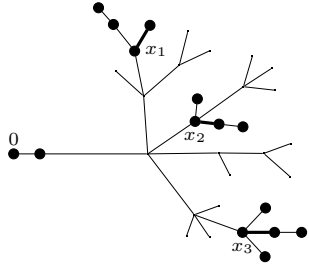

\red{As illustrated in Figures~\ref{ing} and~\ref{acriticalcube}, ingredients of a critical $m$-cube are spelled out through}
\begin{itemize}
\item[{\em (a)}] \red{a stack $S_0(k)$} of $k$ vertices \red{supported by the root (here $k\geq0$, i.e., $S_0(k)$ can be empty);}
\item[{\em (b)}] $m$ pairwise different essential vertices $x_1,\ldots,x_m$ of $T$ and, \red{for each $i=1,2,\ldots,m$, an order-disrespectful} edge $(x_i,\red{x_i[r_i+1]})$ \red{with $1\leq r_i< d(x_i)-1$;}
\item[{\em (c)}] for each $i=1,2,\ldots,m$ \red{and each $\ell=1,2,\ldots,d(x_i)-1$, a stack $S_{i,\ell}=S_{y_{i,\ell}}(t_{i,\ell})$ of $t_{i,\ell}$ vertices supported by the vertex}
$$\red{
y_{i,\ell}:=\begin{cases}
 x_i[\ell\hspace{.4mm}], & \mbox{if $\ell\neq r_i+1;$} \\
 x_i[\ell\hspace{.2mm}]+1, & \mbox{if $\ell=r_i+1$,}
\end{cases}
}$$
\end{itemize}
\red{subject to the requirements}
\begin{itemize}
\item[{\em (d)}] \red{some stacks $S_{i,\ell}$ might be empty, i.e., $t_{i,\ell}\geq0$ for all $i$ and $\ell$. Yet, for each $i$, there must exist an $\ell\in\{1,2,\ldots, r_i\}$ with $t_{i,\ell}>0$ (recall that $(x_i,x_i[r_i+1])$ is order-disrespectful);}
\item[{\em (e)}] \red{$k+m+\sum_{i,\ell}t_{i,\ell}=n$, i.e., the total number of ingredients is $n$.}
\end{itemize}

The critical cube \red{in} the unordered discrete model $\UDnT$ determined by the above information will be denoted as
\begin{equation}\label{criticalrepresentatives}
\left\{\rule{0mm}{4mm}k \l x_1,p_1,q_{1} \l \cdots \l x_m, p_m,q_m \right\}
\end{equation}
\red{where} $p_i=\red{(t_{i,1},\dots, t_{i,r_i})}$ and $q_i =\red{(t_{i,r_i+1}, \dots , t_{i,d(x_i)-1})}$.
Vertical bars are meant to stress the fact that each pair of parameters $p_{i}$ and $q_{i}$ are ordered and attached to $x_i$. Other than that,~\red{(\ref{criticalrepresentatives})} is indeed a set formed by the triples $(x_i,p_{i},q_{i})$ and the singleton~$k$. Figure \ref{acriticalcube} illustrates a \red{typical} critical cube.

\begin{remark}\label{farleysabalkapairsevenfaces}{\em
In any arrow $\red{d}\nearrow \red{c}$ of Farley-Sabalka's modified Hasse diagram, $\red{d}$ is an even face of~$\red{c}$, i.e., in the notation of (\ref{productboundaries}), $\red{d}=\delta_{2r}(\red{c})$ for some $r\in\{1,2,\ldots,\dim(\red{c})\}$.
}\end{remark}

\begin{remark}\label{equivariantFS}{\em
By construction, Farley-Sabalka's gradient field in $\DnT$ is $\Sigma_n$-equivariant and, by passing to the quotient, it yields the corresponding gradient field in $\UDnT$. Consequently, gradient paths can equivalently be analyzed in either the ordered or unordered settings. Indeed, a gradient path in $\UDnT$ corresponds to a ``$\Sigma_n$-orbit'' of gradient paths in $\DnT$. \red{Due to the cup-product descriptions in Subsection~\ref{sectionofcubicalproducts}, we find it more} convenient to perform the gradient-path analysis \red{at the level of the cubical set} $\DnT$.
}\end{remark}

\section{Gradient-path \red{dynamics}}\label{gradientpaths}
Recall from Subsection~\ref{sectionofcubicalproducts} that the product orientation of a $p$-dimensional cube $(c_1,\ldots,c_n)$ in $\DnT$ depends on (the orientation of edges ---from the smaller to the larger vertex--- in $T$ and on) the coordinate order $c_{i_1},\ldots,c_{i_p}$, i.e.~where $i_1<\cdots<i_p$, of the edge-ingredients. In particular, the quotient cube $\{c_1,\ldots,c_n\}$ in $\UDnT$ inherits no well defined orientation. The following definition avoids the problem and is well suited for the analysis of gradient paths in $\DnT$.

\begin{definition}[Gradient orientation, \red{cf.~Subsection 2.3 of \cite{MR2359035}}]
The listing \red{$(x_1,y_1),\ldots,(x_p,y_p)$} of edge-ingredients of a \red{$p$-}cube $c$ in $\DnT$ or in $\UDnT$ is said to be in gradient order if $x_1<\cdots<x_p$, where the latter is the $T$-ordering \red{of vertices} discussed in Section \ref{mainresult}. The gradient orientation of $c$ is defined just as the product orientation, except that the gradient order of the edge-ingredients is used \red{(rather than the coordinate order).}
\end{definition}

\red{In the rest of the paper,} and unless explicitly noted otherwise, we use gradient orientations. In doing so, the definitions of the cubes $\delta_{2r}(c)$ and $\delta_{2r-1}(c)$ in (\ref{productboundaries}) require a corresponding adjustment. Namely, if the edge-ingredients of a $p$-cube $c$ are listed in gradient order as $(x_1,y_1),\ldots,(x_p,y_p)$, then replacing the edge $(x_r,y_r)$ by the vertex $y_r$ or $x_r$ yields $\delta_{2r}(c)$ or $\delta_{2r-1}(c)$, respectively. Remark~\ref{farleysabalkapairsevenfaces} and the expression in (\ref{productboundary}) for cubical boundaries then remain unaltered. A first advantage of gradient orientations is that the map induced at the cochain level by the projection $\pi\colon\DnT\to\UDnT$ \red{involves no signs,}
\begin{equation}\label{simpleform}
\pi^*(\{c\})=\sum_{\sigma\in \Sigma_n}(c)\cdot\sigma.
\end{equation}
(Recall we omit asterisks for duals.) \red{In view of} Remark~\ref{equivariantFS}, \red{the homotopy equivalences in~(\ref{quasi-isomorphisms}) satisfy:}

\begin{lemma}\label{keydiagram}
The following diagram is commutative:
$$\xymatrix{
\mathcal{M}^*(\DnT)\ar[r] \ar[r]^{\overline{\Phi}} & C^*(\DnT)\ar[r] \ar[r]^{\underline{\Phi}} & \mathcal{M}^*(\DnT) \\
\mathcal{M}^*(\UDnT)\ar[r] \ar[u]^{\pi^*} \ar[r]_{\overline{\Phi}} & C^*(\UDnT)\ar[r] \ar[u]^{\pi^*} \ar[r]_{\underline{\Phi}} & \mathcal{M}^*(\UDnT). \ar[u]^{\pi^*} 
}$$
\end{lemma}

\begin{remark}\label{extensiondenotacion}{\em
The Morse differential in $\UDnT$ is trivial (see \cite{MR2216709} or \red{Proposition \ref{trivialmorsediff}} below). Therefore, for each $m\geq0$, a graded basis of $H^m(\UDnT)$ is given by the cohomology classes of the $\overline{\Phi}$-images of the duals of the critical cubes (\ref{criticalrepresentatives}). By abuse of notation\footnote{The context clarifies the meaning.}, the $\pi^*$-image\footnote{We prefer to compute products in the ordered setting in view of the explicit descriptions in Subsection \ref{sectionofcubicalproducts}.} of the cohomology class so determined will also be denoted by the corresponding expression~(\ref{criticalrepresentatives}). There is no loss of information because vertical maps in the previous diagram are injective and, more importantly, they induce injections in cohomology (the latter assertion follows from a standard transfer argument and the torsion-freeness of $H^*(\UDnT)$).
}\end{remark}

This section's goal is the description of a cocycle in $C^*(\DnT)$ that represents a given cohomology class $\{k\l x,p,q\}\in\image$ (Proposition \ref{underneathphi} below). This requires the following discussion of dynamics \red{for} upper-paths \red{that end at} critical cubes.

\begin{definition}\label{eorsor}
An edge-ingredient $(x,y)$ of a cube $c$ of $\DnT$ is said to be 
\begin{itemize}
\item edge order-respecting in $c$, written as ``$\hspace{.2mm}(x,y)$ is $\hspace{.2mm}\eor(c)$'', if there are no edge-ingredients $(a,b)$ in $c$ with $x<a<b<y$.
\item strongly order-respecting in $c$, written as ``$\hspace{.2mm}(x,y)$ is $\hspace{.2mm}\sor(c)$'', if $(x,y)$ is $\eor(c)$ and there is no vertex-ingredient $v$ in $c$ with $x<v<y$.
\end{itemize}
A Farley-Sabalka pairing $\delta_{2i}(c)\nearrow c$ that creates an edge-ingredient that is $\sor(c)$ is said to be of sor type; otherwise, it is said to be of branch type. \red{Likewise, $\delta_{2i}(c)\nearrow c$ is said to be of eor type if the edge-ingredient it creates is $\eor(c)$.} An \red{upper}  elementary path $\delta_{2i}(c)\nearrow c\searrow\delta_j(c)$ is said to be \red{of falling-vertex type (sor type, branch type, respectively) provided  $j=2i-1$ (\hspace{.4mm}$\delta_{2i}(c)\nearrow c$ is of sor type, $\delta_{2i}(c)\nearrow c$ is of branch type, respectively).}
\end{definition}

Note that, if $y$ is the vertex-ingredient in $\delta_{2i}(c)$ that is responsible for a pairing $\delta_{2i}(c)\nearrow c$, say creating the edge-ingredient $(x,y)$ of $c$, then $\delta_{2i-1}(c)$ is obtained from $\delta_{2i}(c)$ by replacing the vertex $y$ by $x$. In other words, in the falling-vertex type path $\delta_{2i}(c)\nearrow c\searrow\delta_{2i-1}(c)$, the vertex-ingredient $y$ \red{``falls''} to its predecesor $x$. \red{In particular,} elementary paths of falling-vertex type have multiplicity 1.

\begin{examples}\label{elementarypropertiesofeorandsor}{\em
Any edge-ingredient $(x,x+1)$ of $c$ is $\sor(c)$. \red{On the other hand, for an essential vertex $x$ and a positive direction $\ell\in\{1,2,\ldots,d(x)-1\}$ from $x$,} an edge-ingredient $(x,\red{x[\ell\hspace{.3mm}]})$ of $c$ is $\sor(c)$ if and only if $c$ has no ingredient, neither vertex nor edge, in \red{any of} the component\red{s} of $T\setminus\{x\}$ \red{lying in} $x$-direction\red{s} \red{$1,2,\ldots,\ell-1$.} Furthermore, if $(x,y)$ is an edge-ingredient of a face $\delta_j(c)$ of some cube $c$ of $\DnT$, then $(x,y)$ is $\sor(\delta_j(c))$ if and only if $(x,y)$ is $\sor(c)$. 
}\end{examples}

\red{The final observation in Examples~\ref{elementarypropertiesofeorandsor} is freely used in the proof of:}
\begin{proposition}\label{eorandsor}
\red{Let $(x_1,y_1),\ldots,(x_p,y_p)$ be the gradient-order listing of the edge-ingredients of a $p$-cube~$c$ in $\DnT$.}
\begin{enumerate}
\item If an arrow $\delta_{2i}(c)\nearrow c$ in the modified Hasse diagram for $\DnT$ \red{is of eor type,} then \red{$(x_i,y_i)$ is $\red{\sor(c)}$ and,} for any $k>2i$, $\delta_k(c)$ is collapsible.
\item If the edge $(x_i,y_i)$ is $\sor(c)$, then there is no upper path starting at a face $\delta_{j}(c)$ with $j<2i-1$ and ending \red{at} a critical cube.
\end{enumerate}
\end{proposition}
\begin{proof}
1. \red{By definition,} $\delta_{2i}(c)\nearrow c$ creates the edge-ingredient \red{$(x_i,y_i)$, which is assumed to be} $\eor(c)$. Since ingredients of $\delta_{2i}(c)$ smaller than \red{$y_i$} are critical, \red{$(x_i,y_i)$} is in fact $\sor(c)$. Thus, for $k\neq2i,2i-1$, \red{$(x_i,y_i)$} is $\sor(\delta_k(x))$ and, therefore, order-respecting in $\delta_k(x)$. On the other hand, for $k>2i$, $\delta_k(c)$ and $c$ have the same ingredients smaller than \red{$y_i$,} so that all ingredients in $\delta_k(c)$ smaller than \red{$(x_i,y_i)$} are critical. Thus, by definition, $\delta_k(x)$ is collapsible for $k>2i$.

\smallskip
2. Under the stated hypothesis, assume (for a contradiction) there is a gradient path
\begin{equation}\label{thepath}
c\searrow\delta_j(c)=:c_0\nearrow d_1\searrow c_1\nearrow\cdots\nearrow d_m\searrow c_m
\end{equation}
with $j<2i-1$, $m\geq0$ and $c_m$ critical. Then $(x_i,y_i)$ is $\sor(c_0)$ and, in particular, $(x_i,y_i)$ is order-respecting in $c_0$, which forces $m>0$. Recursively, if $(x_i,y_i)$ is an edge-ingredient of both $c_{\ell-1}$ and $c_\ell$ (and so of $d_\ell$), and $(x_i,y_i)$ is $\sor(c_{\ell-1})$, then  $(x_i,y_i)$ is forced to be ($\sor(d_\ell)$ and, thus,) $\sor(c_\ell)$. It is not possible that $(x_i,y_i)$ is an edge-ingredient of all the $c_\ell$'s, for then $(x_i,y_i)$ would be $\sor(c_m)$, which is impossible as $c_m$ is critical. Let $k$ be the first integer ($1\leq k\leq m$) for which $(x_i,y_i)$ is not an ingredient of $c_k$ ---so that $(x_i,y_i)$ is $\sor(c_\ell)$ for $0\leq\ell<k$. In particular, $(x_i,y_i)$ is order-respecting in $c_{k-1}$. \red{Thus,} the vertex-ingredient $v$ of $c_{k-1}$ responsible for the pairing $c_{k-1}\nearrow d_k$ in (\ref{thepath}) satisfies $v<y_i$ and, in fact, $v<x_i$, since $(x_i,y_i)$ is $\sor(c_{k-1})$. On the other hand, since the edge $(u,v)$ created by $c_{k-1}\nearrow d_k$ is order-respecting in $d_k$, and since $c_k$ is obtained from $d_k$ by replacing the edge $(x_i,y_i)$ by either $x_i$ or $y_i$, the inequalities $u<v<x_i<y_i$ yield that
\begin{equation}\label{ortoo}
\mbox{$(u,v)$ is order-respecting in $c_k$ too.}
\end{equation}
In particular, $c_k$ is not critical, so $k<m$. Let $w$ be the vertex-ingredient of $c_k$ responsable for the pairing $c_k\nearrow d_{k+1}$. By (\ref{ortoo}), we get the first inequality in $w<v<x_i<y_i$, so
\begin{itemize}
\item ($w$ is an ingredient of $c_k$) $\Rightarrow$ ($w$ is an ingredient of $d_k$ and, therefore, of $c_{k-1}$);
\item ($w$ is unblocked in $c_k$) $\Rightarrow$ ($w$ is unblocked in $d_k$ and, therefore, in $c_{k-1}$).
\end{itemize}
But, by definition, $v$ is the minimal unblocked vertex in $c_{k-1}$, so $v\leq w$, a contradiction.
\end{proof}

Proposition \ref{eorandsor} implies that upper paths ending at critical cubes have a forced behavior most of the time:
\begin{corollary}\label{dynamics}
Let $\gamma$ be an upper path in $\DnT$ that ends at a critical cube. Any \red{upper} elementary factor of $\gamma$ of sor type is of falling-vertex type.
\end{corollary}

\begin{example}\label{dynamicsofupperpathsindim1}{\em
\red{Let us be specific about} the dynamics of an upper path $\gamma\colon c_0\nearrow d_1\searrow c_1\nearrow\cdots\searrow c_m$ that ends at a critical 1-cube $c_m$. By the $\Sigma_n$-equivariance of the gradient field, we can assume $c_0=(u_1,\ldots,u_i,v_1,\ldots,v_j,(y,\red{y[d\hspace{.4mm}]}),w_1,\ldots,w_k)$ with \red{$d\in\{1,2,\ldots,d(y)-1\}$ and}
$$
u_1<\cdots<u_i<y<v_1<\cdots<v_j<\red{y[d\hspace{.4mm}]}<w_1<\cdots< w_k,
$$
i.e., $c_0$ is the \red{$\Sigma_n$-orbit} representative whose ingredients appear in the $T$-ordering. By Corollary~\ref{dynamics}, the start of $\gamma$ is forced to consist of falling-vertex elementary paths, \red{where} the vertices $u_1,\ldots,u_i$ \red{fall}, each at a time, \red{until they form the stack $S_0(i)$ if $i$ vertices supported (and blocked)} by the root. \red{At that point $\gamma$ arrives at the 1-cube $(S_0(i),v_1,\ldots,v_j,(y,\red{y[d\hspace{.4mm}]}),w_1,\ldots,w_k),$ and} we see that $j$ must be positive, for otherwise \red{$\gamma$} would have reached a collapsible \red{$1$-}cube. \red{In} particular $y$ must be an essential vertex \red{and $d>1$.} Then, \red{again by Corollary~\ref{dynamics},} it is the turn of vertices $v_1,\ldots,v_j$ \red{that are forced} to fall, each at a time, \red{until they form stacks $S_{y[\ell\hspace{.3mm}]}(t_\ell)$ of vertices blocked by $y$ in $y$-directions $\ell=1,\ldots,d-1$. At that point $\gamma$ arrives at a 1-cube of the form} 
\begin{equation}\label{pasoapaso}
\red{\left(S_0(i),S_{y[1]}(t_1),\ldots,S_{y[d-1]}(t_{d-1}),(y,y[d\hspace{.4mm}]),w_1,\ldots,w_k\right)\!.}
\end{equation}
\red{Not all of the stacks $S_{y[\ell\hspace{.3mm}]}(t_\ell)$ are empty, so~(\ref{pasoapaso}) has} $(y,\red{y[d\hspace{.4mm}]})$ as a critical edge-ingredient. \red{The falling-vertex process is also forced by Corollary~\ref{dynamics} on those vertices $w_1,\ldots,w_k$ that are located in positive $y$-directions (if any), and this takes $\gamma$ to a 1-cube of the form}
$$
\red{\left(\!S_0(i),\hspace{-.3mm}S_{y[1]}(t_1),\hspace{-.3mm}\ldots\hspace{-.3mm},S_{y[d-1]}(t_{d-1}),\hspace{-.3mm}(y,\hspace{-.3mm}y[d\hspace{.4mm}]),\hspace{-.3mm}S_{y[d\hspace{.3mm}]+1}(t_d),\hspace{-.3mm}S_{y[d+1]}(t_{d+1}),\hspace{-.3mm}\ldots\hspace{-.3mm},\hspace{-.3mm}S_{y[d(y)-1]}(t_{d(y)-1}),\hspace{-.3mm}w_\rho,\hspace{-.3mm}\ldots\hspace{-.3mm},\hspace{-.3mm}w_k\!\right)\!,}
$$
\red{with $w_\rho,\ldots,w_k$ all lying in $y$-direction 0.} Branching \red{starts} from this point on, \red{with} explicit options discussed \red{in the next paragraph.}

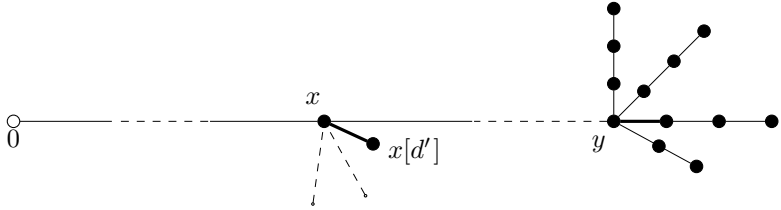
\begin{figure}[h!]
	\centering
	\begin{tikzpicture}[scale=1.0, rotate = 180, xscale = -1]
		\node[style={circle, draw, scale=.5}] (1) at ( 1.62, 3.2) {};
		\node[scale=.1] (2) at ( 2.94, 3.2) {};
		\node[scale=.1] (3) at ( 4.22, 3.2) {};
		\node[fill=black,style={circle, draw, scale=.5}] (4) at ( 5.75, 3.2) {};
		\node[style={circle, draw, scale=.1}] (5) at ( 5.6, 4.3) {};
		\node[scale=.1] (7) at ( 7.73, 3.2) {};
		\node[fill=black,style={circle, draw, scale=.5}] (8) at ( 6.4, 3.5) {};
		\node[style={circle, draw, scale=.1}] (9) at ( 6.3, 4.19) {};
		\node[fill=black,style={circle, draw, scale=.5}] (10) at ( 9.6, 3.2) {};
		\node (21) at ( 9.4, 3.5) {$y$};
              \node (22) at ( 5.6, 2.9) {$x$};
              \node (22) at ( 6.95, 3.6) {\red{$x[d'\hspace{.2mm}]$}};
		\node[fill=black,style={circle, draw, scale=.5}] (22) at ( 10.2, 3.53) {};
		\node[fill=black,style={circle, draw, scale=.5}] (16) at (10,2.8){};
		\node[fill=black,style={circle, draw, scale=.5}] (17) at (10.4,2.4){};
		\node[fill=black,style={circle, draw, scale=.5}] (18) at (10.8,2){};
		\node[fill=black,style={circle, draw, scale=.5}] (12) at ( 11, 3.2) {};
		\node[fill=black,style={circle, draw, scale=.5}] (13) at ( 10.7, 3.8) {};
		\node[fill=black,style={circle, draw, scale=.5}] (14) at (10.3,3.2){};
		\node[fill=black,style={circle, draw, scale=.5}] (15) at (11.7,3.2){};
\node[fill=black,style={circle, draw, scale=.5}] (19) at (9.6,2.7){};
\node[fill=black,style={circle, draw, scale=.5}] (20) at (9.6,2.2){};
\node[fill=black,style={circle, draw, scale=.5}] (30) at (9.6,1.7){};		
		\draw[very thin] (10) -- (19);
		\draw[very thin] (19) -- (20); \draw[very thin] (30) -- (20);
		\draw[very thin] (17) -- (18);
		\draw[very thin] (12) -- (15);
		\draw[very thin] (2) -- (1)node[below]{0};
		\draw (3)[dashed, very thin] -- (2);
		\draw (4)[very thin] -- (3);
		\draw[very thin] (7) -- (4);
		\draw[dashed, very thin] (10) -- (7);
		\draw[very thick] (14) -- (10);
		\draw[very thin] (12) -- (14);
		\draw[very thin] (16) -- (10);
		\draw[very thin] (17) -- (16);
		\draw[very thin] (13) -- (10);
		\draw[dashed, very thin] (5) -- (4);
		\draw[very thick] (8) -- (4);
		\draw[dashed,very thin] (9) -- (4);
	\end{tikzpicture}
	\caption{\red{A portion of the 1-cube} $d_{\red{\lambda}+1}$ \red{with its recently created} edge $(x,\red{x[d'\hspace{.2mm}]})$}
\label{branchtype}
\end{figure}

If \red{no vertices $w_\rho,\ldots,w_k$ are left,} then $\gamma$ \red{would have reached its final critical destination $c_m$.} Otherwise, \red{$w_{\rho}$} is forced to fall until $\gamma$ reaches, via some branch type pairing $c_{\red{\lambda}}\nearrow d_{\red{\lambda}+1}$, the 2-cube $d_{\red{\lambda}+1}$ depicted in Figure~\ref{branchtype}. At this point there are two options for $d_{\red{\lambda}+1}\searrow c_{\red{\lambda}+1}$. In the first option, $c_{\red{\lambda}+1}$ is obtained from $d_{\red{\lambda}+1}$ by replacing the recently created edge $(x,\red{x[d'\hspace{.2mm}]})$ by $x$, i.e., \red{with an upper elementary path} $c_{\red{\lambda}}\nearrow d_{\red{\lambda}+1}\searrow c_{\red{\lambda}+1}$ of falling-vertex type. In such a case, $\gamma$ is forced to continue with the vertex $x$ falling until it is added to the stack of vertices blocked by the root 0. This leaves us at a situation similar to the one at the start of this paragraph. In the second option, $c_{\red{\lambda}+1}$ is obtained from $d_{\red{\lambda}+1}$ by replacing the edge $(y,\red{y[d\hspace{.3mm}]})$ by either of its end points. In such a case, $\gamma$ is forced to continue:
\begin{enumerate}
\item with the falling of the vertices that are now unblocked in the neighborhood of $y$ (see Figure \ref{branchtype}), until they form a \red{stack} of vertices blocked by $x$ ---thus \red{starting} a critical situation around \red{the edge $(x,x[d'])$}--- and, then,
\item with the falling of the vertices (if any) in $x$-direction\red{s} \red{from $d'$ to $d(x)-1$}, which form (possibly empty) \red{stacks} of vertices blocked \red{either} by $x$ or $\red{x[d'\hspace{.2mm}]}$ \red{---thus completing the critical situation around the edge $(x,x[d'])$.}
\end{enumerate}
Again, this leaves us at a situation similar to the one at the start of this paragraph, \red{but} now with the edge $(x,\red{x[d'\hspace{.2mm}]})$ playing the role of the edge $(y,\red{y[d\hspace{.3mm}]})$. The branching process in this paragraph then repeats, necessarily a finite number of times, until all vertices $w_{\red{\rho}},\ldots,w_k$ have been considered, \red{when $\gamma$ reaches its critical destination $c_m$.}
}\end{example}

\begin{proposition}\label{underneathphi}
A cocycle in $C^*(\DnT)$ representing a \red{1-dimensional} cohomology class $\{k\l x,p,q\}$ \red{in} $\image$, \red{with $p=(p_1,\ldots,p_r)$ and $q=(q_1,\ldots,q_s)$,} is given by 
\begin{equation}\label{representingcohomologyclass}
\sum\left(u_1,\ldots,u_k,v_1,\ldots,v_{|p|},(x,\red{x[r+1]}),w_1,\ldots,w_{|q|}\right)\cdot\sigma,
\end{equation}
where the summation runs over
\begin{itemize}
\item all permutations $\sigma\in\Sigma_n$,
\item all possible vertices $u_1\hspace{-.3mm}<\hspace{-.3mm}\cdots\hspace{-.3mm}<\hspace{-.3mm}u_k$ in the component of $T\setminus\{x\}$ in $x$-direction 0,
\item all possible vertices $v_1\hspace{-.3mm}<\hspace{-.3mm}\cdots\hspace{-.3mm}<\hspace{-.3mm}v_{|p|}$ in the component\red{s} of $T\setminus\{x\}$ in $x$-directions \red{from $1$ to $r$ so that, for $i\in\{1,\ldots, r\}$, $p_i$ of the vertices $v_1\hspace{-.3mm}<\hspace{-.3mm}\cdots\hspace{-.3mm}<\hspace{-.3mm}v_{|p|}$ lie in $x$-direction $i$,}
\item all possible vertices $w_1\hspace{-.3mm}<\hspace{-.3mm}\cdots\hspace{-.3mm}<\hspace{-.3mm}w_{|q|}$ in the component\red{s} of $T\setminus\{x\}$ in $x$-directions \red{greater than $r$ so that, for $j\in\{r+1,\ldots, d(x)-1\}$, $q_{j-r}$ of the vertices $w_1\hspace{-.3mm}<\hspace{-.3mm}\cdots\hspace{-.3mm}<\hspace{-.3mm}w_{|q|}$ lie in $x$-direction $j$.}
\end{itemize}
\end{proposition}
\begin{proof}
By construction, the representing cocycle $z$ we need is obtained by chasing, on the left square of the diagram in Lemma \ref{keydiagram}, the dual of the unordered critical cube \red{$\{c\}$} whose ordered \red{critical} representative is 
$
\red{(c):=\left(S_0(k),S_{x[1]}(p_1),\ldots,S_{x[r]}(p_r),(x,x[r+1]),S_{x[r+1]+1}(q_1),S_{x[r+2]}(q_2),\ldots,S_{x[d(x)-1]}(q_s)\right)\!.}
$
By (\ref{quasi-isomorphisms}) and (\ref{simpleform}),
\begin{equation}\label{musumationforz}
z=\overline{\Phi}\circ\pi^*\left(\red{\{c\}}\right)=\sum_{\gamma\in\mathcal{G}}\mu(\gamma)\cdot\mathcal{S}_\gamma,
\end{equation}
where $\mathcal{G}$ is the set of upper paths $\gamma$ that start at a 1-cube $\mathcal{S}_\gamma$ and finish at a 1-cube of the form $c\cdot\sigma$ with $\sigma\in\Sigma_n$. Let $\mathcal{G}'$ be the set of paths $\gamma\in\mathcal{G}$ all whose \red{upper} elementary factors are of falling-vertex type. Since $\mu(\gamma)=1$ for $\gamma\in\mathcal{G}'$, the analysis in Example~\ref{dynamicsofupperpathsindim1} shows that the summands in (\ref{representingcohomologyclass}) arise from the summands in (\ref{musumationforz}) with $\gamma\in\mathcal{G}'$. It thus suffices to show
\begin{equation}\label{cancellation}
\sum_{\gamma\in\mathcal{G}\setminus\mathcal{G}'}\mu(\gamma)\cdot\mathcal{S}_\gamma=0,
\end{equation}
which will be done by constructing an involution $\iota\colon\mathcal{G}\setminus\mathcal{G}'\to \mathcal{G}\setminus\mathcal{G}'$ such that every pair of paths $\gamma$ and $\iota(\gamma)$ have the same origen but opposite multiplicities, i.e.,
\begin{equation}\label{propertiesrequired}
\mathcal{S}_{\iota(\gamma)}=\mathcal{S}_{\gamma}\quad\mbox{and}\quad\mu(\iota(\gamma))=-\mu(\gamma)
\end{equation}
---thus their contributions to (\ref{cancellation}) cancel each other out. For a path $\gamma\in\mathcal{G}\setminus\mathcal{G}'$, let $\gamma_\text{last}=\left(c\nearrow d\searrow e\right)$ denote the last elementary factor of $\gamma$ that is not of falling-vertex type. In the notation of Example~\ref{dynamicsofupperpathsindim1}, $e$ is obtained from $d$ by replacing an edge $(y,\red{y[d\hspace{.3mm}]})$ by either $y$ or $\red{y[d\hspace{.3mm}]}$, and both options are possible. Then $\iota(\gamma)$ is defined so to start with the same factorization of $\gamma$ into elementary paths, except for the elementary factor $\gamma_\text{last}$, for which the other end-point of $(y,\red{y[d\hspace{.3mm}]})$ is taken, and after which the rest of the elementary factors are of falling-vertex type ---just like for $\gamma$. Note that the ending 1-cubes of $\gamma$ and $\iota(\gamma)$ lie in the same $\Sigma_n$-orbit, so $\iota(\gamma)\in\mathcal{G}\setminus\mathcal{G}'$. The required properties (\ref{propertiesrequired}) follow from (the construction and \red{from}) the fact that elementary paths of falling-vertex type have multiplicity~1.
\end{proof}
 
The cancelation phenomenon in the previous proof allows us to give an easy gradient-path explanation of the main result in \cite{MR2216709}: the vanishing of the Morse differential in $\UDnT$. A variant of the cancellation phenomenon will also play an important role in our evaluation of cup products (Theorem~\ref{productsvialowerpaths} below). Thus, in preparation for that argument, we spell out the gradient proof of:

\begin{proposition}\label{trivialmorsediff}
The Morse differential in $\UDnT$ vanishes.
\end{proposition}
\begin{proof}
By Remark \ref{equivariantFS}, \red{it suffices to} do the gradient path analysis directly at the level of $\UDnT$. For a pair of \red{unordered} critical cubes $c^{(k)}$ and $d^{(k-1)}$, let $\Gamma(c,d)$ be the set of \red{mixed} gradient paths $\red{\gamma}\colon c\searrow\bullet\nearrow\bullet\searrow\cdots\searrow d$. By (\ref{morsecoboundary}), we only need to construct an involution $\iota\colon\Gamma(c,d)\to\Gamma(c,d)$ so that, for every $\gamma\in\Gamma(c,d)$, $\mu(\iota(\gamma))=-\mu(\gamma)$. (\red{Recall that} the multiplicity of $\gamma\in\Gamma(c,d)$ is the incidence number for $c\searrow\bullet$ multiplied by the multiplicity of the remaining upper path $\bullet\nearrow\bullet\searrow\cdots\searrow d$.) Let $\Gamma(c,d)_{\text{fall}}$ consist of the paths in $\Gamma(c,d)$ all whose \red{upper} elementary factors are of falling-vertex type. The definition of the restricted $\iota_{\text{fall}}\colon\Gamma(c,d)_{\text{fall}}\to\Gamma(c,d)_{\text{fall}}$ uses the two forms of replacing by a vertex the edge-ingredient at the start of  the path. Likewise, for $\Gamma(c,d)_{\text{branch}}:=\Gamma(c,d)-\Gamma(c,d)_{\text{fall}}$, the definition of the restricted $\iota_{\text{branch}}\colon\Gamma(c,d)_{\text{branch}}\to\Gamma(c,d)_{\text{branch}}$ uses the two forms of replacing by a vertex the edge ingredient at the last upper elementary factor that is not of falling-vertex type.
\end{proof}

Propositions \ref{productsinTn} and \ref{underneathphi} immediately yield:
\begin{corollary}\label{trivialsquares}
The product of two basis elements $\{k,x,p,q\},\{k',x',p',q'\}\in\image$ vanishes provided $x=x'$. In particular, squares of 1-dimensional elements in $\image$ are trivial.
\end{corollary}

\section{Cup products I: Upper gradient paths}\label{sectioncupproducts}
\red{The goal for this section and the next one is to get at a workable description of products
\begin{equation}\label{aproduct}
\red{\left\{k_1\ll x_1,(p_{1,1},\ldots,p_{1,r_1}\hspace{-.3mm}),(q_{1,1},\ldots,q_{1,s_1}\hspace{-.3mm})\right\}{\cdots}
\left\{k_m\ll x_m,(p_{m,1},\ldots,p_{m,r_m}\hspace{-.3mm}),(q_{m,1},\ldots,q_{m,s_m}\hspace{-.3mm})\right\}}
\end{equation}
in $\image$. Associated to such a product, from now on we set $p_i:=(p_{i,1},\ldots,p_{i,r_i})$, $q_i:=(q_{i,1},\ldots,q_{i,s_i})$, $|p_i|:=\sum_{\ell=1}^{r_i}p_{i,\ell}$, $|q_i|:=\sum_{\ell=1}^{s_i}q_{i,\ell}$, and make free use of \emph{(i)} the order-disrespectful edge $(x_i,x_i[r_i+1])$ encoded in the $i$-th factor of~(\ref{aproduct}), of \emph{(ii)} the conditions $k_i+\sum_{j} p_{i,j}+\sum_{j'} q_{i,j'}=n-1$, $r_i+s_i=d(x_i)-1$ and $r_i,s_i\geq1$, and of \emph{(iii)} the fact that, for each $i$, $k_i$ and all of the $p_{i,\ell}$ and $q_{i,\ell}$ are non-negative, with not all of the $p_{i,\ell}$ being zero. Additionally, in view of Corollary~\ref{trivialsquares}, we can safely assume $x_1<\cdots<x_m$. Last, we use the shorthand $$d_i:=d(x_i)-1 \mbox{ \ \ and \ \ }\overline{x}_i:=x_i[r_i+1].$$

We start by tuning up the definition in Section~\ref{mainresult} of the components $C_{i,\red{\ell_i}}$ of $T\setminus\{x_1,\ldots,x_m\}$.} 

\begin{definition}[\red{Leaves and pruned trees}]\label{hojas}
Set $T_{0,1}:=C_{0,1}$ and, for $1\leq i\leq m$ and $1 \leq \red{\ell_i} \leq d(x_i)-1$, 
\begin{equation*}
\red{T_{i,\ell_i} :=\begin{cases}
C_{i,\ell_i}\cup\{x_i\}, & \mbox{if $\ell_i \neq r_i+1$;} \\
C_{i,\ell_i}\setminus\text{Int}(x_i,\overline{x}_i), & \mbox{if $\ell_i=r_i+1$,}
\end{cases}}
\end{equation*}
where $\text{Int}(x_i,\overline{x}_i)$ stands for the interior of the edge $(x_i,\overline{x}_i)$. We think of \red{each} $T_{i,\red{\ell_i}}$ \red{$(0\leq i\leq m)$} as a rooted but possibly pruned tree. Namely, in the notation of Section~\ref{mainresult} and setting $x_0:=0$, the root of $T_{i,\red{\ell_i}}$ is \red{$x_i$, if $i=0$ or if $i>0$ with $\ell_i\neq r_i+1$, whereas the root of $T_{i,r_i+1}$ is $\overline{x}_i$. Furthermore, the} set of pruned leaves of $T_{i,\red{\ell_i}}$ is $L_{i,\red{\ell_i}}:=B(C_{i,\red{\ell_i}})\setminus\{x_i\}$. 
\end{definition}

\begin{remark}\label{separticionaadecuadamente}{\em
\red{Just as} the sets $L_{i,\red{\ell_i}}$ give a partition of $\{x_1,\ldots,x_m\}$, the \red{union of the} trees $T_{i,\red{\ell_i}}$ \red{agrees with the difference} $T\setminus\bigcup_{i=1}^m\text{Int}(x_i,\overline{x}_i)$. \red{Actually,} each vertex of $T$ other than $x_i$ for $ 1 \leq i \leq m$, as well as each semi-open edge $(x,y)\setminus\{y\}$ of $T$ \red{not of the form $(x_i,\overline{x}_i)\setminus\{\overline{x}_i\}$ with $1\leq i\leq m$,} belongs to a tree~$T_{i,\red{\ell_i}}$ for a unique $\ell_i$.
}\end{remark}

Definition~\ref{strongvsweak} is recast by the second part of:
\begin{definition}\label{interactionparameters}
\begin{enumerate}
\item For a $\tau$-tuple of integers $t=(t_1,\ldots,t_\tau)$, we write $t \geq 0$ to mean that $t_j\geq0$ for all $j\in\{1,\ldots,\tau\}$, reserving the expression $t>0$ to mean that $t\geq0$ with $t_j>0$ for at least one $j\in\{1,\ldots,\tau\}$. Also, when $t\geq0$, we write $\underline{t}$ to denote a generic tuple of integers $(t'_1, \dots, t'_\tau)\geq0$ satisfying $t'_j \leq t_j$ for all $j\in\{1,\ldots,\tau\}$ with in fact $t'_j < t_j$ for at least one $j\in\{1,\ldots,\tau\}$. We make no distinction between 1-tuples $(t_1)$ and integer numbers $t_1$ so, accordingly, we use $\underline{t_1}$ instead of $\underline{(t_1)}$. 
\item \red{The interaction parameters $\mathcal{R}_0$, $\mathcal{P}_i:=(\mathcal{P}_{i,1},\ldots,\mathcal{P}_{i,r_i})$ and $\mathcal{Q}_i:=(\mathcal{Q}_{i,1},\ldots,\mathcal{Q}_{i,s_i})$ of the factors in~$(\ref{aproduct})$ are given by}
\begin{align*}
\mathcal{R}_0 & :=n\,+\!\sum_{x_j\in L_{0,1}}(k_j-n), \\
\mathcal{P}_{i,\red{\ell_i}} &:=p_{i,\red{\ell_i}}\,+\!\sum_{x_j\in L_{i,\red{\ell_i}}}(k_j-n),\quad \!\!\!\!\mbox{for $i\in\{1,\ldots,m\}$ and $\hspace{.3mm}\red{\ell_i}\in\{1,\dots, \red{r_{i}}\}$},\quad\!\!\!\!\mbox{and}\\
\mathcal{Q}_{i,\red{\ell_i}} &:=q_{i,\red{\ell_i}}\,+\!\!\!\sum_{x_j\in L_{i,\red{\ell_i+r_{i}}}}\!\!\!\!(k_j-n),\quad \!\!\!\!\mbox{for $i\in\{1,\ldots,m\}$ and $\hspace{.3mm}\red{\ell_i}\in\{1,\dots, \red{s_{i}}\}$.}
\end{align*}
\red{If $\mathcal{R}_0\geq0$, $\mathcal{P}_i\geq0$ and $\mathcal{Q}_i\geq0$ for all $i=1,\ldots,m$, we say that the factors in~$(\ref{aproduct})$ interact weakly and, if in addition $\mathcal{P}_i>0$ for some $i$, we say that the factors in~$(\ref{aproduct})$ interact strongly. Otherwise, we say that the factors in~$(\ref{aproduct})$ do not interact.}
\end{enumerate}
\end{definition}

\red{Although not reflected in the notation, pruned trees and leaves depend on the essential vertices $x_i$, while interaction parameters depend on the complete information encoded by the factors in~(\ref{aproduct}). Latter in the paper we will need to use pruned trees, their pruned leaves, as well as interaction parameters of subproducts of~(\ref{aproduct}). In such a case, we will use a notation of the type $T_{i,\ell_i}(x_1,\ldots,x_m)$, $L_{i,\ell_i}(x_1,\ldots,x_m)$, $\mathcal{R}_0(x_1,\ldots,x_m)$, $\mathcal{P}_{i,\ell_i}(x_1,\ldots,x_m)$, $\mathcal{Q}_{i,\ell_i}(x_1,\ldots,x_m)$, as well as $\mathcal{P}_i(x_1,\ldots,x_m)$ and $\mathcal{Q}_i(x_1,\ldots,x_m)$ in order to clarify the factors under consideration.}

\medskip
Next we adapt the expression in (\ref{representingcohomologyclass}) for usage \red{within} the $T_{i,\red{\ell_i}}$-notation. \red{In terms of} the cocycle representative
\begin{equation}\label{core}
\red{\sum\!\left(U_i,V_i,(x_i,\overline{x}_i),W_i\rule{0mm}{4mm}\right)\!\cdot\!\sigma}:=\!\sum\!\left(u_1,\ldots,u_{k_i},v_1,\ldots,v_{|p_{i}|},(x_i,\red{\overline{x}_i}),w_1,\ldots,w_{|q_{i}|}\right)\!\cdot\sigma
\end{equation}

\vspace{.45mm}\noindent \red{in Proposition~\ref{underneathphi} for $\{k_i\l x_i,p_i,q_i\}$, (\ref{aproduct}) is represented by the sum of all possible products
\begin{equation}\label{imenorquej}
\cdots\left((U_i,V_i,(x_i,\overline{x}_i),W_i)\cdot\sigma_i\rule{0mm}{4mm}\right)\cdots\left((U_j,V_j,(x_j,\overline{x}_j),W_j)\cdot\sigma_j\rule{0mm}{4mm}\right)\cdots.
\end{equation}
A number of vanishing such products can be ruled out as follows. Fix integers $1\leq i<j\leq m$. Proposition~\ref{productsinTn} implies that, if a product~(\ref{imenorquej}) is non-zero, then $(U_i,V_i,(x_i,\overline{x}_i),W_i)$ must have $x_j$, but cannot have $\overline{x}_j$, as one of its vertex ingredients. Likewise, $(U_j,V_j,(x_j,\overline{x}_j),W_j)$ must have $\overline{x}_i$, but cannot have $x_i$, as one of its vertex ingredients. Actually, together with Remark~\ref{separticionaadecuadamente}, this shows that non-zero products~(\ref{imenorquej}) are best organized (and easily evaluated ---see below) by replacing each} $\Sigma_n$-representative
\begin{equation}\label{elrepresentative}
(u_1,\ldots,u_{k_i},v_1,\ldots,v_{|p_{i}|},(x_i,\red{\overline{x}_i}),w_1,\ldots,w_{|q_{i}|})
\end{equation}
\red{in~(\ref{core})} by \red{the one} written in \red{a ``block''} form $(B^i_0,\red{B^i_1},\ldots,B^i_m)$. \red{Here each tuple of ingredients $B^i_j$ starts with the relevant $x_j$- or $\overline{x}_j$-information (if $j>0$), and continues with a repacking of the vertex} ingredients of (\ref{elrepresentative}) \red{that lie} in the trees $T_{j,\red{\ell}}$ \red{for all relevant $\ell$.} In detail, \red{for the $i$-th factor in~(\ref{aproduct}) and each of the corresponding summands in~(\ref{elrepresentative}), let}
\begin{itemize}
\item[{\em (a)}] $B_0^i:=B_{0,1}^i$ be the tuple of \red{vertex ingredients} of (\ref{elrepresentative}) \red{that lie} in $T_{0,1}$, written in $T$-order;
\item[{\em (b)}] $B_i^i:=\left((x_i,\red{\overline{x}_i}),B^i_{i,1}, \ldots, B^i_{i,\red{d_i}}\rule{0mm}{4mm}\right)$, where $B^i_{i,\red{\ell}}$ is the tuple of \red{vertex ingredients} of (\ref{elrepresentative}) \red{that lie} in $T_{i,\red{\ell}}$, written in $T$-order;
\item[{\em (c)}] \red{If $i<j$,} $B^i_j:=(\red{x_j}, B^i_{j,1},\ldots, B^i_{j,d_j})$, where $B^i_{j,\red{\ell}}$ is the tuple of \red{vertex ingredients} of (\ref{elrepresentative}) \red{that lie} in $T_{j,\red{\ell}}$, written in $T$-order;
\item[{\em (d)}] \red{If $j<i$, $B^i_j:=(\overline{x}_j, B^i_{j,1},\ldots, B^i_{j,d_j})$, where $B^i_{j,\ell}$ is the tuple of \red{vertex ingredients} of (\ref{elrepresentative}) \red{that lie} in $T_{j,\ell}$, written in $T$-order.}
\end{itemize}
\red{Thus, summands in (\ref{core}) that have a chance to contribute with non-vanishing products~(\ref{imenorquej}) to a cocycle representative of~(\ref{aproduct}) can be written as}
\begin{align*}
&\left(\rule{0mm}{5mm}\red{B^i_{0,1}}\ll \cdots \ll
\overline{x}_{i'},B^i_{i',1}, \dots ,B^i_{i',\red{d_{i'}}}\ll\cdots\ll
(x_i,\overline{x}_i),B^i_{i,1}, \dots, B^i_{i,d_i}\ll\ldots\ll
{x}_{i''},B^i_{i'',1},\dots, B^i_{i'',\red{d_{i''}}}\ll \cdots
\right)\cdot\sigma,
\end{align*}
\red{where vertical bars are used interchangeably by commas, and are intended to make reading easier. Proposition~\ref{productsinTn} then implies that a product~(\ref{imenorquej}), written as}
\begin{align*}
\left(\rule{-1.1mm}{5mm}\left(\red{B^1_{0,1}}\ll(x_1,\overline{x}_1),B^1_{1,1},\dots, B^1_{1,d_1}\ll{\cdots}\ll \right.\right.&\left.\left.\hspace{-2mm}x_m,B^1_{m,1},\dots, B^1_{m,\red{d_m}}\right)\hspace{-.8mm}\cdot\hspace{-.5mm}\sigma_1\rule{0mm}{5mm}\right)\cdots \\
\cdots&\left(\rule{-1.1mm}{5mm}\left(\red{B^m_{0,1}}\ll\overline{x}_1,B^m_{1,1},\dots, B^m_{1,\red{d_1}}\ll{\cdots}\ll (x_m,\overline{x}_m),B^m_{m,1},\dots, B^m_{m,d_m}\right)\hspace{-.8mm}\cdot\hspace{-.5mm}\sigma_m\right)\!,
\end{align*}
is non-zero if and only if $\sigma_i=\sigma_j=:\sigma$ and $B^i_{\red{t,\ell}}=B^j_{\red{t,\ell}}=: B_{\red{t,\ell}}$ for all \red{relevant} $i,j,\red{t,\ell}$, in which case~(\ref{imenorquej}) becomes
\begin{equation}\label{cico}
\sign(\widetilde{\sigma})
\left(\red{B_{0,1}}\l(x_1,\overline{x}_1),B_{1,1},\cdots, B_{1,d_1}\l\cdots\l (x_m,\overline{x}_m),B_{m,1},\cdots,B_{m,d_m}\rule{0mm}{4mm}\right)\!\cdot\sigma,
\end{equation}
where $\widetilde{\sigma}$ is the permutation determined by the sequence of positions of the edges $(x_1,\overline{x}_1),\ldots,(x_m,\overline{x}_m)$ in \red{the} tuple $(\red{B_{0,1}}\l(x_1,\overline{x}_1),B_{1,1},\dots, B_{1,d_1}\l\cdots\l (x_m,\overline{x}_m),B_{m,1},\dots,B_{m,d_m})\!\cdot\sigma$. \red{Note that} the cube \red{in~(\ref{cico}) is} product-oriented (as required by Proposition~\ref{productsinTn}), and that \red{(\ref{cico}) agrees} with the gradient-oriented cube
$$
\left(\rule{0mm}{4mm}\red{B_{0,1}}\l(x_1,\overline{x}_1),B_{1,1},\cdots,B_{1,d_1}\l\cdots\l (x_m,\overline{x}_m),B_{m,1},\cdots,B_{m,d_m}\right)\cdot\sigma,
$$
since $x_1<\cdots<x_m$. \red{This proves the first half of the next generalization of Proposition~\ref{underneathphi}:}

\begin{proposition}\label{interactionlengths}
\red{The product~$(\ref{aproduct})$} is represented in $C^*(\DnT)$ by the gradient-oriented cocycle 
\begin{equation}\label{gradientrepresentative}
\sum\!\left(\red{B_{0,1}}\ll(x_1,\overline{x}_1),B_{1,1},\ldots,B_{1,d_1}\ll\cdots\ll (x_m,\overline{x}_m),B_{m,1},\ldots,B_{m,d_m}\rule{0mm}{4mm}\right)\!\cdot\sigma,
\end{equation}
where the summation runs over all permutations $\sigma\in\Sigma_n$ and all possible tuples $B_{\red{t,\ell}}$ of vertices written in $T$-order, taken from the corresponding \red{pruned} trees $T_{\red{t,\ell}}$, \red{and having the following lengths:} Any block $\red{B_{0,1}}$ \red{must have} $\mathcal{R}_0$ ingredients, while any block $B_{\red{t,\ell}}$ \red{with $t>0$ must have} $\mathcal{P}_{\red{t,\ell}}$ ingredients for $1 \leq \red{\ell} \leq \red{r_t}$, and $\mathcal{Q}_{\red{t,\ell-r_t}}$ ingredients for $\red{r_t}< \red{\ell} \leq \red{d_t}$. In particular, $(\ref{aproduct})$ vanishes provided \red{its factors do not interact.}
\end{proposition}

Note that $\mathcal{R}_0+\sum_{i,\red{\ell}}\mathcal{P}_{i,\red{\ell}}+\sum_{i,\red{\ell}}\mathcal{Q}_{i,\red{\ell}}=n-m$ in Definition \ref{interactionparameters}. This is compatible with the fact that cubes in (\ref{gradientrepresentative}), \red{if any,} have $n$ ingredients. See also Corollary \ref{interactionproducts} below. 

\begin{proof}
\red{It remains to prove the assertions about the sizes of blocks $B_{t,\ell}$, and that all possible such blocks appear in~(\ref{gradientrepresentative}). As for the sizes, proceeding by} induction on $m$ (with Proposition~\ref{underneathphi} grounding the argument), it suffices to consider a product \red{$\pi_1\cdot\pi_2$ with}
\begin{equation}\label{aditionalfactor}\begin{split}
\pi_1&=\left[\rule{0mm}{5mm}\left(\red{B_{0,1}}\ll(x_1,\overline{x}_1),B_{1,1},\ldots,B_{1,d_1}\ll\cdots\ll(x_m,\overline{x}_m),B_{m,1},\ldots,B_{m,d_m})\right)\cdot\sigma\,\right],\\\pi_2&=\left[\rule{0mm}{5mm}\left(U\ll(x_{m+1},\overline{x}_{m+1}),V_1,\ldots,V_{d_{m+1}}\right)\cdot\sigma'\,\right],
\end{split}\end{equation}
where $x_1<\cdots <x_m<x_{m+1}$, and \red{where the structure} of the blocks $B_{\red{t,\ell}}$ is as specified in the proposition. \red{Here we are assuming \red{({\em a})} that $U$ is a tuple of} $k_{m+1}$ \red{vertex} ingredients \red{written in $T$-order and lying in $x_{m+1}$-direction 0, ({\em b}) that any tuple} $V_{\red{\ell}}$ \red{with $1 \leq \red{\ell} \leq \red{r_{m+1}}$ consists of} $p_{m+1,\red{\ell}}$ \red{vertex} ingredients \red{written in $T$-order and lying in $x_{m+1}$-direction~$\ell$, and ({\em c}) that any tuple} $V_{\red{\ell+r_{m+1}}}$ \red{\!with $1 \leq \red{\ell} \leq s_{m+1}$ consists of} $q_{m+1,\red{\ell}}$ \red{vertex} ingredients \red{written in $T$-order and lying in $x_{m+1}$-direction $\ell+r_{m+1}$. In addition, we make the conventions $d_{m+1}:=d(x_{m+1})-1$ and $\overline{x}_{m+1}:=x_{m+1}[r_{m+1}+1]$, and assume the relations $d_{m+1}=r_{m+1}+s_{m+1}$, $r_{m+1}\geq1\leq s_{m+1}$ and $k_{m+1}+\sum_{\ell=1}^{r_{m+1}} p_{m+1,\ell}+\sum_{\ell=1}^{s_{m+1}} q_{m+1,\ell}=n-1$. Furthermore,} signs and orientations \red{will be} ignored in the rest of the proof, \red{as they} have been carefully \red{addressed} in the discussion \red{previous to this proposition.} In particular, we can safely \red{work at the unordered-cube level, thus ignoring the} permutations $\sigma$ and $\sigma'$ in (\ref{aditionalfactor}) \red{and, instead, thinking of tuples of ingredients as sets of ingredients.}

Consider the pruned trees $T_{\red{t,\ell}}:=\red{T_{t,\ell}(x_1,\ldots,x_m)}$ and $T'_{\red{t,\ell}}:=\red{T_{t,\ell}(x_1,\ldots,x_{m+1})}$, as well as the pruned leaves $L_{\red{t,\ell}}:=\red{L_{t,\ell}(x_1,\ldots,x_m)}$ and $L'_{\red{t,\ell}}:=L_{t,\ell}(x_1,\ldots,x_{m+1})$. There are three cases, depending on whether the edge $(x_{m+1},\overline{x}_{m+1})$ belongs to $T_{0,1}$, \red{or to} $T_{\red{t,\ell}}$ with \red{$1\leq t\leq m$ and} $1 \leq \red{\ell} \leq \red{r_t}$, or \red{to} $T_{\red{t,\ell}}$ with \red{$1\leq t\leq m$ and} $\red{r_t}<\red{\ell\leq d_t}$, and the argument is virtually identical in each. We consider only the situation depicted in Figure~\ref{caso2}, where \red{the edge} $(x_{m+1},\overline{x}_{m+1})$ belongs to $T_{\red{t,\ell}}$ \red{for some $t\in\{1,2,\ldots,m\}$ and some $\ell\in\{1,2,\ldots,r_t\}$. In such a} case we have
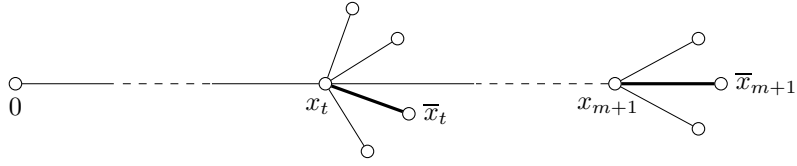
\begin{figure}
	\centering
	\begin{tikzpicture}[scale=1.0, rotate = 180, xscale = -1]
		\node[style={circle, draw, scale=.5}] (1) at ( 1.62, 3.2) {};
		\node at ( 1.62, 3.5) {$0$};
		\node[scale=.1] (2) at ( 2.94, 3.2) {};
		\node[scale=.1] (3) at ( 4.22, 3.2) {};
		\node[style={circle, draw, scale=.5}] (4) at ( 5.74, 3.2) {};
		\node at (5.64, 3.5) {$\red{x_t}$};
		\node[style={circle, draw, scale=.5}] (5) at ( 6.1, 2.2) {};
		\node[style={circle, draw, scale=.5}] (6) at ( 6.7, 2.6) {};
		\node[scale=.1] (7) at ( 7.73, 3.2) {};
		\node[style={circle, draw, scale=.5}] (8) at ( 6.85, 3.6) {};
		\node[style={circle, draw, scale=.5}] (9) at ( 6.3, 4.1) {};
		\node[style={circle, draw, scale=.5}] (10) at ( 9.59, 3.2) {};
		\node at ( 9.5, 3.5) {$x_{m+1}$};
		\node[style={circle, draw, scale=.5}] (11) at ( 10.7, 2.6) {};
		\node[style={circle, draw, scale=.5}] (12) at ( 11, 3.2) {};
		\node[style={circle, draw, scale=.5}] (13) at ( 10.7, 3.8) {};
		\draw[very thin] (2) -- (1);
		\draw[dashed, very thin] (3) -- (2);
		\draw[very thin] (4) -- (3);
		\draw[very thin] (7) -- (4);
		\draw[dashed, very thin] (10) -- (7);
		\draw[very thick] (12)node[right]{\red{$\,\overline{x}_{m+1}$}} -- (10);
		\draw[very thin] (11) -- (10);
		\draw[very thin] (13) -- (10);
		\draw[very thin] (5) -- (4);
		\draw[very thin] (6) -- (4);
		\draw[very thick] (8)node[right]{\red{\,$\overline{x}_t$}} -- (4);
		\draw[very thin] (9) -- (4);
	\end{tikzpicture}
	\caption{The edge $(x_{m+1},\overline{x}_{m+1})$ belongs to $T_{\red{t},3}$, so \red{the path from $x_t$ to $x_{m+1}$ does not pass through an essential vertex $x_j$}}
\label{caso2}
\end{figure}
\begin{itemize}
\item[(i)] $T_{\red{\tau},\red{\lambda}}=T'_{\red{\tau},\red{\lambda}}$ and $L_{\red{\tau},\red{\lambda}}=L'_{\red{\tau},\red{\lambda}}$, \red{for $1\leq \tau\leq m$ as long as $\tau\neq t$ or $\lambda\neq\ell$;}
\item[(ii)] $T_{\red{t},\red{\ell}}\setminus\text{Int}(x_{m+1},\overline{x}_{m+1})=T'_{\red{t},\red{\ell}}\,\bigsqcup \left( \bigcup\limits_{\red{\lambda}=1}^{d_{m+1}} T'_{m+1,\red{\lambda}}\right)$;
\item[(iii)] $L'_{\red{t},\red{\ell}}=L_{\red{t},\red{\ell}}\cup\{x_{m+1}\}$ and $L'_{m+1,\red{\lambda}}=\varnothing$ for $\red{\lambda} \in \{1, \dots , d_{m+1}\}$.
\end{itemize}

\medskip
By Proposition~\ref{productsinTn}, \red{the product $\pi_1\cdot\pi_2$ of the elements in}~(\ref{aditionalfactor}) vanishes unless
$$
\left(\{\overline{x}_1,\ldots,\overline{x}_m\}\sqcup \red{B_{0,1}}\sqcup\left(\,\bigsqcup_{\substack{\red{1\leq\tau\leq m}\\\red{1\leq\lambda\leq d_\tau}}} \!\!B_{\red{\tau},\red{\lambda}}\right) \!\right) \setminus  B_{\red{t},\red{\ell}}\subseteq U, \qquad
\{x_{m+1}\}\sqcup \left( \bigsqcup_{\red{\lambda=1}}^{d_{m+1}} V_{\red{\lambda}} \right) \subseteq B_{\red{t},\red{\ell}}
$$
and
$$
U\setminus\left(\!\left(\{\overline{x}_1,\ldots,\overline{x}_m\}\sqcup \red{B_{0,1}}\sqcup\left(\,\bigsqcup_{\substack{\red{1\leq\tau\leq m}\\\red{1\leq\lambda\leq d_\tau}}} \!\!B_{\red{\tau},\red{\lambda}}\right) \!\right) \setminus B_{\red{t},\red{\ell}} \right) = B_{\red{t},\red{\ell}}\setminus \left(\{x_{m+1}\}\sqcup \left( \bigsqcup_{\red{\lambda=1}}^{d_{m+1}} V_{\red{\lambda}} \right)\!\right)=:B'_{t,\ell},
$$
in which case $$\red{\pi_1\cdot\pi_2=(B_{0,1}\l(x_1,\overline{x}_1),\mathcal{B}_1\l\cdots\l(x_m,\overline{x}_m),\mathcal{B}_m\l (x_{m+1},\overline{x}_{m+1}),V_1,\ldots,V_{d_{m+1}}),}$$ where $\mathcal{B}_\tau$ is a shorthand for the sequence $B_{\tau,1},\ldots,B_{\tau,d_{\tau}}$ provided $\tau\neq t$, whereas $\mathcal{B}_t$ stands for the sequence $$B_{t,1},\ldots,B_{t,\ell-1},B'_{t,\ell},B_{t,\ell+1},\ldots,B_{t,d_t}.$$ The induction is complete \red{in view of items (i)--(iii) above and}
\begin{align*}
\l B'_{\red{t},\red{\ell}} \l&=\l B_{\red{t},\red{\ell}} \l-\left(1+\sum_{\red{\lambda}=1}^{\red{r_{m+1}}}p_{m+1,\red{\lambda}}+\sum_{\red{\lambda}=1}^{\red{s_{m+1}}}q_{m+1,\red{\lambda}}\right)\\&=\hspace{.8mm}p_{\red{t},\red{\ell}}\hspace{.5mm}+\!\!\sum_{x_{\red{\lambda}}\in L_{\red{t},\red{\ell}}}\!\!(k_{\red{\lambda}}-n)\hspace{.3mm}-\hspace{.3mm}(n-k_{m+1})\hspace{.8mm}=\hspace{.8mm}p_{\red{t},\red{\ell}}\hspace{.5mm}+\!\!\sum_{x_{\red{\lambda}}\in L'_{\red{t},\red{\ell}}}\!\!(k_{\red{\lambda}}-n),
\end{align*}
\red{which shows that $B'_{\red{t},\red{\ell}}$ has the prescribed cardinality.}
\red{The inductive analysis makes it clear also that all blocks $B_{t,\ell}$ with the structure indicated in the proposition indeed appear in~(\ref{gradientrepresentative}).}
\end{proof}

\begin{corollary}\label{interactionproducts}
The product~$(\ref{aproduct})$ agrees with the basis element $\{\mathcal{R}_0\l x_1,\mathcal{P}_{1},\mathcal{Q}_{1}\l\cdots\l x_m,\mathcal{P}_{m},\mathcal{Q}_{m}\}$ provided the factors of~$(\ref{aproduct})$ interact strongly.
\red{Recall} $\mathcal{P}_i= (\mathcal{P}_{i,1}, \mathcal{P}_{i,2}, \dots , \mathcal{P}_{i,\red{r_i}})$ and $\mathcal{Q}_i= (\mathcal{Q}_{i,1}, \mathcal{Q}_{i,2}, \dots , \mathcal{Q}_{i,\red{s_i}})$.
\end{corollary}

\begin{proof}
\red{By the strong interaction hypothesis, a summand in (\ref{gradientrepresentative}) that is the target of a lower gradient path $\gamma$ must actually be critical (and $\gamma$ must be constant) with ingredients equal to those associated to $\{\mathcal{R}_0\l x_1,\mathcal{P}_{1},\mathcal{Q}_{1}\l\cdots\l x_m,\mathcal{P}_{m},\mathcal{Q}_{m}\}$.} The conclusion then follows from (\ref{quasi-isomorphisms}) and (\ref{simpleform}).
\end{proof}

\begin{lemma}\label{systemofequations}
\red{Fix} essential vertices $x_1<\cdots<x_m$ \red{and take positive integer numbers $r_i$ and $s_i$ with $r_i+s_i=d(x_i)-1$ for $1\leq i\leq m$. Let} $R_0$, $P_{\red{i,\ell}}$, $Q_{\red{i,k}}$, \red{with} $ 1 \leq \red{i} \leq m$, $1 \leq \red{\ell} \leq r_{\red{i}}$ \red{and} $ 1 \leq \red{k} \leq s_{\red{i}}$,  be non-negative integers satisfying $n-\red{m}=R_0+\sum_{i=1}^m\hspace{-.2mm}\left(\hspace{.4mm}\sum_{\red{\ell}=1}^{r_{\red{i}}}P_{i,\red{\ell}}+\sum_{k=1}^{s_{\red{i}}} Q_{i,k}\hspace{.4mm}\right)$. Then the system
\begin{align*}
n\,+\!\sum_{x_j\in L_{0,1}}(k_j-n)&=R_0,\\
p_{i,\red{\ell}}\,+\!\sum_{x_j\in L_{i,\red{\ell}}}(k_j-n)&=P_{i,\red{\ell}}\quad (i=1,\ldots,m, \;\;\red{\ell}=1,\dots, r_{\red{i}}), \\
q_{i,\red{k}}\,+\hspace{-2mm}\sum_{x_j\in L_{i,\red{k+r_i}}}\hspace{-2mm}(k_j-n)&=Q_{i,\red{k}}\quad (i=1,\ldots,m, \;\;\red{k}=1,\dots, s_{\red{i}}),
\end{align*}
has a unique solution of non-negative integer numbers $\{k_i,p_{i,1}, \dots ,p_{i, r_{\red{i}}}, q_{i,1}, \dots , q_{i, s_{\red{i}}} \}_{i=1}^m$  satisfying the condition $\red{n-1}=k_i+\sum_{\red{\ell}=1}^{r_{\red{i}}} p_{i,\red{\ell}}+ \sum_{\red{k}=1}^{s_{\red{i}}}q_{i,\red{k}}$ \red{\,for each $i\in\{1,\ldots,m\}$. If, in addition,} for each $\red{i\in\{1,\ldots,m\}}$ there exists \red{$\ell\in\{1,\ldots,r_i\}$} with $ P_{\red{i,\ell}}>0$, then \red{the unique solution satisfies that, for each $i\in\{1,\dots,m\}$, there exists} $\ell\in\{1, \dots r_{\red{i}}\}$ \red{with} $p_{\red{i,\ell}}>0$.
\end{lemma}
\begin{proof}
The two sets of equations with $i=m$ reduce to $p_{m,\red{\ell}}=P_{m,\red{\ell}}$ ($\red{\ell}=1,2, \dots, r_{\red{m}}$) and $q_{m,\red{k}}=Q_{m,\red{k}}$ ($\red{k}=1,2, \dots, s_{\red{m}}$). \red{This also determines}
\begin{equation}\label{determines}
k_m:=n-\sum_{\red{\ell}=1}^{r_{\red{m}}}P_{m,\red{\ell}}-\sum_{\red{k}=1}^{s_{\red{m}}}Q_{m,\red{k}}-1=R_0+\sum_{j=1}^{m-1}\left(\sum_{\red{\ell}=1}^{r_{\red{j}}}P_{j,\red{\ell}}+\sum_{k=1}^{s_{\red{j}}}Q_{j,k}+1\right)\geq0.
\end{equation}
The rest of the equations can be written as
\begin{align*}
n\,+\!\!\sum_{x_j\in L_{0,1}\setminus\{x_m\}}\!\!\!(k_j-n)&=R'_0:=R_0\red{{}+{}}\!\left.\begin{cases} n-k_m, & \mbox{if $x_m\in L_{0,1}$} \\ 0, & \mbox{otherwise} \end{cases}\right\},\\
p_{i,\red{\ell}}\,+\!\!\sum_{x_j\in L_{i,\red{\ell}}\setminus\{x_m\}}\!\!\!(k_j-n)&=P'_{i,\red{\ell}}:=P_{i,\red{\ell}}\red{{}+{}}\!\left.\begin{cases} n-k_m, & \mbox{if $x_m\in L_{i,\red{\ell}}$} \\ 0, & \mbox{otherwise} \end{cases}\right\},\\
q_{i,\red{k}}\,+\!\!\!\!\!\!\sum_{x_j\in L_{i,\red{k}+r_{\red{i}}}\setminus\{x_m\}}\!\!\!\!\!\!\!(k_j-n)&=Q'_{i,\red{k}}:=Q_{i,\red{k}}\red{{}+{}}\!\left.\begin{cases} n-k_m, & \mbox{if $x_m\in L_{i,\red{k}+r_{\red{i}}}$} \\ 0, & \mbox{otherwise} \end{cases}\right\},
\end{align*}
\red{for} $\,i=1,\ldots,m-1$, $\,\red{\ell}=1,\dots, r_{\red{i}}\,$ \red{and} $\,\red{k}=1,\dots, s_{\red{i}}$. \red{The result then follows by induction since}
$$
R'_0+\sum_{j=1}^{m-1}\left(\sum_{\red{\ell}=1}^{r_{\red{j}}}P'_{j,\red{\ell}}+\sum_{k=1}^{s_{\red{j}}}Q'_{j,k}+1\right)=R_0+\sum_{j=1}^{m-1}\left(\sum_{\red{\ell}=1}^{r_{\red{j}}}P_{j,\red{\ell}}+\sum_{k=1}^{s_{\red{j}}}Q_{j,k}+1\right)\!\red{{}+{}}n-\red{k}_m $$ $$
=R_0+\sum_{j=1}^{m}\left(\sum_{\red{\ell}=1}^{r_{\red{j}}}P_{j,\red{\ell}}+\sum_{k=1}^{s_{\red{j}}}Q_{j,k}+1\right)=n,
$$
\red{where the second equality uses~(\ref{determines}).}
\end{proof}

\begin{proof}[Proof of Theorem~\ref{casogeneral}]
Corollary~\ref{interactionproducts} and Lemma~\ref{systemofequations} \red{yield} a set theoretic identification \red{$\mathcal{S}_m=\mathcal{B}_m$, where $\mathcal{S}_m$ is the set of products~(\ref{aproduct}) whose factors interact strongly, and $\mathcal{B}_m$ is} the $m$-dimensional basis of $\image$ with basis elements $\{R_0\l x_1,\red{(P_{1,1},\ldots,P_{1,r_1})},\red{(Q_{1,1},\ldots,Q_{1,s_1})}\l\cdots\l x_m,\red{(P_{m,1},\ldots,P_{m,r_m})},\red{(Q_{m,1},\ldots,Q_{m,s_m})}\}$. Together with Corollary~\ref{trivialsquares} and Proposition~\ref{interactionlengths}, this completes the proof, where $\langle k,x,p,q\rangle\in V_nT$ is identified with (the $\pi^*$-preimage of) $\{k\l x,p,q\}\in\image$.
\end{proof}

Note that the cohomology ring $H^*(\UDnT)$ is generated by 1-dimensional classes, a fact already known from~\cite{MR2355034}. It is not true that a product \red{(\ref{aproduct})} vanishes when its factors interact but non-strongly. The description of such products relies on the dynamics of lower gradient paths.

\section{Cup products II: Lower gradient paths}\label{lgp}
\red{Let $\Pi_1$ stand for a product~$(\ref{aproduct})$ whose factors interact strongly, so Corollary~\ref{interactionproducts} applies. Choose an additional 1-dimensional basis element $\{k_x\l x,(p_{x,1},\ldots,p_{x,r_x}),(q_{x,1},\ldots,q_{x,s_x})\}$ of $\hspace{.4mm}\image$ with $x<x_1<\cdots<x_m$ and where the standard conditions and conventions are assumed, namely, 
\begin{equation}\label{namely}
p_x:=(p_{x,1},\ldots,p_{x,r_x})>0 \ \ \mbox{and} \ \  q_x:=(q_{x,1},\ldots,q_{x,s_x})\geq0,
\end{equation}
where $r_x\geq1\leq s_x$, $r_x+s_x=d_x:=d(x)-1$, $|p_x|:=\sum_{\ell=1}^{r_x}p_{x,\ell}$, $|q_x|:=\sum_{\ell=1}^{s_x}q_{x,\ell}$, $k_x+|p_x|+|q_x|=n-1$ and $\overline{x}:=x[r_x+1]$. Consider the interaction parameters $P_{i}:=\mathcal{P}_{i}(x_1,\ldots,x_m)$ and $Q_{i}:=\mathcal{Q}_{i}(x_1,\ldots,x_m)$ of the factors of $\Pi_1$ ($i\in\{1,\ldots, m\}$), as well as the first three interaction parameters $R_0:=\mathcal{R}_0(x,x_1,\ldots,x_m)$, $P_x:=\mathcal{P}_1(x,x_1,\ldots,x_m)$ and $Q_x:=\mathcal{Q}_1(x,x_1,\ldots,x_m)$ of the factors of $\Pi_2:=\{k_x\l x,p_x,q_x\}\cdot\Pi_1$. This section is devoted to proving:}
\begin{theorem}\label{productsvialowerpaths}
In the situation above, if the factors of $\Pi_2$ interact but non-strongly, then
\begin{align}
\Pi_2&=-\sum_{\red{a}}\left\{R_0-|\red{a}|\ll x,\red{a},Q_x\ll x_1, P_{1},Q_{1}\ll\cdots\ll x_m, P_{m},Q_{m}\right\}\label{laquefaltaba}\\ 
&\;\;\;\;{}+\sum_{\ell=1}^{s_x-1}\sum_{a,b}\left\{R_0-|a|-b-1\ll x,Q_x^{(\ell,a,b)},Q_x^{(\ell,+)}\ll x_1, P_{1},Q_{1}\ll\cdots\ll x_m, P_{m},Q_{m}\right\}\label{dosmas}\\
&\;\;\;\;{}-\sum_{\ell=1}^{s_x-1}\sum_{a,b}\left\{R_0-|a|-b\ll x,Q_x^{(\ell,a,b)},Q_x^{(\ell,-)}\ll x_1, P_{1},Q_{1}\ll\cdots\ll x_m, P_{m},Q_{m}\right\}\label{unomas}.
\end{align}
In the above expression we set $a:=(a_1,\ldots,a_{r_x})$, $|a|:=a_1+\cdots+a_{r_x}$,
$Q_x^{(\ell,+)}:=(Q_{x,\ell+1},Q_{x,\ell+2},\ldots,Q_{x,s_x})$,
$Q_x^{(\ell,-)}:=(Q_{x,\ell+1}-1,Q_{x,\ell+2},\ldots,Q_{x,s_x})$ and
$Q_{x}^{(\ell,a,b)}:=(a_1,\ldots,a_{r_x},Q_{x,1}+b+1,Q_{x,2},\ldots,Q_{x,\ell})$.
The summation in~(\ref{laquefaltaba}) runs over all $r_x$-tupes $a$ of non-negative integer numbers satisfying $1\leq|\red{a}|\leq R_0$. The inner summation in~(\ref{dosmas}) runs over all $r_x$-tupes $a$ of non-negative integer numbers and all non-negative integer numbers~$b$ satisfying $|a|+b<R_0$. The inner summation in~(\ref{unomas}) is empty if $Q_{x,\ell+1}=0$, otherwise it runs over all $r_x$-tupes $a$ of non-negative integer numbers and all non-negative integer numbers $b$ satisfying $|a|+b\leq R_0$.
\end{theorem}

Since summands in \red{(\ref{laquefaltaba})--(\ref{unomas})} are basis elements, Theorem~\ref{productsvialowerpaths} and the results in the previous section give a \red{recursive method} to effectively asses cup-products in $\image\cong H^*(B_nT)$.

\begin{proof}[Proof of Theorem~\ref{productsvialowerpaths} (preparation)]
We have seen that $\Pi_2$ is represented in $C^*(\DnT)$ by the gradient-oriented cocycle
\begin{equation}\label{elmeromero}
\sum\left(\red{B_{0,1}}\ll(x,\overline{x}),\red{B_{x,1},\ldots,B_{x,d_x}}\ll(x_1,\overline{x}_1),B_{1,1},\ldots,B_{1,d_1}\ll\cdots\ll(x_m,\overline{x}_m),B_{m,1},\ldots,B_{m,d_m}\right)\cdot\sigma,
\end{equation}
where the summation runs over all permutations $\sigma\in\Sigma_n$ and over all possible blocks $B_{\ast,\ast}$ of vertices written in $T$-order, taken from the corresponding trees $T_{\ast,\ast}$ determined by the \red{factors of $\Pi_2$,} and having sizes as prescribed in Proposition~\ref{interactionlengths} in terms of the relevant interaction parameters. \red{The goal now is to identify the $\underline{\Phi}$-image of~(\ref{elmeromero}) which, by~(\ref{quasi-isomorphisms}), is the element} in $\mathcal{M}^*(\DnT)$
\begin{equation}\label{pi}
\sum_{\gamma\in\mathcal{G}}\mu(\gamma)\cdot\mathcal{S}_\gamma.
\end{equation}
Here $\mathcal{G}$ is the set of lower paths $\gamma$ starting at an $(m+1)$-critical cube $\mathcal{S}_\gamma$ and finishing at a summand of~(\ref{elmeromero}). \red{We start by identifying (in the next two paragraphs) key characteristics of ending cubes for paths in~$\mathcal{G}$.}

\red{Firstly,} the condition $x<x_1$ forces one of the four configurations \red{depicted} in Figure~\ref{four}. In any \red{of those configurations,} vertices $x_i$ with $i>1$ lie either on \red{a} component of $T\setminus\{x_1\}$ in \red{positive} $x_1$-direction or, else, ``below'' the horizontal segment joining the root and $x_1$. \red{As a result, the equalities $P_{i}=\mathcal{P}_{i}(x_1,\ldots,x_m)=\mathcal{P}_{i+1}(x,x_1,\ldots,x_m)$ and $Q_{i}=\mathcal{Q}_{i}(x_1,\ldots,x_m)=\mathcal{Q}_{i+1}(x,x_1,\ldots,x_m)$ hold for $i=1,\ldots ,m$. The interaction hypotheses then yield
\begin{equation}\label{px0}
P_x=0,
\end{equation}
the $r_x$-tuple consisting of zeros. This and~(\ref{namely}) rule out} the two configurations on the right of Figure~\ref{four}, as well as \red{the one on the} bottom left, \red{since the equality $P_x=p_x$ is forced for those configurations.} The only possible configuration, \red{i.e., the one on the top left of Figure~\ref{four}, will be assumed in the rest of the section.}

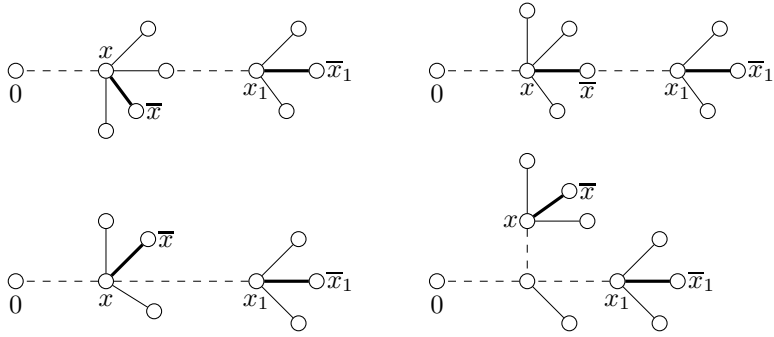
\begin{figure}
	\centering
	\begin{tikzpicture}[scale=.8, rotate = 180, xscale = -1]
		\node[style={circle, draw, scale=.6}] (1) at ( -4, 1.0) {};
		\node[style={circle, draw, scale=.6}] (2) at ( -2.5, 1) {};
		\node[style={circle, draw, scale=.6}] (3) at ( -2.5, 2) {};
		\node[style={circle, draw, scale=.6}] (4) at (-1.8,0.3) {};
		\node[style={circle, draw, scale=.6}] (5) at ( -1.5, 1) {};
		\node[style={circle, draw, scale=.6}] (6) at ( -2, 1.67) {};
		\node[style={circle, draw, scale=.6}] (7) at ( 0, 1) {};
		\node[style={circle, draw, scale=.6}] (8) at ( .7, .3) {};
		\node[style={circle, draw, scale=.6}] (9) at ( 1, 1) {};
		\node[style={circle, draw, scale=.6}] (10) at ( .5, 1.67) {};
		\node[style={circle, draw, scale=.6}] (11) at ( -4.0, 4.5) {};
		\node[style={circle, draw, scale=.6}] (12) at ( -2.5, 4.5) {};
		\node[style={circle, draw, scale=.6}] (50) at ( -1.7, 5) {};
		\node[style={circle, draw, scale=.6}] (13) at ( -2.5, 3.5) {};
		\node[style={circle, draw, scale=.6}] (14) at ( -1.8, 3.8) {};
		\node[style={circle, draw, scale=.6}] (15) at ( 0, 4.5) {};
		\node[style={circle, draw, scale=.6}] (16) at ( .7, 3.8) {};
		\node[style={circle, draw, scale=.6}] (17) at ( .7, 5.2) {};
		\node[style={circle, draw, scale=.6}] (18) at ( 1, 4.5) {};
		\node[style={circle, draw, scale=.6}] (22) at ( 3, 4.5) {};
		\node[style={circle, draw, scale=.6}] (23) at ( 4.5, 4.5) {};
		\node[style={circle, draw, scale=.6}] (24) at ( 4.5, 3.5) {};
		\node[style={circle, draw, scale=.6}] (25) at ( 5.2, 3) {};
		\node[style={circle, draw, scale=.6}] (26) at ( 4.5, 2.5) {};
		\node[style={circle, draw, scale=.6}] (27) at ( 5.5, 3.5) {};
		\node[style={circle, draw, scale=.6}] (28) at ( 6, 4.5) {};
		\node[style={circle, draw, scale=.6}] (29) at ( 6.7, 5.2) {};
		\node[style={circle, draw, scale=.6}] (30) at ( 7, 4.5) {};
		\node[style={circle, draw, scale=.6}] (31) at ( 6.7, 3.8) {};
		\node[style={circle, draw, scale=.6}] (32) at ( 5.2, 5.2) {};
		\node[style={circle, draw, scale=.6}] (41) at ( 3, 1.0) {};
		\node[style={circle, draw, scale=.6}] (42) at ( 4.5, 1) {};
		\node[style={circle, draw, scale=.6}] (43) at ( 4.5, 0) {};
		\node[style={circle, draw, scale=.6}] (44) at (5.2,0.3) {};
		\node[style={circle, draw, scale=.6}] (45) at ( 5.5, 1) {};
		\node[style={circle, draw, scale=.6}] (46) at ( 5, 1.67) {};
		\node[style={circle, draw, scale=.6}] (47) at ( 7, 1) {};
		\node[style={circle, draw, scale=.6}] (48) at ( 7.7, .3) {};
		\node[style={circle, draw, scale=.6}] (49) at ( 8, 1) {};
		\node[style={circle, draw, scale=.6}] (40) at ( 7.5, 1.67) {};
	       \draw[dashed, very thin] (42) -- (41)node[below]{\raisebox{-3mm}{$0$}};
		\draw[very thin] (43) -- (42)node[below]{\raisebox{-2mm}{$x$}};
		\draw[very thin] (44) -- (42);
		\draw[very thin] (46) -- (42);
		\draw[very thick] (45)node[below]{$\overline{x}$} -- (42);
		\draw[very thin, dashed] (47) -- (45);
		\draw[very thin] (48) -- (47)node[below]{\raisebox{-2mm}{$x_1$}};
		\draw[very thick] (49)node[right]{$\overline{x}_1$} -- (47);
		\draw[very thin] (40) -- (47);
		\draw[dashed, very thin] (2) -- (1)node[below]{\raisebox{-3mm}{$0$}};
		\draw[very thin] (3) -- (2)node[above]{\raisebox{.5mm}{$x$}};
		\draw[very thin] (4) -- (2);
		\draw[very thin] (5) -- (2);
		\draw[very thick] (6)node[right]{$\overline{x}$} -- (2);
		\draw[dashed, very thin] (7) -- (5);
		\draw[very thin] (8) -- (7)node[below]{\raisebox{-2mm}{$x_1$}};
		\draw[very thick] (9)node[right]{$\overline{x}_1$} -- (7);
		\draw[very thin] (10) -- (7);
		\draw[dashed, very thin] (12) -- (11)node[below]{\raisebox{-3mm}{$0$}};
		\draw[dashed, very thin] (15) -- (12)node[below]{\raisebox{-2mm}{$x$}};
		\draw[very thin](13) -- (12);
		\draw[very thin](50) -- (12);
		\draw[very thick]  (14)node[right]{$\overline{x}$} -- (12);
		\draw[very thin] (16) -- (15)node[below]{\raisebox{-2mm}{$x_1$}};
		\draw[very thin] (17) -- (15);
		\draw[very thick] (18)node[right]{$\overline{x}_1$} -- (15);
		\draw[very thin, dashed] (23) -- (22)node[below]{\raisebox{-3mm}{$0$}};
		\draw[very thin, dashed] (24)node[left]{$x$} -- (23);
		\draw[very thin] (27) -- (24);
		\draw[very thick] (25)node[right]{$\overline{x}$} -- (24);
		\draw[very thin] (26) -- (24);
		\draw[dashed, very thin] (28) -- (23);
		\draw[very thin] (29) -- (28)node[below]{\raisebox{-2mm}{$x_1$}};
		\draw[very thin] (31) -- (28);
		\draw[very thick] (30)node[right]{$\overline{x}_1$} -- (28);
		\draw[very thin] (32) -- (23);
	\end{tikzpicture}
	\caption{The four possible configurations with $x<x_1$}
\label{four}
\end{figure}

\red{Secondly, redundant summands in~(\ref{elmeromero}) can be neglected, as none of those can be the destination of a lower path. Furthermore, (\ref{px0}) shows that no summand in~(\ref{elmeromero}) is critical. We thus focus on collapsible summands in~(\ref{elmeromero}) which (in addition to their size and distribution properties summarized at the start of the proof) are forced to satisfy the following two properties: For one,} ingredients of each \red{$B_{0,1}$ that are} smaller than~$\red{x}$ form a stack of vertices blocked by the root of $T$. \red{In addition,} for any \red{$i\in\{1,2,\ldots,m\}$ with} $\overline{x}_i$ smaller than $\overline{x}$, all ingredients of each $B_{i,\red{\ell}}$ \red{($1\leq\ell\leq d_i$)} are blocked (this uses \red{the fact that $P_{i}$ is not the zero tuple}), so \red{the} tuple $((x_i,\overline{x}_i),B_{i,1},\dots, B_{i,d_i})$ \red{assembles a (unique, by block-size limitations) critical situation around $x_i$.} It follows that \red{each} summand \red{$c\cdot\sigma$} in~(\ref{elmeromero}) \red{relevant for~(\ref{pi})} \red{is} collapsible by a branch-type pairing \red{that creates the edge $(x,\overline{x})$, as depicted in}
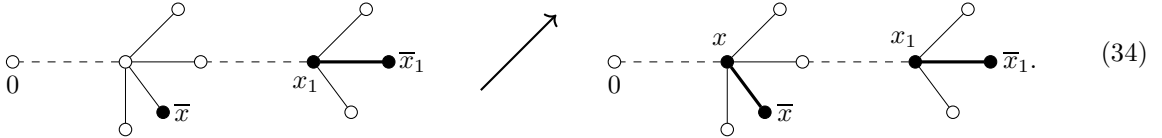
\begin{equation}\label{endingpairing}
\raisebox{-10mm}{
	\begin{tikzpicture}[scale=1.0, rotate = 180, xscale = -1]
		\node[style={circle, draw, scale=.5}] (1) at ( -5, 1.0) {};
		\node at ( -5, 1.3) {$0$};
		\node[style={circle, draw, scale=.5}] (2) at ( -3.5, 1) {};
		\node[style={circle, draw, scale=.5}] (3) at ( -3.5, 1.9) {};
		\node[style={circle, draw, scale=.5}] (4) at (-2.8,0.3) {};
		\node[style={circle, draw, scale=.5}] (5) at ( -2.5, 1) {};
		\node[fill=black,style={circle, draw, scale=.5}] (6) at ( -3, 1.67) {};
		\node[fill=black,style={circle, draw, scale=.5}] (7) at ( -1, 1) {};
		\node at ( -1.1, 1.3) {$x_1$};
		\node[style={circle, draw, scale=.5}] (8) at ( -.3, .3) {};
		\node[fill=black,style={circle, draw, scale=.5}] (9) at ( 0, 1) {};
		\node[style={circle, draw, scale=.5}] (10) at ( -.5, 1.67) {};
		\node[style={circle, draw, scale=.5}] (41) at ( 3, 1.0) {};
		\node at ( 3, 1.3) {$0$};
		\node[fill=black,style={circle, draw, scale=.5}] (42) at ( 4.5, 1) {};
		\node at ( 4.4, .7) {$x$};
		\node[style={circle, draw, scale=.5}] (43) at ( 4.5, 1.9) {};
		\node[style={circle, draw, scale=.5}] (44) at (5.2,0.3) {};
		\node[style={circle, draw, scale=.5}] (45) at ( 5.5, 1) {};
		\node[fill=black,style={circle, draw, scale=.5}] (46) at ( 5, 1.67) {};
		\node[fill=black,style={circle, draw, scale=.5}] (47) at ( 7, 1) {};
		\node at ( 6.85, .7) {$x_1$};
		\node[style={circle, draw, scale=.5}] (48) at ( 7.7, .3) {};
		\node[fill=black,style={circle, draw, scale=.5}] (49) at ( 8, 1) {};
		\node[style={circle, draw, scale=.5}] (40) at ( 7.5, 1.67) {};
		\node (11) at (1.1,1.5){};
		\node (12) at (2.35,0.25){};
		\draw[thick,->] (11) -- (12);
	    \draw[very thin, dashed] (42) -- (41);
		\draw[very thin] (43) -- (42);
		\draw[very thin] (44) -- (42);
		\draw[very thick] (46)node[right]{$\hspace{.3mm}\overline{x}$} -- (42);
		\draw[very thin] (45) -- (42);
		\draw[dashed, very thin] (47) -- (45);
		\draw[very thin] (48) -- (47);
		\draw[very thick] (49)node[right]{$\hspace{.3mm}\red{\overline{x}_1.}$} -- (47);
		\draw[very thin] (40) -- (47);
		\draw[dashed, very thin] (2) -- (1);
		\draw[very thin] (3) -- (2);
		\draw[very thin] (4) -- (2);
		\draw[very thin] (5) -- (2);
		\draw[very thin] (6)node[right]{$\hspace{.3mm}\overline{x}$} -- (2);
		\draw[dashed, very thin] (7) -- (5);
		\draw[very thin] (8) -- (7);
		\draw[very thick] (9)node[right]{\red{$\overline{x}_1$}} -- (7);
		\draw[very thin] (10) -- (7);
	\end{tikzpicture}}
\end{equation}

\red{Note that any $(m+1)$-cube $c_0$ that has been identified on the right of~(\ref{endingpairing}) as a potential destination of a path $\gamma\in\mathcal{G}$} supports a gradient path $\lambda\colon\red{c_0}\searrow c_1\nearrow\cdots\searrow c_t\hspace{.5mm}$ with $c_t$ a critical $m$-cube. For instance, \red{start by replacing} the edge $(x_1,\overline{x}_1)$ in $\red{c_0}$ by $x_1$, and let the rest of the path consist of falling-vertex elementary factors. It follows that the concatenation of $\gamma$ and $\lambda$ and, therefore, $\gamma$ itself obey the rule in Corollary \ref{dynamics}: any upper elementary factor of sor type is of falling-vertex type. Such a fact, \red{together with cancellation phenomena similar to the one in the proof of Proposition~\ref{underneathphi}, is used in the rest of the argument in order to analyze paths determining}~(\ref{pi}). As in the proof of Proposition~\ref{trivialmorsediff}, the analysis \red{can equivalently be done} at the level of $C^*(\UDnT)$, which means that \red{an ordered cube} $c\cdot\sigma$ can be replaced by \red{the corresponding} orbit $\{c\}$. \red{Following the lead in} Proposition \ref{underneathphi}, we \red{first identify} the \red{actual} sets of paths whose contribution in~(\ref{pi}) give (\ref{laquefaltaba})--(\ref{unomas}).

\red{The summation  in~(\ref{laquefaltaba}) arises from a set} $\mathcal{L}^-_0\subset\mathcal{G}$ of paths \red{having} a single ``lock'' dynamics. Explicitly, \red{each $r_x$-tuple $a$ of non-negative integer numbers satisfying $0<|a|\leq R_0$ determines a lower gradient path $\red{\lambda^-_{a,0}}\in\mathcal{L}^-_0$ that departs} from the critical $(m+1)$-cube
$$
\left\{R_0-|\red{a}| \ll x,\red{a},Q_{x}\ll x_1,P_{1},Q_{1}\ll\ldots\ll x_m,P_{m},Q_{m}\right\}
$$
\red{by} replacing the edge $(x,\overline{x})$ by $\overline{x}$ ---this opens the lock. Then $\red{\lambda^-_{a,0}}$ continues with the falling of the $|\red{a}|$ vertices \red{that were} blocked by $x$, after which $\red{\lambda^-_{a,0}}$ ends with the pairing that closes the lock by creating the edge $(x,\overline{x})$ \red{required} in (\ref{endingpairing}). \red{Since both opening and closing locks} are associated to the same face (the gradient-orientated $\delta_2$-face), and since falling-vertex elementary paths have multiplicity 1, we see from (\ref{multiplicitydefinition}) that $\mu(\red{\lambda^-_{a,0}})=-1$. \red{Thus, $\mathcal{L}^-_0\subseteq\mathcal{G}$} yields (\ref{laquefaltaba}).

The set of paths $\red{\mathcal{L}_0^-}$ is contained in a slightly larger subset $\mathcal{L}^{-}\subset\mathcal{G}$ \red{which consists of paths $\lambda_{a,b}^-$, where $a$ runs over $r_x$-tuples of non-negative integer numbers and $b$ runs over non-negative integers numbers satisfying} $|a|>0$ and $ |a|+\red{b}\leq R_0$. \red{Explicitly, $\lambda_{a,b}^-$} starts by taking face $\delta_2$ (lock opening) of \red{the} critical $(m+1)$-cube
$$
\left\{R_0-|a|-\red{b}\ll x,a,\red{Q_{x}+\red{(b,0,\ldots,0)}}\ll x_1,P_{1},Q_{1}\ll\ldots\ll x_m,P_{m},Q_{m}\right\}.
$$
Here and below we take the coordinate-wise sum of tuples. Then \red{$\lambda_{a,b}^-$ continues with the falling of} the $|a|$ vertices that were blocked by~$x$, followed \red{(if $b>0$)} by the falling of the \red{$b$} vertices \red{$S_b(\overline{x})$,} to finish with the falling of \red{$\overline{x}+b$} until it creates the required branch-type pairing (\ref{endingpairing}) ---which closes the lock. As in the case of~$\red{\mathcal{L}^-_0},$ paths in $\mathcal{L}^{-}$ have multiplicity $-1$. Likewise, there is the family $\mathcal{L}^{+}\subset\mathcal{G}$ consisting of paths \red{$\lambda_{a,b}^+$, with $a$ and $b$ as above, except that the inequality $|a|+b\leq R_0$ is replaced by the strict inequality $|a|+b<R_0$. Explicitly, $\lambda_{a,b}^+$ starts} by taking face $\delta_1$ (inverse lock opening) of \red{the} critical $(m+1)$-cube
$$
\left\{R_0-|a|-\red{b-1}\ll x,\red{a},Q_{x}+(\red{b+1,0,\ldots,0})\ll x_1,P_{1},Q_{1}\ll\ldots\ll x_m,P_{m},Q_{m}\right\}.
$$
Then \red{$\lambda_{a,b}^+$ continues with the falling of~$x$, followed by the falling of the} $|a|$ vertices \red{that were blocked by $x$, followed (if $b>0$) by the falling of} the~$\red{b}$ vertices \red{$S_{\overline{x}+1}(b)$,} to finish with the falling of \red{$\overline{x}+b+1$} until it creates the required branch-type pairing (\ref{endingpairing}) ---which closes the lock. Note that paths in $\mathcal{L}^{+}$ have multiplicity $+1$.

\medskip
\red{Figure~\ref{figchart} summarizes dynamics of paths in $\mathcal{L}^-$ (top) and paths in $\mathcal{L}^+$ (bottom), with lock opening/closing represented by arrows. Note the shifting on the $b$ vertices falling from $x$-direction $r_x+1$, as well as on the vertices that make up $B_{x,r_x+1}$ at the end of the path. The key point is that, if $b>0$, the paths $\lambda^-_{a,b}$ and $\lambda^+_{a,b-1}$} share origen, so their contributions in (\ref{pi}) cancel each other out. \red{The only} unmatched paths are those in $\mathcal{L}^{-}$ \red{with} parameter $b=0$, i.e., paths in~$\red{\mathcal{L}^-_0,}$ whose contribution in (\ref{pi}) has been shown to yield~(\red{\ref{laquefaltaba}}).

\begin{figure}[h!]
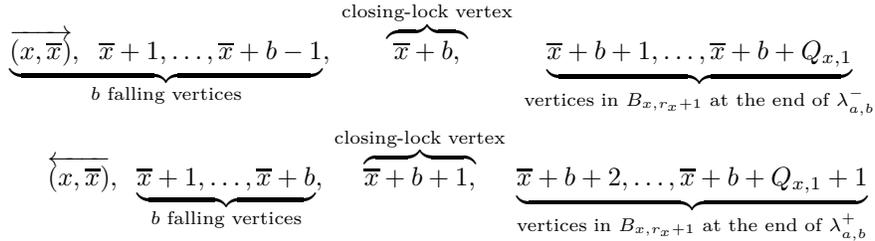

\begin{align*}
\underbrace{\overrightarrow{(x,\overline{x})},\hspace{2mm}\overline{x}+1,\ldots,\overline{x}+b-1}_{\text{$b$ falling vertices}},&\;\overbrace{\;\overline{x}+b,\;}^{\text{closing-lock vertex}}\;\underbrace{\overline{x}+b+1,\ldots,\overline{x}+b+Q_{x,1}}_{\text{vertices in $B_{x,r_x+1}$ at the end of $\lambda_{a,b}^-$}} \\
\overleftarrow{(x,\overline{x})},\hspace{2mm}\underbrace{\overline{x}+1,\ldots,\overline{x}+b}_{\text{$b$ falling vertices}},\;&\overbrace{\overline{x}+b+1,}^{\text{closing-lock vertex}}\;\underbrace{\overline{x}+b+2,\ldots,\overline{x}+b+Q_{x,1}+1}_{\text{vertices in $B_{x,r_x+1}$ at the end of $\lambda_{a,b}^+$}}
\end{align*}
\caption{Dynamics of paths in $\mathcal{L}^-$ (top) and $\mathcal{L}^+$ (bottom)}
\label{figchart}
\end{figure}

\red{By construction, $\mathcal{L}^-\cup\mathcal{L}^+$ consists of those paths in $\mathcal{G}$ that start by taking a face $\delta_i$ with $i=1,2$ of a critical $(m+1)$-cube with edges $$(x,\overline{x}), (x_1,\overline{x}_1),\ldots,(x_m,\overline{x}_m),$$ and that evolve exclusively though falling-vertex elementary paths before reaching the required pairing~(\ref{endingpairing}). Next we describe similar sets of paths contributing in~(\ref{pi}) with~(\ref{dosmas}) and~(\ref{unomas}). In such sets of paths, an edge
\begin{equation}\label{roler}
(x,x[r+1]) \mbox{ \ \ with \ }r\neq r_x
\end{equation}
plays the role of the edge $(x,\overline{x})=(x,x[r_x+1])$ in $\mathcal{L}^\pm$.}

\medskip\noindent
\red{{\bf Paths $\mathcal{K}_\ell^-$ with $1\leq\ell\leq s_x-1$} ($r=r_x+\ell$, in the notation of~(\ref{roler})): If $Q_{x,\ell+1}=0$, set $\mathcal{K}^-_\ell=\varnothing$, otherwise $\mathcal{K}^-_\ell$ consists of paths $\kappa^-_{\ell,a,b}\in\mathcal{G}$, where $a$ runs over $r_x$-tuples of non-negative integer numbers and $b$ runs over non-negative integer numbers satisfying $|a|+b\leq R_0$. Explicitly, if $a=(a_1,\ldots,a_{r_x})$, then $\kappa^-_{\ell,a,b}$ starts by taking face $\delta_2$ of the critical $(m+1)$-cube
\begin{align*}
\left\{R_0-|a|-\red{b}\ll\right.\hspace{-.8mm}x,(a_1,\ldots,a_{r_x},&Q_{x,1}+b+1,Q_{x,2},\ldots,Q_{x,\ell}),\\&(Q_{x,\ell+1}-1,Q_{x,\ell+2},\ldots,Q_{x,s_x})\ll x_1,P_{1},Q_{1}\ll\ldots\left.\ll x_m,P_{m},Q_{m}\right\},
\end{align*}
and evolves through falling-vertex elementary paths as depicted by the chart\footnote{As in the case of $\mathcal{L}^\pm$, the $|a|$ vertices falling from $x$-directions 1 through $r_x$ are not shown in the chart.}
\begin{align*}
\overrightarrow{(x,x[r+1])},\hspace{2mm}\underbrace{\overline{x},\ldots,\overline{x}+b-1}_{\text{$b$ falling vertices}},&\;\overbrace{\;\overline{x}+b,\;}^{\text{closing-lock vertex}}\;\underbrace{\overline{x}+b+1,\ldots,\overline{x}+b+Q_{x,1}}_{\text{vertices in $B_{x,r_x+1}$ at the end of $\kappa_{\ell,a,b}^-$}}
\end{align*}
before reaching the required pairing~(\ref{endingpairing}). Both opening and closing locks of $\kappa^-_{\ell,a,b}$ are associated to a (gradient-oriented) $\delta_2$ face, so that $\mu(\kappa^-_{a,b})=-1$. The contribution in~(\ref{pi}) of the paths in $\mathcal{K}^-_1\cup\cdots\cup\mathcal{K}^-_{s_x-1}$ thus gives raise to~(\ref{unomas}). Note that no path that starts from the origin of a given $\kappa^-_{a,b}$ by taking face $\delta_1$ ---instead of $\delta_2$---, and that evolves through falling-vertex elementary paths, can arrive to a summand of~(\ref{elmeromero}). This is why the contribution to~(\ref{pi}) of the set of paths in the next paragraph does not cancel out terms in~(\ref{unomas}).}

\medskip\noindent
\red{{\bf Paths $\mathcal{K}_\ell^+$ with $1\leq\ell\leq s_x-1$} ($r=r_x+\ell$, in the notation of~(\ref{roler})): $\mathcal{K}^+_\ell$ consists of paths $\kappa^+_{\ell,a,b}\in\mathcal{G}$, where $a$ runs over $r_x$-tuples of non-negative integer numbers and $b$ runs over non-negative integer numbers satisfying $|a|+b<R_0$. Explicitly, if $a=(a_1,\ldots,a_{r_x})$, then $\kappa^+_{\ell,a,b}$ starts by taking face $\delta_1$ of the critical $(m+1)$-cube
\begin{align*}
\left\{R_0-|a|-b-1\ll\right.\hspace{-.8mm}x,(a_1,\ldots,a_{r_x},&Q_{x,1}+b+1,Q_{x,2},\ldots,Q_{x,\ell}),\\&(Q_{x,\ell+1},Q_{x,\ell+2},\ldots,Q_{x,s_x})\ll x_1,P_{1},Q_{1}\ll\ldots\left.\ll x_m,P_{m},Q_{m}\right\},
\end{align*}
and evolves through falling-vertex elementary paths as depicted by the chart
\begin{align*}
\underbrace{\overleftarrow{(x,x[r+1])},\hspace{2mm}\overline{x},\ldots,\overline{x}+b-1}_{\text{$b+1$ falling vertices}},&\;\overbrace{\;\overline{x}+b,\;}^{\text{closing-lock vertex}}\;\underbrace{\overline{x}+b+1,\ldots,\overline{x}+b+Q_{x,1}}_{\text{vertices in $B_{x,r_x+1}$ at the end of $\kappa_{\ell,a,b}^+$}}
\end{align*}
before reaching the required pairing~(\ref{endingpairing}). Now $\mu(\kappa^+_{a,b})=1$, so the contribution in~(\ref{pi}) of the paths in $\mathcal{K}^+_1\cup\cdots\cup\mathcal{K}^+_{s_x-1}$ gives raise to~(\ref{dosmas}). Again, no path that starts from the origin of a given $\kappa^+_{a,b}$ by taking face $\delta_2$ ---instead of $\delta_1$---, and that evolves through falling-vertex elementary paths, can arrive to a summand of~(\ref{elmeromero}).}

\begin{remark}\label{casoespecial}{\em
\red{Since the closing-lock pairing~(\ref{endingpairing}) must come from $x$-direction $r_x+1$, paths corresponding to cases with $r<r_x$ in~(\ref{roler}) have no contribution in~(\ref{pi}). Specifically, any path $\gamma\in\mathcal{G}$ that starts from a critical cell with edges $(x,x[r+1]),(x_1,\overline{x}_1),\ldots,(x_m,\overline{x}_m)$, where $r<r_x$, by taking a face $\delta_i$ with $i=1,2$, and that reaches the pairing~(\ref{endingpairing}) through falling-vertex elementary paths, has a companion path $\gamma'$ that starts from the same critical cell by taking the face $\delta_{3-i}$, and that also evolves through falling-vertex elementary paths until it reaches the closing-lock pairing~(\ref{endingpairing}) ---so that $\mu(\gamma')=-\mu(\gamma)$ and $(\gamma')'=\gamma$. Note that, in the ordered setting, $\gamma$ and its companion path $\gamma'$ arrive to summands of~(\ref{elmeromero}) whose ingredients differ only by a permutation (so $\gamma'\in\mathcal{G}$ as well). The phenomenon noticed in this remark is in fact the key to finishing the proof of the main result in this section.}
}\end{remark}

\noindent\red{\emph{Proof of Theorem~\ref{productsvialowerpaths} (conclusion).} Let $\mathcal{J}$ stand for the set of paths analyzed up to this point, i.e., the paths in $\mathcal{G}$ that (I) depart} from a critical $(m+1)$-cube \red{with gradient-ordered} edges $(x,\red{x[\ell]}),(x_1,\overline{x}_1),\ldots,(x_m,\overline{x}_m)$, (II) \red{start by taking} the face $\delta_1$ or $\delta_2$ and (III)~\red{reach} the ending branch-type pairing (\ref{endingpairing}) \red{exclusively through} falling-vertex elementary paths. It suffices to construct an involution \red{$\iota\colon \mathcal{G}'\to\mathcal{G}'$, with $\mathcal{G}':=\mathcal{G}\setminus\mathcal{J}$,} such that each pair of paths $\gamma$ and $\iota(\gamma)$ share origin and have opposite multiplicity. \red{With this in mind, we first note that} condition (II) is forced by conditions (I) and (III). \red{Indeed,} in any gradient path $e\searrow e'\nearrow\cdots$ all whose upper elementary factors are of falling-vertex type,
\begin{equation}\label{presentes}
\mbox{\it the edge ingredients of $e'$ are present in all steps of the path.}
\end{equation}
Therefore \red{$\mathcal{G}'$} is partitioned into two sets, $\red{\mathcal{G}'}_\text{fall}$ and $\red{\mathcal{G}'}_\text{branch}$, where the former set consists of the paths in $\mathcal{G}$ that satisfy (III) without satisfying~(I), and the latter set consists of the paths in $\mathcal{G}$ that do not satisfy (III). We construct involutions $\iota_\text{fall}\colon\red{\mathcal{G}'}_\text{fall}\to\red{\mathcal{G}'}_\text{fall}$ and $\iota_\text{branch}\colon\red{\mathcal{G}'}_\text{branch}\to\red{\mathcal{G}'}_\text{branch}$ with the required properties.

For a path $\red{\gamma}=a_0\searrow b_1\nearrow a_1\searrow\cdots\searrow b_k\nearrow a_k$ in $\red{\mathcal{G}'}_\text{fall}$, \red{the observation in~(\ref{presentes}) and the form of the closing-lock pairing $b_k\nearrow a_k$ imply that} all edges $(x_i,\overline{x}_i)$, $1\leq i\leq m$, must be ingredients of \red{$a_0$. The additional edge of the critical $(m+1)$-cube $a_0$ must then have the form $(y,y[d])$, with $y\not\in\{x,x_1,\ldots,x_m\}$, which is then replaced by either $y$ or $y[d]$ at the beginning of $\gamma$. Given the form of $b_k\nearrow a_k$, $y$ must lie in $x$-direction $r_x+1$. Then,} as in the proof of Proposition~\ref{trivialmorsediff}, the definition of $\iota_{\text{fall}}$ is based on the two options \red{for $a_0\searrow b_1$, as both lead to summands of~(\ref{elmeromero}) ---unlike the situation in Remark~\ref{casoespecial}, the ending cube of $\iota_{\text{fall}}(\gamma)$ might fail to be in the $\Sigma_n$-orbit of the ending cube of $\gamma$.} Likewise, the definition of $\iota_{\text{branch}}$ is based on the two forms of replacing by a vertex the edge ingredient at the last upper elementary factor that is not of falling-vertex type.
\end{proof}

\section{Exterior-face basis for trees with binary core}
We have made a careful distinction between $\image$ and $H^*(\UDnT;R)$ in the previous sections so to provide clear proof arguments. In this section we use the resulting algebro-combinatorial description of cup-products and have no need to make any further distinction between these isomorphic rings. Accordingly, we transfer the notation and descriptions of elements in $\image$ back to $H^*(\UDnT;R)$. \red{In particular, the notation and conventions in the paragraph containing~(\ref{aproduct}) will be carried over this final section, directly in the context of $H^*(\UDnT;R)$, with the simplifications discussed below.} 

\begin{definition}
A tree $T$ is said to have binary core provided that, for each essential vertex $x$ of $T$, at most two of the components of $T\setminus\{x\}$ in $x$-directions $1,2,\ldots,d_x$ carry essential vertices (recall $d_x:=d(x)-1$).
\end{definition}

Throughout this section, $T$ stands for a tree with binary core (e.g.~an actual binary tree). In addition, we assume that the chosen planar embedding of $T$ has been adjusted so that, for any essential vertex~$x$ of $T$, 
\begin{equation}\label{laassumpcion}
\mbox{\emph{no component of $T\setminus\{x\}$ in $x$-direction $j$ with $1\leq j\leq d_x-2$ carries an essential vertex.}}
\end{equation}
There are two reasons for sticking to such an hypothesis. For one, the existence of non-vanishing products whose factors are given by weak-interacting basis elements $$\{k_i\l x_i,(p_{i,1},\ldots,p_{i,r_i}),(q_{i,1},\ldots,q_{i,s_i})\}$$ with $x_1<\cdots<x_m$, i.e., the obstructions in Remark~\ref{obstruction}, is somehow restricted (cf.~Example~\ref{linearcase}), while our description of the corresponding product is greatly simplified. Explicitly, in the setting and notation of Theorem~\ref{productsvialowerpaths}, since the top left configuration in Figure~\ref{four} holds,~(\ref{laassumpcion}) forces $s_x=1$, i.e., the edge $(x,\overline{x})$ must lie in the largest $x$-direction, with $x_1$ then lying in the second largest $x$-direction $r_x=d_x-1$. In particular, the product $\Pi_2$ takes the simpler form
\begin{equation}\label{formareducida}
\Pi_2=-\sum\left\{R_0-|\red{a}|\ll x,\red{a},Q_x\ll x_1, P_{1},Q_{1}\ll\cdots\ll x_m, P_{m},Q_{m}\right\},
\end{equation}
where the sum runs over all $r_x$-tuples of integer numbers $a=(a_1,\ldots, a_{r_x})$ with $a>0$ and $|a|\leq R_0$.

The second advantage for working under the situation in~(\ref{laassumpcion}) is that, for $1\leq i\leq m$ and $j\leq d_i-2$, any set of pruned leaves $L_{i,j}$ associated to a product~(\ref{aproduct}) is empty. As a result, the corresponding $C_{i,j}$-local interaction is ``vacuous'' in the sense that the $C_{i,j}$-instance of~(\ref{localinteraction}) simplifies to $\ell_{C_{i,j}}(\nu_i)\geq0$ ---a condition which is certainly true. In fact, still in the context of~(\ref{aproduct}), there will be no local interactions in the positive $x_i$-directions leading to a weak interaction situation as long as $p_{i,j}>0$ for some $j\leq\min\{r_i,d_i-2\}$ (cf.~(\ref{px0})). In particular, it makes sense to reset the notation for pruned leaves in the presence of~(\ref{laassumpcion}): we shall set $L_1(x_i):=L_{i,d_i-1}$ and $L_2(x_i):=L_{i,d_i}$ when $i>0$, and $L_1(x_0):=L_{0,1}$ (recall from Definition~\ref{hojas} that $x_0$ stands for the root of $T$).

Expression~(\ref{formareducida}) suggests redefining some of the basis elements $\langle k,x,(p_1,\ldots,p_r),(q_1,\ldots,q_s)\rangle\in H^1(B_nT)$ in the proof of Theorem~\ref{casogeneral}. Namely, for the purposes of this section, if $p_1=\cdots=p_{r-1}=0$ and $s=1$, we set
\begin{equation}\label{chbs}
\langle k,x,(p_1,\ldots,p_r),(q_1,\ldots,q_s)\rangle:=\sum \left\{k-|a|\ll x,(a_1,\ldots,a_{r-1},p_r+a_r),(q_1)\right\},
\end{equation}
where the summation runs over all $r$-tuples $a=(a_1,\ldots,a_r)\geq0$ with $|a|\leq k$, otherwise we keep $$\langle k,x,(p_1,\ldots,p_r),(q_1,\ldots,q_s)\rangle:=\left\{k\ll x,(p_1,\ldots,p_r),(q_1,\ldots,q_s)\right\}.$$

\begin{remark}\label{angles}{\em
We use the angle-bracket notation $\left\langle \rule{0mm}{3mm}k,x,p,q\right\rangle$ since we have reserved the parenthesis notation for cubes in $\DnT$ (as tuples of their ingredients). Additionally, the angle-bracket notation is intended to stress the change of basis in~(\ref{chbs}).
}\end{remark}

A central task in this section is the analysis of the relationship between \red{ordered\footnote{In the sense that $x_1<\cdots<x_m$.}} products
\begin{equation}\label{dosproductos}
\langle k_1,x_1,p_{1},q_{1}\rangle\cdots\langle k_m,x_m,p_{m},\red{q}_{m}\rangle \mbox{ \ and \ } \{ k_1\l x_1,p_{1},q_{1}\}\cdots\{ k_m\l x_m,p_{m},q_{m}\}.
\end{equation}
We say that any of these products is a strong interaction product if the \red{factors of the product on the right hand-side of~(\ref{dosproductos})} interact strongly (in the sense of Definition~\ref{interactionparameters}). 

\begin{remark}\label{configuraciones}{\em Corollary~\ref{interactionproducts}, Proposition~\ref{interactionlengths} and Theorem~\ref{productsvialowerpaths} show that both products in (\ref{dosproductos}) are (possibly empty) linear combinations of basis elements $\{\centerdot\l x_1,\centerdot,\centerdot\l\cdots\l x_m,\centerdot,\centerdot\}$. Such a linear combination will be written as $$\sum\raisebox{.6mm}{$\centerdot$}\,\{\centerdot\l x_1,\centerdot,\centerdot\l\cdots\l x_m,\centerdot,\centerdot\}.$$ Here and below, a dot~`$\centerdot$' stands for either an unspecified ring coefficient, or an unspecified tuple\footnote{As in Definition~\ref{interactionparameters}, we make no distinction between integer numbers and 1-tuples.} of integer numbers, $t=(t_1,t_2,\ldots)\geq0$, satisfying $t>0$ when the tuple immediately follows an essential vertex~$x_i$ (the context clarifies the option).
}\end{remark}

\begin{theorem}\label{binarycore}
Let $T$ be a tree with binary core, $R$ be a commutative ring with 1, and $n\geq1$. Then $H^*(B_nT;R)\cong\Lambda_R(K_nT)$. In detail: \emph{(i)} An ordered product $\langle k_1,x_1,p_{1},q_{1}\rangle\cdots\langle k_m,x_m,p_{m},q_{m}\rangle$ is non-zero if and only if it is a strong interaction product. \emph{(ii)} Two ordered strong interaction products agree if and only if they have the same factors. \emph{(iii)} A graded basis of $H^*(\UDnT)$ is given by the set  of ordered strong interaction products.
\end{theorem}

The crux of the matter in the proof of Theorem~\ref{binarycore} is getting at a precise description of the conditions that have to be satisfied by some of the unspecified dot ingredients in
\begin{equation}\label{OP}
\aop=\sum\raisebox{.6mm}{$\centerdot$}\,\{\centerdot\l x_1,\centerdot,\centerdot\l\cdots\l x_m,\centerdot,\centerdot\}.
\end{equation}
With this in mind, the product in~(\ref{OP}) will be denoted by $\varpi$ throughout the section, setting $$R_0:=\mathcal{R}_0(x_1,\ldots,x_m),\,\; P_{i,j}:=\mathcal{P}_{i,j}(x_1,\ldots,x_m),\,\; Q_{i,j}:=\mathcal{Q}_{i,j}(x_1,\ldots,x_m),$$ $P_i:=(P_{i,1}, \dots, P_{i,r_i})$ and $Q_i:=(Q_{i,1}, \dots, Q_{i,s_i})$, $1\leq i\leq m$, for the corresponding interaction parameters.  Furthermore, we set
\begin{equation}\label{prelevel}
B_i\coloneqq(x_i,P_{i},Q_{i})\ \ \text{and} \ \ \dt{B}_i\coloneqq(x_i,\centerdot,\centerdot),
\end{equation}
where the latter expression stands for any triple with unspecified tuples in the second and third coordinates (subject to the usual restrictions). Additionally, the $i$-th factor on the left hand-side of (\ref{OP}) will denoted by~$\phi_i$. For instance, in terms of the notation set forth in Definition~\ref{interactionparameters}, $$\phi_i=\{k_i\l x_i,p_{i},q_{i}\}+\sum\{\underline{k_i}\l x_i,\centerdot,q_{i}\},$$ with a possibly empty summation, whereas Corollary~\ref{interactionproducts} asserts that the second product in (\ref{dosproductos}) is trivial or agrees with $\{R_0\l B_1\l\ldots\l B_m\}$ under, respectively, the no-interaction or strong-interaction condition of the factors.

In the following results, some of which are true for general trees, we make free use of the notation and considerations above. Likewise, the use of cup-product descriptions in Sections~\ref{sectioncupproducts} and~\ref{lgp}, with the simplification in~(\ref{formareducida}), we will referred generically as ``interaction reasons''.

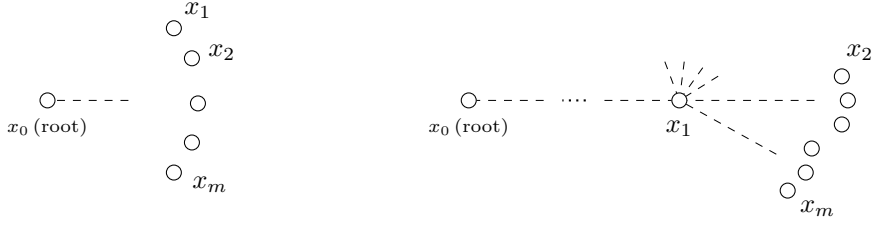
\begin{figure}
	\centering
	\begin{tikzpicture}[scale=.8, rotate = 180, xscale = -1]
		\node[style={circle, draw, scale=.6}] (1) at ( -6, 4.5) {};
		\node (2) at ( -4.5, 4.5) {};
		\node[style={circle, draw, scale=.6}] at ( -3.9, 5.7) {};
		\node[style={circle, draw, scale=.6}] at ( -3.9, 3.3) {};
		\node[style={circle, draw, scale=.6}] at ( -3.6, 5.2) {};
		\node[style={circle, draw, scale=.6}] at ( -3.6, 3.8) {};
		\node[style={circle, draw, scale=.6}] at ( -3.5, 4.55) {};
		\node at ( -3.5, 3) {$x_1$}; \node at ( -3.1, 3.7) {$x_2$}; \node at ( -3.3, 5.9) {$x_m$};

		\node[style={circle, draw, scale=.6}] (3) at ( 1, 4.5) {};
		\node (11) at ( 2.4, 4.5) {}; \node (12) at ( 3.1, 4.5) {};
		\node[style={circle, draw, scale=.6}] (4) at ( 4.5, 4.5) {};
		\node (5) at ( 6.3, 5.5) {}; \node (6) at ( 6.9, 4.5) {};
		\node (7) at ( 5.3, 4) {}; \node (8) at ( 5, 3.8) {}; 
		\node (9) at ( 4.6, 3.7) {}; \node (10) at ( 4.2, 3.7) {};
		
		\draw[dotted, thick] (12) -- (11);
		\draw[dashed, very thin] (2) -- (1)node[below]{\raisebox{-3mm}{\scriptsize $x_0\,$(root)}};
		\draw[very thin, dashed] (12) -- (4);
		\draw[very thin, dashed] (11) -- (3)node[below]{\raisebox{-3mm}{\scriptsize $x_0\,$(root)}};
		\draw[very thin, dashed] (4)node[below]{\raisebox{-3mm}{$x_1$}} -- (5);
		\draw[very thin, dashed] (6) -- (4);
		\draw[very thin, dashed] (7) -- (4); \draw[very thin, dashed] (8) -- (4);
		\draw[very thin, dashed] (9) -- (4); \draw[very thin, dashed] (10) -- (4);
		\node[style={circle, draw, scale=.6}] at (6.6, 5.7) {};
		\node[style={circle, draw, scale=.6}] at (6.7, 5.3) {};
		\node[style={circle, draw, scale=.6}] at (6.3, 6) {};
		\node[style={circle, draw, scale=.6}] at (7.3, 4.5) {};
		\node[style={circle, draw, scale=.6}] at (7.2, 4.1) {};
		\node[style={circle, draw, scale=.6}] at (7.2, 4.9) {};
		\node at (7.5, 3.7) {$x_2$}; \node at ( 6.8, 6.3) {$x_m$};

	\end{tikzpicture}
	\caption{Configurations of essential vertices in Lemma~\ref{variasconfigs}.}
\label{configs}
\end{figure}

\begin{lemma}\label{variasconfigs}
\begin{enumerate}[(1)]
\item\label{T01}
Assume $L_1(x_0)=\{x_1,x_2,\ldots,x_m\}$ (left configuration in Figure~\ref{configs}), then
$$
\varpi=\begin{cases}
\{R_0\l B_1\l\cdots\l B_m\} +\sum\raisebox{.6mm}{$\centerdot$}\,\{\underline{R_0}\l \dt{B}_1\l\cdots\l \dt{B}_m\}, & \mbox{if $R_0\geq0$;} \\
\rule{93pt}{0mm}0, & \mbox{otherwise.}
\end{cases}
$$
\item\label{T12}
Assume $L_{1}(x_1)=\{x_2,x_3,\ldots,x_u\}$ and $L_{2}(x_1)=\{x_{u+1},\ldots,x_{m-1},x_m\}$ with $1\leq u\leq m$ (right configuration in Figure~\ref{configs}) with $u=1$ $($i.e. $L_1(x_1)=\varnothing$ and $L_2(x_1)=\{x_2,\ldots,x_m\})$ if $s_1=1$, then 
$$
\varpi=\begin{cases}
\{R_0\l B_1\l\cdots\l B_m\}+
\sum\raisebox{.6mm}{$\centerdot$}\,\{\underline{R_0}\l \dt{B}_1\l\cdots\l \dt{B}_m\}+\sum\raisebox{.6mm}{$\centerdot$}\,\{R_0\l x_1,P_{1},\underline{Q_{1}}\l\dt{B}_2\l\cdots\l \dt{B}_m\}, & \mbox{if $Q_1\geq0;$} \\ 0, & \mbox{otherwise.}
\end{cases}
$$
\end{enumerate}
\end{lemma}
\begin{proof}
The first assertion follows by direct inspection of the expression
$$
\left(\{k_1\l x_1,p_{1},q_{1}\} +\sum\{\underline{k_1}\l x_1,\centerdot,\centerdot\}\right)\cdots\left(\{k_m\l x_m,p_{m},q_{m}\}+\sum\{\underline{k_m}\l x_m,\centerdot,\centerdot\}\right), 
$$
noticing that the only non-vacuous interaction occurs in the tree $T_{0,1}$ (so that $P_i=p_i$ and $Q_i= q_i$ for $1\leq i\leq m$). The second assertion is proved in a similar way, noticing that this time non-vacuous interactions occur only either on $T_{1,d_1}$ or $T_{1,d_1-1}$ (or both). In any case, $R_0=k_1$, $P_{i}=p_{i}$  for $ 1\leq i \leq m$, while $Q_{i}=q_i$ for $2\leq i\leq m$. 
\end{proof}

A key situation with $L_1(x_1)\cup L_2(x_1)=\{x_2,x_3,\ldots,x_m\}$ not covered by Lemma~\ref{variasconfigs}(\ref{T12}) is:
\begin{lemma}\label{T11}
Assume $L_{1}(x_1)=\{x_2,x_3,\ldots,x_m\}$. Then the product of $\langle k_1, x_1,(p_{1,1},\ldots,p_{1,d_1-1}),(q_{1,1})\rangle$ with $\{R\l x_2,p_{2},q_{2}\l\cdots\l x_m,p_{m},q_{m}\}$ vanishes provided $p_{1,1}=\cdots=p_{i,d_1-2}=0$ and $p_{1,d_1-1}+R\leq n$.
\end{lemma}
\begin{proof}
We proceed by induction on $p_{1,d_1-1}+R-n=p_{1,d_1-1}+\sum_{j=2}^m(t_j-n)\in\{0,-1,-2,\ldots\}$, where
$$\{R\l x_2,p_{2},q_{2}\l\cdots\l x_m,p_{m},q_{m}\}=\{t_2\l x_2,p_{2},q_{2}\}\cdots\{t_m\l x_m,p_{m},q_{m}\}$$
is the unique strong-interaction factorization of $\{R\l x_2,p_{2},q_{2}\l\cdots\l x_m,p_{m},q_{m}\}$ noted in the proof of Theorem~\ref{casogeneral}. Since $p_{1,j}=0$ for $j=1,\ldots,d_1-2$, the induction is grounded for $p_{1,d_1-1}+R-n=0$ by
\begin{align*}
\{k_1\l \, & x_1,(p_{1,1},\ldots,p_{1,d_1-1}),(q_{1,1})\}\cdot\left(\{t_2\l x_2,p_{2},q_2\}\cdots\{t_m\l x_m,p_{m},q_{m}\}\rule{0mm}{4mm}\right)=\\
&=-\sum\{ k_1-|a|\l x_1,a,(q_{1,1})\l x_2,p_2,q_2\l\cdots\l x_m,p_m,q_m\}\\
&=-\sum\{k_1-|a|\l x_1,(a_1,\ldots,a_{d_1-2},p_{1,d_1-1}+a_{d_1-1}),(q_{1,1})\} \left(\{t_2\l x_2,p_{2},q_{2}\}\cdots\{t_m\l x_m,p_{m},q_{m}\}\rule{0mm}{4mm}\right),
\end{align*}
where both summations run over tuples $a=(a_1,\ldots,a_{d_1-1})>0$ with $|a|\leq k_1$. The inductive step then follows by noticing that, for $p_{1,d_1-1}+R-n<0$,
\begin{align*}
\langle k_1, x_1,&\,(p_{1,1},\ldots,p_{1,d_1-1}),(q_{1,1})\rangle\cdot\left(\rule{0mm}{4mm}\{t_2\l x_2,p_2,q_2\}\cdots\{t_m\l x_m,p_m,q_m\}\right)=\\
&=\langle k_1-1, x_1, (p_{1,1},\ldots,p_{1,d_1-2},p_{1,d_1-1}+1),(q_{1,1})\rangle\cdot\left(\rule{0mm}{4mm}\{t_2\l x_2,p_2,q_2\}\cdots\{t_m\l x_m,p_m,q_m\}\right),
\end{align*}
as $\{k_1-|a|\l x_1, (p_{1,1}+a_1,\ldots,p_{1,d_1-2}+a_{d_1-2},p_{1,d_1-1}),(q_{1,1})\}\cdot\left(\rule{0mm}{4mm}\{t_2\l x_2,p_2,q_2\}\cdots\{t_m\l x_m,p_m,q_m\}\right)$ vanishes for $a=(a_1,\ldots,a_{d_1-2},0)\geq0$ with $|a|\leq k_1$ by interaction reasons.
\end{proof}

\begin{corollary}\label{nonstrongiszero}
If the factors on the left of~(\ref{OP}) do not yield a strong interaction product, then $\varpi=0$.
\end{corollary}
\begin{proof}
By focusing on the factors $\phi_i$ of $\varpi$ that are involved in a faulty interaction parameter, it suffices to consider three cases:  $L_1(x_0)=\{x_1,\ldots,x_m\}$, $L_{2}(x_1)=\{x_2,\ldots,x_m\}$ and $L_{1}(x_1)=\{x_2,\ldots,x_m\}$. The first two cases are covered by Lemma~\ref{variasconfigs}. On the other hand, there are two options for the instances of the third case that are not covered by Lemma~\ref{variasconfigs}(\ref{T12}): either $p_{1,j}>0$ for some $j\in\{1,2,\ldots,d_1-2\}$ or, else, $p_{1,j}=0$ for all $j\in\{1,2,\ldots,d_1-2\}$ ---in both cases $s_1=1$. In the latter option, the result follows from Lemma~\ref{T11}; in the former option we have $\langle k_1,x_1,p_1,q_1\rangle=\{k_1 \l x_1,p_1,q_1\}$ while the condition $p_{1,d_1-1}+\mathcal{R}_0(x_2,\ldots,x_m)<n$ is forced by the no-strong-interaction hypothesis, so that the result follows by interaction reasons in view of Lemma~\ref{variasconfigs}(\ref{T01}).
\end{proof}

The proof of Theorem~\ref{binarycore} will be complete once we set a one-to-one correspondence between the set of ordered strong interaction products $\varpi$ and the graded basis of $H^*(B_nT;R)$ formed by the elements in~(\ref{criticalrepresentatives}). With this in mind, we start with a two-step approach to the missing case in Lemma~\ref{variasconfigs}(\ref{T12}):

\begin{lemma}\label{T11full}
Assume $L_{1}(x_1)=\{x_2,x_3,\ldots,x_m\}$ with $s_1=1$. Then
\begin{equation}\label{progresivamente}
\varpi=\begin{cases}
\{R_0\l B_1\l\cdots\l B_m\}+\sum\raisebox{.6mm}{$\centerdot$}\,\{\underline{R_0}\l \dt{B}_1\l\cdots\l \dt{B}_m\}+\sum\raisebox{.6mm}{$\centerdot$}\,
\{R_0\l x_1,\underline{P_{1}},Q_{1}\l\dt{B}_2\l\cdots\l \dt{B}_m\}, & \mbox{if $P_{1}>0;$} \\
0, & \mbox{otherwise.}
\end{cases}
\end{equation}
\end{lemma}
\begin{proof}
Interactions occur only in $T_{1,d_1-1}$, so $R_0=k_1$, $Q_i=q_{i}$ for $1\leq i\leq m$, and $P_{i}=p_{i}$ for $2\leq i\leq m$. By Corollary~\ref{nonstrongiszero}, only the case $P_{1}>0$ needs to be argued. Use Lemma~\ref{variasconfigs}(\ref{T01}) to write $\varpi =\phi_1\cdot(\phi_2\cdots\phi_m)$ as
$$
\left(\{R_0\l x_1,p_{1},Q_{1}\} +\sum\{\underline{R_0}\l x_1,\centerdot,Q_{1}\}\right)
\left(\{R'_0\l B_2\l\cdots\l B_m\}+\sum\{\underline{R'_0}\l \dt{B}_2\l\cdots\l\dt{B}_m\}\right),
$$
where $R'_0=\mathcal{R}_0(x_2,\ldots,x_m)$ (so $P_{1,d_1-1}=p_{1,d_1-1}+R'_0-n$).
The result then follows by direct inspection, though this time~(\ref{formareducida}) needs to be used in the analysis of the products giving rise to the terms in both summations of~(\ref{progresivamente}).
\end{proof}

\begin{proposition}\label{T1epsilon}
Assume $L_{1}(x_1)=\{x_2,x_3,\ldots,x_u\}$ and $L_{2}(x_1)=\{x_{u+1},\ldots,x_{m-1},x_m\}$, with $1<u<m$ and $s_1=1$. Then
$$
\varpi=\begin{cases}
\{R_0\l B_1\l\cdots\l B_m\}+\sum\raisebox{.6mm}{$\centerdot$}\,\{\underline{R_0}\l \dt{B}_1\l\cdots\l \dt{B}_m\}+\sum\raisebox{.6mm}{$\centerdot$}\,\{R_0\l x_1,\underline{P_{1},Q_{1}}\l\dt{B}_2\l\cdots\l \dt{B}_m\},&\mbox{if $P_{1}>0\leq Q_{1}$;}\\0,&\mbox{otherwise.}
\end{cases}
$$
Here and below each expression $\underline{P_{1},Q_{1}}$ is meant to represent a pair $V_1,W_1$ of unspecified tuples of integer numbers with $V_1=(V_{1,1},\ldots,V_{1,d_1-1})$, $W_1=(W_{1,1})$ and such that $V_1>0\leq W_1$ and $(V_1,W_1)<(P_{1},Q_{1})$ in the product ordering, i.e., $V_{1,j}\leq P_{1,j}$ for $j=1,2,\ldots, d_1-1$ and $W_{1,1}\leq Q_{1,1}$, with at least one of the last $d_1$ inequalities being strict.
\end{proposition}
\begin{proof}
By Corollary~\ref{nonstrongiszero}, it suffices to consider the case $P_1>0\leq Q_1$. Lemmas~\ref{variasconfigs}(\ref{T01}) and \ref{T11full} allow us to write $\varpi=(\phi_1\cdots\phi_u)\cdot(\phi_{u+1}\cdots\phi_m)$ as the product of 
$$
\{R_0\l x_1,P_{1},q_{1}\l B_2\l\ldots\l B_u\}
+\sum\raisebox{.6mm}{$\centerdot$}\,\{\underline{R_0}\l \dt{B}_1\l \dt{B}_2\l\ldots\l \dt{B}_u\}+\sum\raisebox{.6mm}{$\centerdot$}\,\{R_0\l x_1,\underline{P_{1}},q_{1}\l \dt{B}_2\l\ldots\l \dt{B}_u\}
$$
with
$$
\{R'_0\l B_{u+1}\l\ldots\l B_m\}+\sum\raisebox{.6mm}{$\centerdot$}\,\{\underline{R'_0}\l \dt{B}_{u+1}\l\ldots\l \dt{B}_m\},
$$
where $R'_0=\mathcal{R}_0(x_{u+1},\ldots,x_m)$ (so $Q_{1,1}=q_{1,1}+R'_0-n$). The result follows by inspection.
\end{proof}

We are now ready to set up the strategy for completing the proof of Theorem~\ref{binarycore}. By Lemma~\ref{systemofequations}, Remark~\ref{configuraciones} and Corollary~\ref{nonstrongiszero}, the goal reduces to describing, for fixed essential vertices $x_1<\cdots<x_m$, a~partial ordering $\preceq$ on the set of basis elements $\{t_0\l x_1,u_{1},v_{1}\l\cdots\l x_m,u_{m},v_{m}\}$ of $H^m(\UDnT)$ such that any strong interaction product (\ref{OP}) can be expressed by a congruence
\begin{equation}\label{congruence}
\aop\equiv\{R_0\l B_1\l\cdots\l B_m\}
\end{equation}
modulo basis elements that are $\preceq$-smaller than $\{R_0\l B_1\l\cdots\l B_m\}$. The partial ordering $\preceq$ we need becomes apparent by writing either of the triples $(x_1,\underline{P_1,Q_1})$, $(x_1,\underline{P_1},Q_1)$ and $(x_1,P_1,\underline{Q_1})$ in Proposition \ref{T1epsilon} and Lemmas~\ref{T11full} and~\ref{variasconfigs}(\ref{T12}), respectively, as $\underline{B_1}$. Indeed, in such terms, the ($P_1>0\leq Q_1$)-conclusions in those results can be written as 
\begin{equation}\label{apparent}
\varpi=\{R_0\l B_1\l\cdots\l B_m\}+\sum\raisebox{.6mm}{$\centerdot$}\,\{\underline{R_0}\l \dt{B}_1\l\cdots\l \dt{B}_m\}+\sum\raisebox{.6mm}{$\centerdot$}\,\{R_0\l \underline{B_1}\l\dt{B}_2\l\cdots\l \dt{B}_m\}.
\end{equation}

\begin{definition}\label{levels}
The $\ell$-th level of pruned leaves $\mathcal{L}_\ell$ of the essential vertices $x_1<\cdots<x_m$ is 
$$
\mathcal{L}_\ell=\mathcal{L}_\ell(x_1,\ldots,x_m):=
\begin{cases}
\;L_1(x_0), & \mbox{if $\ell=1$;}\\
\;\bigcup_{x_i\in\mathcal{L}_{\ell-1}}\left(\rule{0mm}{4mm}L_1(x_i)\cup L_2(x_i)\right), & \mbox{if $\ell>1$.}
\end{cases}
$$
The interaction level of the vertices $x_1<\cdots<x_m$ is the largest $\ell$ such that $\mathcal{L}_\ell\neq\varnothing$. Furthermore, extending the notation introduced in~(\ref{prelevel}) and~(\ref{apparent}), let $B^{(\ell)}$ denote the collection of blocks $B_i$ with $x_i\in\mathcal{L}_\ell$, and let $\dt{B}^{(\ell)}$ stand for any collection of blocks $\dt{B}_i$ with $x_i\in\mathcal{L}_\ell$. On the other hand, $\underline{B^{(\ell)}}$ stands for any collection of blocks $(x_i,V_{i},W_{i})$, with $x_i\in\mathcal{L}_\ell$, satisfying:
\begin{itemize}
\item $V_i>0\leq W_i$ and $(V_i,W_i)\leq(P_{i},Q_{i})$ (the latter in the product ordering) for all $x_i\in\mathcal{L}_\ell$, and
\item $(V_i,W_i)\neq(P_{i},Q_{i})$ for at least one $x_i\in\mathcal{L}_\ell$.
\end{itemize}
\end{definition}

Note that the definition of $\underline{B^{(\ell)}}$ is less restrictive than actually requiring $\underline{B^{(\ell)}}$ to be a collection of blocks $\underline{B_i}$ with $x_i\in\mathcal{L}_\ell$. As in Proposition \ref{T1epsilon}, the condition we want for $\underline{B^{(\ell)}}$ is based on a \emph{strict} product-order inequality. The reason for this becomes apparent in the proof of Proposition \ref{induccinfinal} below.

\begin{example}\label{baseinductiva}{\em
 Lemma~\ref{variasconfigs}(\ref{T01}) gives $\varpi=\{R_0,B^{(1)}\}+\sum\raisebox{.6mm}{$\centerdot$}\,\{\underline{R}_0,\dt{B}^{(1)}\}$ in interaction level 1 (under a strong condition hypothesis). Likewise, (\ref{apparent}) becomes
\begin{equation}\label{ell2L1} 
\varpi=\{R_0\l B^{(1)}\l B^{(2)}\}+\sum\raisebox{.6mm}{$\centerdot$}\,\{\underline{R_0}\l \dt{B}^{(1)}\l \dt{B}^{(2)}\}+\sum\raisebox{.6mm}{$\centerdot$}\,\{R_0\l \underline{B^{(1)}}\l \dt{B}^{(2)}\}
\end{equation}
in interaction level 2 (with $\mathcal{L}_1=\{x_1\}$, so $B^{(1)}$ consist of $B_1$ alone). In full generality:
}\end{example}

\begin{proposition}\label{induccinfinal}
Let $x_1<\cdots<x_m$ be essential vertices having interaction level $\ell$. If $\varpi$ is a strong interaction product, then
\begin{equation}\label{etiquetafinal}
\begin{split}
\varpi\,=&\,\,\{R_0\l B^{(1)}\l\cdots\l B^{(\ell)}\}+\sum\raisebox{.6mm}{$\centerdot$}\,\{\underline{R_0}\l\dt{B}^{(1)}\l\cdots\l\dt{B}^{(\ell)}\}\\
&+\sum\raisebox{.6mm}{$\centerdot$}\,\{R_0\l\underline{B^{(1)}}\l\dt{B}^{(2)}\l\cdots\l\dt{B}^{(\ell)}\}+\cdots+
\sum\raisebox{.6mm}{$\centerdot$}\,\{R_0\l B^{(1)}\l\cdots\l B^{(\ell-2)}\l \underline{B^{(\ell-1)}}\l\dt{B}^{(\ell)}\}.
\end{split}
\end{equation}
\end{proposition}

\begin{proof}[Proof of Theorem~\ref{binarycore} (conclusion)]
Partially order the set of basis elements $\{v_0\l x_1,v_{1},w_{1}\l\cdots\l x_m,v_{m},w_{m}\}$ by means of a level-wise lexicographical comparison of their $v$- and $w$-ingredients. Then~(\ref{etiquetafinal}) yields the required congruence~(\ref{congruence}).
\end{proof}

\begin{proof}[Proof of Proposition~\ref{induccinfinal}]
The argument is by direct computation, proceeding by induction on $\ell$ and with Example~\ref{baseinductiva} grounding the induction. The real challenge consists on setting a suitable notation so arguments can be seen clearly. With this in mind, we start by checking the situation in the special case $\mathcal{L}_1=\{x_1\}$ (so $R_0=k_1$), i.e., the generalization of~(\ref{ell2L1}) to higher interaction levels. In such a situation
\begin{equation}\label{seborrara}
\mathcal{L}_\lambda(x_2,\ldots,x_m)=\mathcal{L}_{\lambda+1}(x_1,\ldots,x_m),\mbox{ for $\lambda\geq2$.}
\end{equation}
Accordingly, we reset notation and start level-number counting at 2 (rather than at 1) for $x_2<\cdots<x_m$, so to make it compatible with that for $x_1<\cdots<x_m$. Thus,~(\ref{seborrara}) gets replaced by
\begin{equation}\label{sustituyo}
\mathcal{L}_\lambda(x_2,\ldots,x_m)=\mathcal{L}_\lambda(x_1,\ldots,x_m),\mbox{ for $\lambda\geq3$.}
\end{equation}
Let $x_2,x_3,\ldots,x_t$ be the essential vertices lying on the component of $T\setminus\{x_1\}$ in $x_1$-direction $d_1-1$, while $x_{t+1},x_{t+2},\ldots,x_m$ be the vertices lying on the component of $T\setminus\{x_1\}$ in $x_1$-direction $d_1$ ($1\leq t\leq m$). Then, if $B^{(\lambda)}$, $\dt{B}^{(\lambda)}$ and $\underline{B^{(\lambda)}}$ stand for collections defined by all the vertices $x_1,\ldots,x_m$, we write
\begin{equation}\label{porpartes}
B^{(\lambda)}_{[\epsilon]}, \;\; \dt{B}^{(\lambda)}_{[\epsilon]} \;\;\mbox{or}\;\; \underline{B^{(\lambda)}_{[\epsilon]}},
\end{equation}
with $\varepsilon=1$, to denote the corresponding parts coming only from the vertices $x_2,\ldots,x_t$. Likewise, the case $\epsilon=2$ in~(\ref{porpartes}) stands for the parts that come from the vertices $x_{t+1},\ldots,x_m$. For instance, $B^{(\lambda)}=B^{(\lambda)}_{[1]}\cup B^{(\lambda)}_{[2]}$. In these terms, we use induction to write $\varpi=\phi_1\cdot(\phi_2\cdots\phi_t)\cdot(\phi_{t+1}\cdots\phi_m)$ as the product of
the three expressions
$$\{R_0\l x_1,p_{1},q_{1}\}+\sum\{\underline{R_0}\l x_1,\centerdot,q_{1}\},$$ 
$$
\left\{R'_0\ll B_{[1]}^{(2)}\ll\cdots\ll B_{[1]}^{(\ell)}\right\}+
\sum\raisebox{.6mm}{$\centerdot$}\,\left\{\underline{R'_0}\ll\dt{B}_{[1]}^{(2)}\ll\cdots\ll\dt{B}_{[1]}^{(\ell)}\right\}+\sum_{3\leq j\leq \ell}\raisebox{.6mm}{$\centerdot$}\,\left\{R'_0\ll B_{[1]}^{(2)}\ll\cdots\ll B_{[1]}^{(j-2)}\ll \underline{B_{[1]}^{(j-1)}}\ll\dt{B}_{[1]}^{(j)}\ll\cdots\ll \dt{B}_{[1]}^{(\ell)}\right\}\!,
$$
and
$$
\left\{R''_0\ll B_{[2]}^{(2)}\ll\cdots\ll B_{[2]}^{(\ell)}\right\}+
\sum\raisebox{.6mm}{$\centerdot$}\,\left\{\underline{R''_0}\ll\dt{B}_{[2]}^{(2)}\ll\cdots\ll\dt{B}_{[2]}^{(\ell)}\right\}
+\sum_{3\leq j\leq \ell}\raisebox{.6mm}{$\centerdot$}\,\left\{R''_0\ll B_{[2]}^{(2)}\ll\cdots\ll B_{[2]}^{(j-2)}\ll \underline{B_{[2]}^{(j-1)}}\ll\dt{B}_{[2]}^{(j)}\ll\cdots\ll \dt{B}_{[2]}^{(\ell)}\right\}\!,
$$
where $R'_0=\mathcal{R}_0(x_2,\ldots,x_t)$ and $R''_0=\mathcal{R}_0(x_{t+1},\ldots,x_m)$. Note the compactified notation for the two summations running over $j$, each of which really stands for sums of summations as in~(\ref{etiquetafinal}). Note also that the interaction level of the vertices $x_2,\ldots,x_t$ (or $x_{t+1},\ldots,x_m$) could be smaller than $\ell$, in which case some of the corresponding collections of blocks are empty. Then, by direct inspection and interaction reasons (using~(\ref{formareducida}) when $s_1=1$ and the interaction parameter under consideration lies in $x_1$-direction $d_1-1$), the product of the three expressions above takes the form~(\ref{etiquetafinal}). This completes the proof when $\mathcal{L}_1$ is a singleton.

In general, $\mathcal{L}_1$ consists of, say, vertices $x_1=x_{i_1}<\cdots<x_{i_k}$, and we evaluate $\varpi$ as the length-$k$ product
\begin{equation}\label{lengthk}
(\phi_1\cdots\phi_{i_2-1})(\phi_{i_2}\cdots\phi_{i_3-1})\cdots(\phi_{i_k}\cdots\phi_m).
\end{equation}
(This time there is no need to reset notation so to get the analogue of~(\ref{sustituyo}) to hold.) We have just seen that the $w$-th factor in~(\ref{lengthk}) takes the form
$$
\left\{r_{i_w}\ll B_{[w]}^{(1)}\ll{\cdots}\ll B_{[w]}^{(\ell)}\right\}+
\sum\raisebox{.6mm}{$\centerdot$}\,\left\{\underline{r_w}\ll\dt{B}_{[w]}^{(1)}\ll{\cdots}\ll\dt{B}_{[w]}^{(\ell)}\right\}+
\sum_{2\leq j\leq \ell}\raisebox{.6mm}{$\centerdot$}\,\left\{r_{i_w}\ll B_{[w]}^{(1)}\ll{\cdots}\ll B_{[w]}^{(j-2)}\ll \underline{B_{[w]}^{(j-1)}}\ll\dt{B}_{[w]}^{(j)}\ll{\cdots}\ll \dt{B}_{[w]}^{(\ell)}\right\}\!,
$$
where
$$
B_{[w]}^{(1)}:=B_{i_w}, \ \ \dt{B}_{[w]}^{(1)}:=\dt{B}_{i_w}, \ \ \underline{B_{[w]}^{(1)}}:=\underline{B_{i_w}}
$$
and, for interaction levels larger than 1, a subindex `[$w$]' in a collection of blocks indicates that only blocks in positive $x_{i_w}$-directions are to be taken. The required form (\ref{etiquetafinal}) for the product of all these expressions follows again from direct inspection ---this time without requiring the use of~(\ref{formareducida}).
\end{proof}


\begin{thebibliography}{10}

\bibitem{MR2701024}
Aaron~David Abrams.
\newblock {\em Configuration spaces and braid groups of graphs}.
\newblock ProQuest LLC, Ann Arbor, MI, 2000.
\newblock Thesis (Ph.D.)--University of California, Berkeley.

\bibitem{jorgetereyo}
Jorge Aguilar-Guzm\'{a}n, Jes\'{u}s Gonz\'{a}lez, and Teresa Hoekstra-Mendoza.
\newblock Farley-{S}abalka's {M}orse-theory model and the higher topological
  complexity of ordered configuration spaces on trees.
\newblock {\em Discrete Comput. Geom.}, 67(1):258--286, 2022.

\bibitem{MR3253777}
Francis Connolly and Margaret Doig.
\newblock On braid groups and right-angled {A}rtin groups.
\newblock {\em Geom. Dedicata}, 172:179--190, 2014.

\bibitem{MR2216709}
Daniel Farley.
\newblock Homology of tree braid groups.
\newblock In {\em Topological and asymptotic aspects of group theory}, volume
  394 of {\em Contemp. Math.}, pages 101--112. Amer. Math. Soc., Providence,
  RI, 2006.

\bibitem{MR2359035}
Daniel Farley.
\newblock Presentations for the cohomology rings of tree braid groups.
\newblock In {\em Topology and robotics}, volume 438 of {\em Contemp. Math.},
  pages 145--172. Amer. Math. Soc., Providence, RI, 2007.

\bibitem{MR2171804}
Daniel Farley and Lucas Sabalka.
\newblock Discrete {M}orse theory and graph braid groups.
\newblock {\em Algebr. Geom. Topol.}, 5:1075--1109, 2005.

\bibitem{MR2355034}
Daniel Farley and Lucas Sabalka.
\newblock On the cohomology rings of tree braid groups.
\newblock {\em J. Pure Appl. Algebra}, 212(1):53--71, 2008.

\bibitem{MR1358614}
Robin Forman.
\newblock A discrete {M}orse theory for cell complexes.
\newblock In {\em Geometry, topology, \& physics}, Conf. Proc. Lecture Notes
  Geom. Topology, IV, pages 112--125. Int. Press, Cambridge, MA, 1995.

\bibitem{MR1926850}
Robin Forman.
\newblock Discrete {M}orse theory and the cohomology ring.
\newblock {\em Trans. Amer. Math. Soc.}, 354(12):5063--5085, 2002.

\bibitem{MR1873106}
Robert Ghrist.
\newblock Configuration spaces and braid groups on graphs in robotics.
\newblock In {\em Knots, braids, and mapping class groups---papers dedicated to
  {J}oan {S}. {B}irman ({N}ew {Y}ork, 1998)}, volume~24 of {\em AMS/IP Stud.
  Adv. Math.}, pages 29--40. Amer. Math. Soc., Providence, RI, 2001.

\bibitem{MR2028588}
Tomasz Kaczynski, Konstantin Mischaikow, and Marian Mrozek.
\newblock {\em Computational homology}, volume 157 of {\em Applied Mathematical
  Sciences}.
\newblock Springer-Verlag, New York, 2004.

\bibitem{MR3105945}
Tomasz Kaczynski and Marian Mrozek.
\newblock The cubical cohomology ring: an algorithmic approach.
\newblock {\em Found. Comput. Math.}, 13(5):789--818, 2013.

\bibitem{MR2833585}
Jee~Hyoun Kim, Ki~Hyoung Ko, and Hyo~Won Park.
\newblock Graph braid groups and right-angled {A}rtin groups.
\newblock {\em Trans. Amer. Math. Soc.}, 364(1):309--360, 2012.

\bibitem{MR3426912}
Ki~Hyoung Ko, Joon~Hyun La, and Hyo~Won Park.
\newblock Graph 4-braid groups and {M}assey products.
\newblock {\em Topology Appl.}, 197:133--153, 2016.

\bibitem{MR2516176}
Lucas Sabalka.
\newblock On rigidity and the isomorphism problem for tree braid groups.
\newblock {\em Groups Geom. Dyn.}, 3(3):469--523, 2009.

\bibitem{MR3773741}
Steven Scheirer.
\newblock Topological complexity of {$n$} points on a tree.
\newblock {\em Algebr. Geom. Topol.}, 18(2):839--876, 2018.

\end{thebibliography}

\bigskip

{\sc \ 

Departamento de Matem\'aticas

Centro de Investigaci\'on y de Estudios Avanzados del I.P.N.

Av.~Instituto Polit\'ecnico Nacional n\'umero~2508, San Pedro Zacatenco

M\'exico City 07000, M\'exico.}

\tt jesus@math.cinvestav.mx

idskjen@math.cinvestav.mx

\end{document}